\documentclass[11pt,oneside]{amsart}
\usepackage{amsmath}
\usepackage{amsfonts}
\usepackage{amssymb}
\usepackage{amsthm}
\usepackage[colorlinks]{hyperref} 
\usepackage{wasysym}
\usepackage{subcaption}
\usepackage{graphicx}
\usepackage{braket}
\usepackage{dashrule}
\usepackage[all]{xy}
\usepackage{multirow}
\usepackage{enumitem}
\usepackage{tikz}
\usetikzlibrary{arrows,decorations.pathmorphing,backgrounds,positioning,fit,calc}
\usepackage{tikz-cd}
\usepackage{framed}

\usepackage{latexsym}
\usepackage{stix}

\normalfont\upshape

\theoremstyle{plain}

\numberwithin{equation}{section}

\theoremstyle{remark}

\DeclareMathOperator{\spn}{Span}

\newcommand{\nocontentsline}[3]{}
\newcommand{\tocless}[2]{\bgroup\let\addcontentsline=\nocontentsline#1{#2}\egroup}

\newcommand{\NN}{\mathbb{N}}

\newcommand{\nc}{\newcommand}
\nc{\bla}{\phantom{bbbbb}}

\newcommand{\beq}{\begin{equation}}
\newcommand{\eeq}{\end{equation}}
\newcommand{\barr}{\begin{array}}
\newcommand{\earr}{\end{array}}
\newcommand{\beqar}{\begin{eqnarray}}
\newcommand{\eeqar}{\end{eqnarray}}
\newtheorem{theorem}{Theorem}[section]
\newtheorem{corollary}[theorem]{Corollary}
\newtheorem{lemma}[theorem]{Lemma}
\newtheorem{prop}[theorem]{Proposition}
\newtheorem{defn}[theorem]{Definition}
\newtheorem{remark}[theorem]{Remark}

\newtheorem{exit}[theorem]{Example}

\newtheorem{thm}[theorem]{Theorem}

\newenvironment{rem}{\begin{remark}\rm}{\end{remark}}

\newcommand{\RR}{{\mathbb R }}
\nc{\FF}{ {\mathbb F} }
\nc{\HH}{ {\mathbb H} }
\newcommand{\ZZ}{{\mathbb Z }}

\newcommand{\grad}{{\text{grad}}}

\nc{\umax}{{U_{\max}}}

\newcommand{\liek}{{\mathfrak k}}

\newcommand{\lieks}{{\liek}^*}

\nc{\lieq}{{\mathfrak q}}
\nc{\liez}{{\mathfrak z}}
\nc{\lieqs}{{\lieq}^*}
\nc{\lieg}{{\mathfrak g}}
\nc{\liegs}{{\lieg}^*}
\nc{\liep}{{\mathfrak p}}
\nc{\lieps}{{\liep}^*}

\newcommand{\codim}{{\text{codim}}}

\title{Morse theory without nondegeneracy}

\author{Frances Kirwan and Geoffrey Penington}

\address[Kirwan]{Mathematical Institute\\ Oxford University\\ Woodstock Road, Oxford OX2 6GG\\  UK} 
\email{kirwan@maths.ox.ac.uk}

\address[Penington]{Department of Physics\\
University of California\\ 
Berkeley, CA 94720-7300, USA}
\email{geoffp@berkeley.edu}

\begin{document}

\begin{abstract}
We describe an extension of Morse theory to smooth functions on compact Riemannian manifolds, without any nondegeneracy assumptions except that the critical locus must have only finitely many connected components.
\end{abstract}

\maketitle


Let $M$ be a compact manifold (without boundary). Classical Morse theory \cite{Milnor} studies the topology of $M$ using a smooth real-valued function $f \colon M\to \RR$ on $M$ with nondegenerate (and therefore isolated)  critical points; the Morse inequalities relate the Betti numbers of $M$ to the numbers of critical points of $f$ with given Morse index. It is well known that it is possible to relax this nondegeneracy assumption, for example to allow the connected components of the critical locus to be submanifolds with nondegeneracy in the normal directions, so that $f$ is a Morse--Bott function  \cite{Bott1,Bott}. Our aim here is to remove the hypothesis of nondegeneracy, except for the much weaker requirement that the critical locus of $f$ has only finitely many connected components (or even that $f$ has only finitely many critical values: see Remark \ref{rem:OS} below). 

Morse theory has been generalised in many different ways in the decades since its introduction (cf.\! for example \cite{BH1, Farber, Jost, Nicolaescu, Shubin,  Wehrheim, Witten82}), and there are several different approaches used to obtain Morse inequalities for suitable smooth functions $f \colon M \to \RR$.  In the main the different approaches involve the choice of a Riemannian metric $g$ on $M$, and the associated gradient flow $\grad(f) $ of $f$, or at least some sort of pseudo-gradient vector field (and from this viewpoint Conley's index theory \cite{Conley, Salamon} is more general still, studying smooth flows which are not necessarily gradient flows).

(i) {\bf Attaching handles}. The original approach of Morse \cite{Milnor} for $f \colon M\to \RR$ with nondegenerate critical points was to study the topology of $f^{-1}(-\infty, a]$ as $a$ varies in $\RR$, and to show that this is unchanged as $a$ increases except when $a$ passes through a critical value, when a handle is attached for each critical point $p$ with $f(p)=a$.

(ii) {\bf Morse stratifications}. An approach which has been used to extend Morse theory to suitable smooth functions with non-isolated critical points \cite{Bott, Bottindom, K}  is to stratify $M$ according to the limiting behaviour of the gradient flow of $f$. Thus we get $M = \bigsqcup_{0\leqslant j \leqslant k} S_j$, where each $S_j$ is a locally closed submanifold of $M$ which retracts onto its intersection $C_j$ with the critical set $\text{Crit}(f)$ for $f$, and where $U_j =  \bigsqcup_{0\leqslant i \leqslant j} S_i$ is open in $M$.
For any field $\FF$, and under suitable orientability assumptions (which can be ignored if we use $\ZZ/2\ZZ$ coefficients), the Thom--Gysin sequences 
$$ \cdots \to H_{i +1 - \codim S_j }(S_j;\FF) \to
 H_i(U_{j-1};\FF)  \to  H_i(U_j;\FF) \to H_{i - \codim S_j }(S_j;\FF) \to  \cdots
$$
associated to the inclusions of $U_{j-1} = U_j \setminus S_j $ into $U_j$  (for $1 \leqslant j \leqslant k$),  combined with  homeomorphisms from neighbourhoods of $S_j$ in $U_j$ to neighbourhoods of the zero section in the normal bundles to the strata $S_j$, can be used to derive  Morse inequalities of the form
\begin{equation} \label{Morseineq} P_t(M) = \sum_{j=0}^k t^{\codim S_j} P_t(C_j)  \,\, - \,\, (1+t)R(t)
\end{equation}
where $P_t(M) = \sum_{i \geqslant 0} t^i \dim_\FF H_i(M;\FF)$ is the Poincar\'e polynomial of $M$ and $R(t)$ is a polynomial with nonnegative coefficients. Here $\codim S_j = \text{ind}_f(C_j)$ is the Morse index of $f$ along $C_j$.

(iii) {\bf Morse homology}. An approach which goes back to Milnor, Thom and Smale and which was reinvigorated by Witten \cite{Bismut,Smale,Smale2,MSchwarz,Witten82} is to use a Morse function $f \colon M \to \RR$ and a suitable Riemannian metric $g$ on $M$ to derive the Morse inequalities by defining a complex 
$$ \cdots \xrightarrow{\partial} C_{i+1}^{\text{Morse}}(f,g) \xrightarrow{\partial} C_{i}^{\text{Morse}}(f,g) \xrightarrow{\partial} C_{i-1}^{\text{Morse}}(f,g) \xrightarrow{\partial} \cdots
$$
in terms of the critical set $\text{Crit}(f)$ and showing that its homology is isomorphic to the homology of $M$. 
 Here Hodge theory can be used to describe the (de Rham) cohomology  $H^j(M;\RR) = (H_j(M;\RR))^*$ in terms of harmonic forms on $M$. Motivated by ideas from supersymmetry, Witten used the smooth function $f \colon M \to \RR$ to defined a modification $\Delta_t$ (depending on $t \in \RR$) of the Laplacian $\Delta$, with $\Delta_0 = \Delta$, and studied the spectrum of $\Delta_t$ as $t \to \infty$. Solutions to $\Delta_t = 0$ as $t \to \infty$ localise where $\text{d}\!f$ is small, i.e. near $\text{Crit}(f)$, and tunnelling effects lead to the definition of the Morse--Witten complex.

These approaches to the Morse inequalities 
have been extended in different ways from classical Morse (and Morse--Smale) functions to Morse--Bott and minimally degenerate functions \cite{AustinB, BH, Bott, Bottindom, Frau, K,ZZ}, to Novikov inequalities for closed 1-forms \cite{ BF,BF2,BF3,BF4,Novikov, Novikov2, Pazhitnov}, to allow $M$ to have boundary \cite{Akaho,BNR,ChangLiu, KM, Laudenbach} and in suitable circumstances  to be non-compact or infinite-dimensional \cite{AM, AD, BF,BF2, BF3, BF4, BS, Dask, EG, FS, Floer1, Floer2,Fukaya,Salamon}; some approaches using stratifications do not require $M$ to be a manifold \cite{GM,Quinn,Wilkin,Woolf}, and discrete versions of Morse theory (cf. \cite{Forman,KK,Nanda, Nanda2}) have also been studied.
For the main results in this article we will assume that $M$ is a (finite-dimensional) compact  Riemannian manifold without boundary, though we will briefly consider other situations. However our only hypothesis  on $f \colon M \to \RR$ will be that it is smooth and 
 that the critical locus $\text{Crit}(f)$ has only finitely many connected components; then its set of critical values $\text{Critval}(f)= f(\text{Crit}(f))$ is a finite union of compact connected subsets of $\RR$ and by Sard's theorem has measure zero, so is finite. 

In order to obtain such a generalisation of the classical Morse inequalities, we need more ingredients than those in (\ref{Morseineq}). We will use a \lq system of Morse neighbourhoods' for $f$ in the sense described below (or slightly more generally as given in Definition \ref{Mn}). For any smooth $f \colon M \to \RR$ it is possible to choose a system of Morse neighbourhoods (see Proposition \ref{Mnexist} below), and the resulting inequalities will be independent of this choice.

\begin{figure}[t]
\begin{subfigure}{0.45\textwidth}
\includegraphics[width = 0.8\linewidth]{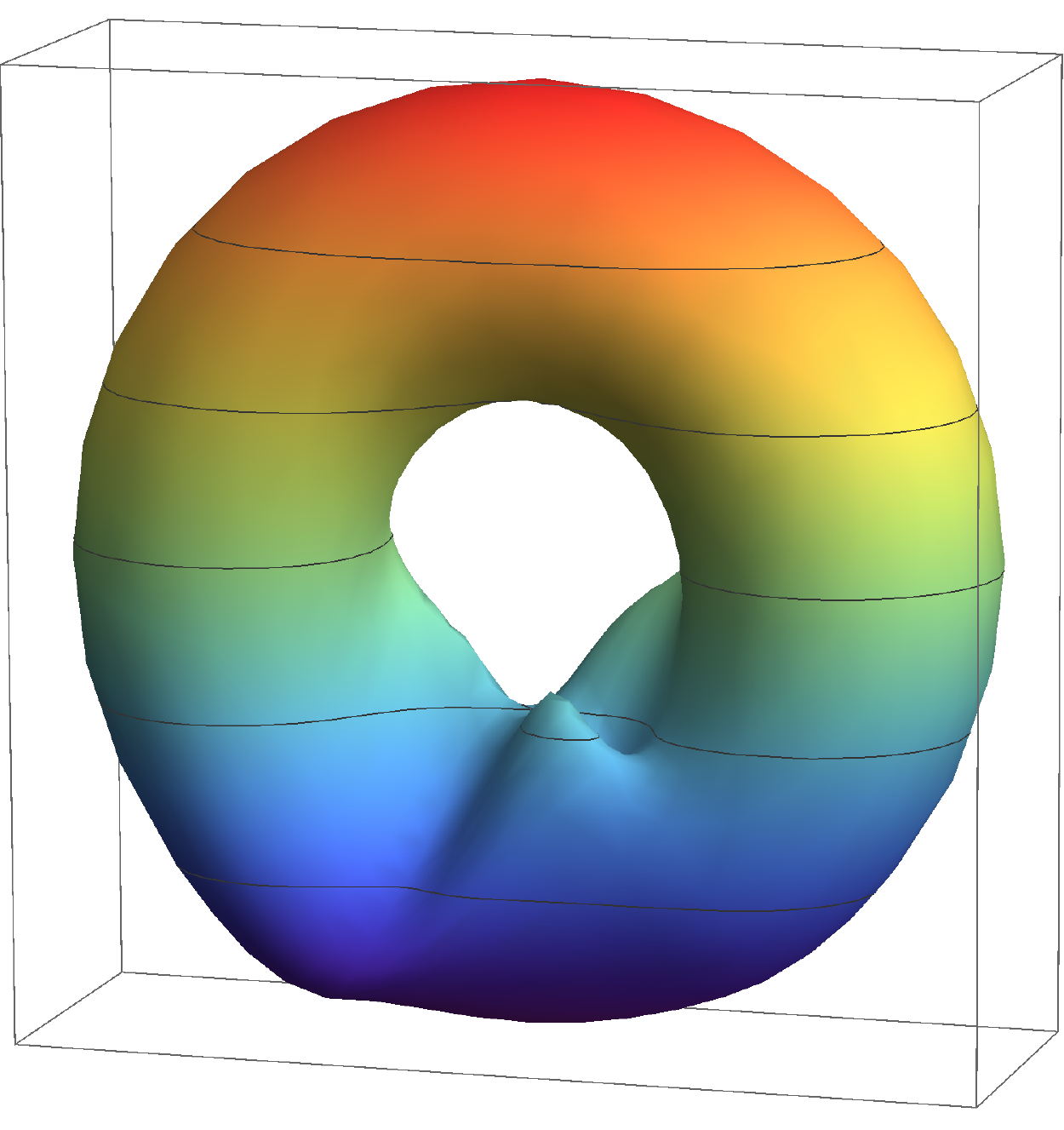}
\centering
\end{subfigure}
\begin{subfigure}{0.45\textwidth}
\includegraphics[width = 0.8\linewidth]{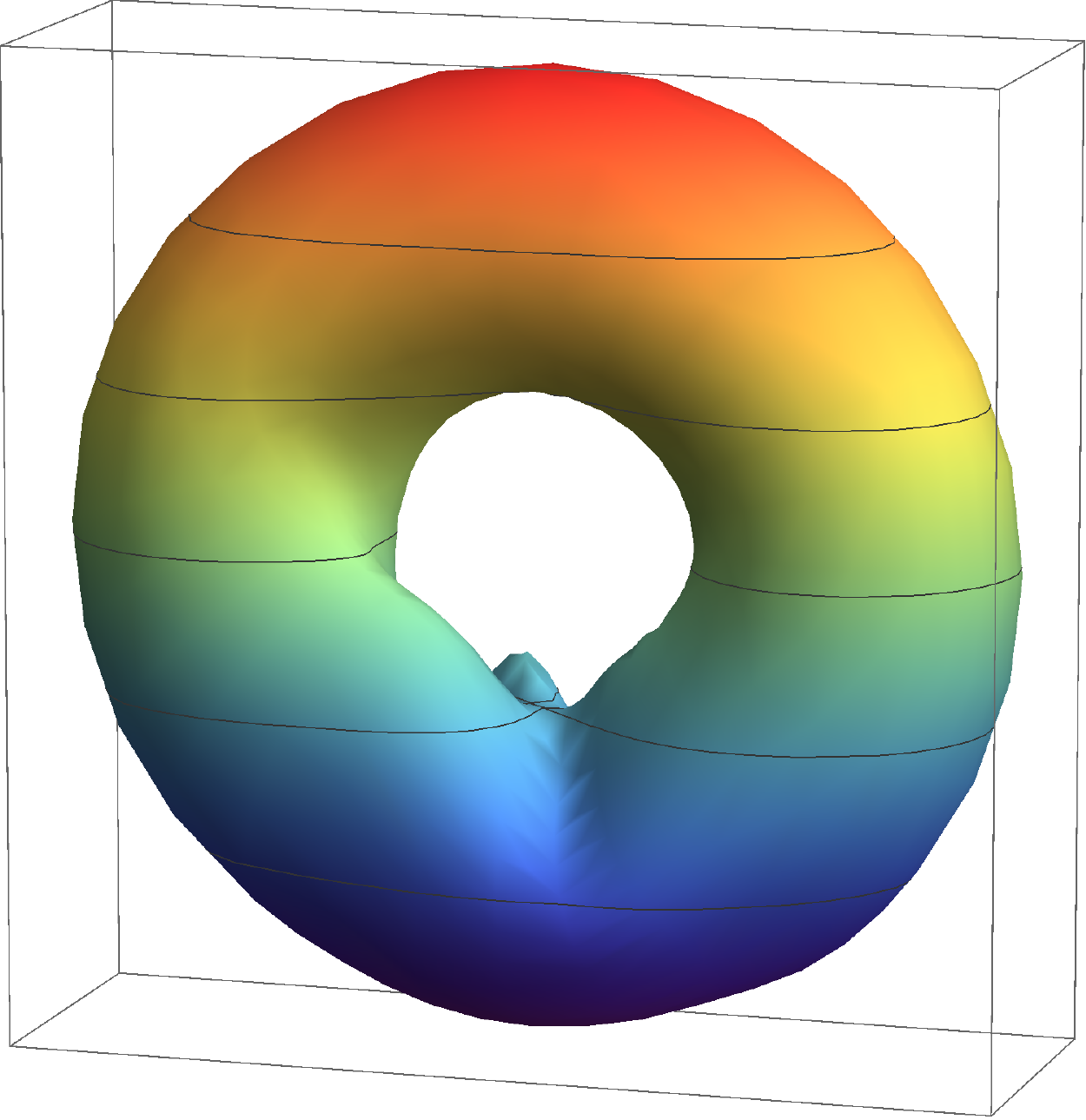}
\centering
\end{subfigure}
\caption{The height function on this torus in $\RR^3$ has two local maxima, one local minimum, one saddle point and a degenerate critical point given locally by $f(x,y) = xy(x-y)$. The right-hand figure is the same torus, viewed from the reverse side.}
\label{fig:funnytorus}
\end{figure}

So let $f \colon M \to \RR$ be a smooth function whose critical locus $\text{Crit}(f)$ has finitely many connected components.  For $c \in \text{Critval}(f)$ let $\text{Crit}_c(f)= f^{-1}(c) \cap \text{Crit}(f)$ and let $\mathcal{D}_c$ be the set of its connected components. 
Then $\mathcal{D} = \bigcup_{c \in \text{Critval}(f)} \mathcal{D}_c$ is the set of connected components of the critical set $\text{Crit}(f)$.
A system of (strict) Morse neighbourhoods for $f$ is given by
$$ \{ \mathcal{N}_{C,n} : C \in \mathcal{D} , \,\,\, n \geqslant 0 \} $$
such that 
if $C \in \mathcal{D}_c$ and $ n \geqslant 0$ then 

(a) $\mathcal{N}_{C,n}$ is a neighbourhood of $C$ in $M$ containing no other critical points for $f$, with $\mathcal{N}_{C,n+1} \subseteq (\mathcal{N}_{C,n})^{\circ}$ (where $(\mathcal{N}_{C,n})^{\circ} = \mathcal{N}_{C,n} \setminus  \partial \mathcal{N}_{C,n}$ is the interior of $\mathcal{N}_{C,n}$) and $\bigcap_{m \geqslant 0} \mathcal{N}_{C,m} = C$;

(b) $\mathcal{N}_{C,n}$ is a compact submanifold of $M$ with corners (locally modelled on $[0,\infty) ^2 \times \RR^{\dim M -2}$) and has boundary
$$\partial \mathcal{N}_{C,n} = \partial_+ \mathcal{N}_{C,n} \cup \partial_- \mathcal{N}_{C,n}$$
where $\partial_\pm \mathcal{N}_{C,n}$ is a compact submanifold of $M$ with boundary and
$$ \partial(\partial_+\mathcal{N}_{C,n}) \, = \, \partial(\partial_- \mathcal{N}_{C,n}) \, = \, (\partial_+ \mathcal{N}_{C,n}) \cap  (\partial_- \mathcal{N}_{C,n}) \,\,  \subseteq \, f^{-1}(c);$$

(c)  the gradient vector field $\grad(f)$ on $M$ associated to its Riemannian metric $g$  satisfies\\
(i) $(\partial_+ \mathcal{N}_{C,n})^\circ = \partial _+ \mathcal{N}_{C,n} \setminus  \partial(\partial_+ \mathcal{N}_{C,n})  \subseteq f^{-1}(c,\infty)$ with the restriction of $\grad (f)$ to $(\partial_+ \mathcal{N}_{C,n})^\circ$ pointing inside $\mathcal{N}_{C,n}$, and  \\
(ii) $(\partial_- \mathcal{N}_{C,n})^\circ = \partial _- \mathcal{N}_{C,n} \setminus  \partial(\partial_-\mathcal{N}_{C,n})  \subseteq f^{-1}(-\infty,c)$ with the restriction of $\grad (f)$ to $(\partial_- \mathcal{N}_{C,n})^\circ$ pointing outside $\mathcal{N}_{C,n}$.

An example of such Morse neighbourhoods is shown in Figure \ref{fig:strict} in $\S$\ref{sec:neighbourhoods}. Let $P_t(\mathcal{N}_{C,n},\partial_\pm \mathcal{N}_{C,n})$ denote the relative Poincar\'e polynomial 
 $ \sum_{i \geqslant 0} t^i \dim_\FF H_i(\mathcal{N}_{C,n},\partial_\pm \mathcal{N}_{C,n};\FF)$. 
We will see that the gradient flow $\grad(f)$, combined with excision, induces isomorphisms of relative homology
$$\phi^{m,n}_i : H_i(\mathcal{N}_{C,n},\partial_\pm \mathcal{N}_{C,n};\FF) \to H_i(\mathcal{N}_{C,m},\partial_\pm \mathcal{N}_{C,m};\FF)$$
for $m>n$.  

\begin{rem}
Each $C \in \mathcal{D}$ is an isolated invariant set in the sense of Conley \cite{Conley} for the gradient flow. Each neighbourhood $\mathcal{N}_{C,n}$ is an isolating neighbourhood of $C$ in Conley's sense and the pair $(\mathcal{N}_{C,n}, \partial_-\mathcal{N}_{C,n})$ is an index pair. The homotopy type of this pair represents the Conley index of $C$. It is shown in \cite{Conley} that this index, and hence its homology $H_*(\mathcal{N}_{C,n},\partial_- \mathcal{N}_{C,n};\FF)$, is independent of the choice of isolating neighbourhood.
\end{rem}

By taking the limit as $m \to \infty$ we can define vector spaces 
\begin{equation} \label{Hlimit} H_i(\mathcal{N}_{C,\infty},\partial_\pm \mathcal{N}_{C,\infty};\FF) \end{equation} with isomorphisms 
$\phi^{\infty,n}_i : H_i(\mathcal{N}_{C,n},\partial_\pm \mathcal{N}_{C,n};\FF) \to H_i(\mathcal{N}_{C,\infty},\partial_\pm \mathcal{N}_{C,\infty};\FF)$
compatible with $\phi^{m,n}_i$ for all $m,n,i$.
Let $$P_t(\mathcal{N}_{C,\infty},\partial_\pm \mathcal{N}_{C,\infty}) =  \sum_{i \geqslant 0} t^i \dim_\FF H_i(\mathcal{N}_{C,\infty},\partial_\pm \mathcal{N}_{C,\infty};\FF)
=  \sum_{i \geqslant 0} t^i \dim_\FF H_i(\mathcal{N}_{C,n},\partial_\pm \mathcal{N}_{C,n};\FF)
$$ 
for any $n \geqslant 0$. 
Our main theorem, Theorem \ref{mainthm}, combined with Proposition \ref{Mnexist}, tells us that for each $C \in \mathcal{D}$ the vector space 
$H_i(\mathcal{N}_{C,\infty},\partial_\pm \mathcal{N}_{C,\infty};\RR)$, up to canonical isomorphism (and hence also the Poincar\'e polynomial $P_t(\mathcal{N}_{C,\infty},\partial_\pm \mathcal{N}_{C,\infty})$),  is independent of the choice of system of strict Morse neighbourhoods and of the Riemannian metric on $M$, 
and that 
$M$  satisfies the descending Morse inequalities
$$  P_t(M) = \sum_{C \in \mathcal{D}}  P_t(\mathcal{N}_{C,\infty},\partial_- \mathcal{N}_{C,\infty})  \, - \, (1+t)R_{\downarrow}(t)\,\,\,\, \,\,\,  \mbox{ where } R_{\downarrow}(t) \geqslant 0
$$
and 
 the ascending Morse inequalities
$$  P_t(M) = \sum_{C \in \mathcal{D}}  P_t(\mathcal{N}_{C,\infty},\partial_+\mathcal{N}_{C,\infty})  \, - \, (1+t)R_{\uparrow}(t)\,\,\,\, \,\,\,  \mbox{ where } R_{\uparrow}(t) \geqslant 0.
$$
Here we write $R(t) \geqslant 0$ when all the coefficients of the polynomial $R(t)$ are non-negative.

When $f$ is Morse--Bott then the normal bundle to $C \in \mathcal{D}$ decomposes as $TM|^+_C \oplus TM|^-_C$ where the Hessian of $f$ is positive definite on $TM|^+_C$ and negative definite on $TM|^-_C$. We can take $\mathcal{N}_{C,n}$ to be the image under the exponential map determined by the Riemannian metric of a product of disc bundles in $TM|^+_C$ and $TM|^-_C$, and then $\partial_\pm\mathcal{N}_{C,n}$ is the product of one disc bundle with the boundary of the other. If these bundles are orientable then $H_i(\mathcal{N}_{C,n},\partial_\pm \mathcal{N}_{C,n};\RR)$ is isomorphic to $H_{i-\text{rank}(TM|^\mp_C)}({C};\RR)$ and we recover the Morse inequalities (\ref{Morseineq}).

\begin{rem} When $M$ is oriented then by Poincar\'e duality
$P_t(M) = t^{\dim M} P_{(1/t)}(M)$, and it follows from Alexander-Spanier
duality (cf. \cite{Hatcher} Theorem 3.43) that  $P_t(\mathcal{N}_{C,n},\partial_- \mathcal{N}_{C,n})$ is equal to $ t^{\dim M} P_{(1/t)}(\mathcal{N}_{C,n},\partial_+ \mathcal{N}_{C,n})$ for any $n \geqslant 0$. Hence 
$P_t(\mathcal{N}_{C,\infty},\partial_- \mathcal{N}_{C,\infty}) =t^{\dim M}  P_{(1/t)}(\mathcal{N}_{C,\infty},\partial_+ \mathcal{N}_{C,\infty})$, and
so the descending and ascending Morse inequalities are equivalent. 
\end{rem}

\begin{rem} \label{rem:OS}
For our proof of the Morse inequalities to hold we can weaken the assumption that there are only finitely many connected components of the critical locus of $f$. It is enough to assume that its set of critical {values} $\text{Critval}(f)= f(\text{Crit}(f))$ is an isolated (or equivalently finite) subset of $\RR$; then our proof shows that the Morse inequalities are true when $\mathcal{D}$ is replaced with $\{ \text{Crit}(f) \cap f^{-1}(c) : c \in \text{Critval}(f) \}$. 
Indeed since the first version of this paper was written it has been proved in \cite{OS} Lemma 3.8 that if $\text{Critval}(f)$ is finite then so is the set $\mathcal{L}(c)$ of connected components of $f^{-1}(c)$ for any $c \in \text{Critval}(f)$; using this our proof shows that the Morse inequalities are true when $\mathcal{D}$ is replaced with $\{ \text{Crit}(f) \cap L :  L \in  \mathcal{L}(c)  \mbox{ for some } c \in \text{Critval}(f)\}$.

The paper \cite{OS} studies the Reeb space of a smooth function $f:M \to \RR$ with $\text{Critval}(f)$ finite. It shows that the Reeb space, which is the space of connected components of the level sets of $f$ endowed with the natural quotient topology, is a finite (multi-)graph whose vertices correspond to those connected components which meet $\text{Crit}(f)$ of level sets of $f$. The resulting Reeb graph is related to but not in general the same as the quiver (or directed graph) we will associate to $f$ together with the Riemannian structure on $M$. It is shown in \cite{OS} that the Reeb graph can be decorated by associating to each edge the diffeomorphism type of a closed connected manifold of dimension $\dim M -1$ which is a connected component of a level set, and associating to each vertex the diffeomorphism type of a compact connected manifold with boundary which is a connected component of $f^{-1}[c-\epsilon, c+ \epsilon]$ for some $c \in \text{Critval}(f)$ and sufficiently small $\epsilon >0$. Moreover any such decorated graph, with suitable compatibility conditions on the decorations, can be realised as the decorated Reeb graph of a smooth function $f: M \to \RR$ with $\text{Critval}(f)$ finite for some $M$, and if $M$ is fixed then any continuous map from $M$ to a connected acyclic (multi-)graph inducing a surjection on fundamental groups is homotopic to the quotient map to the Reeb graph of a smooth function  $f: M \to \RR$ with $\text{Critval}(f)$ finite.
\end{rem}

We will see that the Morse inequalities provided by Theorem \ref{mainthm} can be proved in several different ways which generalise the different approaches to the classical Morse inequalities. In particular they follow from a 
spectral sequence of multicomplexes 
 which refines the information given by the vector spaces $H_i(\mathcal{N}_{C,\infty},\partial_\pm \mathcal{N}_{C,\infty};\FF)$ and in the Morse--Smale situation reduces to the Morse--Witten complex. 

When $(f,g)$ is Morse--Smale then there is an associated graph $\Gamma$ (more precisely a multi-digraph or quiver) embedded in $M$, with vertices given by the critical points of $f$ and arrows from  a critical point $p$ of index $i$ to a critical point $q$ of index $i-1$ given by the (finitely many) gradient flow lines from $p$ to $q$ (cf. for example \cite{CohenNorbury}). The Morse--Witten complex is the differential module with basis $\text{Crit}(f)$ and differential 
$$\partial p = \sum_{q \in \text{Crit}(f)} n(p,q) q $$
where $n(p,q)$ is the number of flow lines from $p$ to $q$, counted with signs determined by suitable choices of local orientations. It is usually regarded as a complex via the grading by index, but we can also regard it as a differential module supported on the quiver $\Gamma$, where $\Gamma$ is graded by the function $f$. From the viewpoint of the representation theory of quivers the Morse--Witten complex is a representation of $\Gamma$ over $\FF$ such that each vertex is represented by a copy of $\FF$, each arrow by the identity on $\FF$ (up to a suitable sign) and the sum of all the arrows is the differential $\partial$. Equivalently we can replace multiple arrows from $p$ to $q$ with just one, represented by the number of flow lines from $p$ to $q$ counted with signs. When we grade this representation using the homological degree and the critical value instead of the index, then we have a multicomplex supported on $\Gamma$. 

In our more general situation we can define a quiver $\Gamma$ with vertices $\{v_C:C \in \mathcal{D}\}$ labelled by the connected components $C$ of $\text{Crit}(f)$, and an arrow from $C_+$ to $C_-$ for each connected component of 
$$\{ x \in M : \{ \psi_{t}(x): t \leqslant 0\} \mbox{  has a limit point in $C_+$ and } \{ \psi_{t}(x): t \geqslant 0\} \mbox{  has a limit point in $C_-$} \}$$
when this subset is closed in $f^{-1}(f(C_-),f(C_+))$. Here $\psi_t(x)$  describes the downwards gradient flow for $f$ from $x$ at time $t$. 

This quiver  $\Gamma$ equipped with the real-valued function $v_C \mapsto f(C)$ on its set of vertices is an $\RR$-quiver  in the sense that $\Gamma$ is graded by this function: if there is an arrow in $\Gamma$ from $v_{C_+}$ to $v_{C_-}$ then $f(C_+) > f(C_-)$. However the quiver gives us more information about the existence of gradient flow lines from $C_+$ to $C_-$ than the requirement that $f(C_+) > f(C_-)$: such gradient flow lines can only exist if there is a path in the quiver $\Gamma$ from $v_{C_+}$ to $v_{C_-}$.

The analogues of Morse--Witten complexes are in general given by 
spectral sequences 
 abutting to $H_*(M;\FF)$, 
with $E_1$ pages given by multicomplexes supported on the quivers $\Gamma$ such that a vertex $v_C$ is represented  by the vector space
$H_*(\mathcal{N}_{C,\infty},\partial_- \mathcal{N}_{C,\infty};\FF)$; in the Morse--Smale case described above the spectral sequence degenerates after the $E_1$ page.

The layout of this paper is as follows. In $\S$\ref{sec:neighbourhoods} 
we study systems of Morse neighbourhoods and give a first proof of the main theorem, Theorem \ref{mainthm}; this generalises the proof of the classical Morse inequalities via attaching handles. In $\S$\ref{sec:witten}  we describe another approach using Witten's deformation technique which was the original motivation for this project. 
 Morse stratifications and Morse covers with associated spectral sequences are defined in $\S$\ref{sec:stratifications} and $\S$\ref{sec:spectral}, giving  further proofs of the Morse inequalities, and the quiver $\Gamma$ described above is introduced.
 In $\S$\ref{sec:quivers} we consider 
  multicomplexes supported on acyclic quivers and generalise the spectral sequences constructed in $\S$\ref{sec:spectral}.  
 In $\S$\ref{sec:multi} we study these further and in particular consider their relationship with the quivers and multicomplexes associated to Morse--Smale perturbations of a smooth function $f:M \to \RR$ whose critical locus has finitely many connected components. $\S$\ref{sec:examples} considers some examples, including the height function on a torus in $\RR^3$ pictured in Figure \ref{fig:funnytorus} and a related function on a compact oriented surface of genus two. Finally
$\S$\ref{sec:applications} discusses some possible extensions of our results. 

The authors would like to thank Vidit Nanda, Graeme Segal, Edward Witten and Jon Woolf for valuable discussions, and to dedicate this paper to the memory of Michael Atiyah, without whom this collaboration would not have been possible. This work was supported in part by AFOSR award FA9550-16-1-0082 and DOE award DE-SC0019380.

\section{Systems of Morse neighbourhoods} \label{sec:neighbourhoods}

Let $M$ be a compact  Riemannian manifold without boundary, and let $f \colon M \to \RR$  be a smooth function on $M$  whose critical locus $\text{Crit}(f)$ has finitely many connected components, with its set of critical {values} $\text{Critval}(f)= f(\text{Crit}(f)) = \{c_1,\ldots, c_p\}$. 

In order to state our main result, Theorem \ref{mainthm}, we first define the notion of a system of Morse neighbourhoods for $f$. The definition (see Definition \ref{Mn} below) is slightly more general than that given in the introduction; we will use the terminology \lq a system of {\em strict} Morse neighbourhoods' for the latter.

\begin{defn} \label{strictMn}
Let $f \colon M\to \RR$ be a smooth function on the compact  Riemannian manifold $M$ such that the set $\mathcal{D}$ of connected components of the critical locus $\text{Crit}(f)$ is finite.
A {\it system of strict Morse neighbourhoods} for $f$ is given by
$$ \{ \mathcal{N}_{C,n} : C \in \mathcal{D} , \,\,\, n \geqslant 0 \} $$
such that 
if $f$ takes value $c$ on $C \in \mathcal{D}$, and $ n \geqslant 0$, then 

(a) $\mathcal{N}_{C,n}$ is a neighbourhood of $C $ in $M$ containing no other critical points for $f$, with $\mathcal{N}_{C,n+1} \subseteq (\mathcal{N}_{C,n})^{\circ}$ 
and $\bigcap_{m \geqslant 0} \mathcal{N}_{C,m} = C $;

(b) $\mathcal{N}_{C,n}$ is a compact submanifold of $M$ with corners (locally modelled on $[0,\infty) ^2 \times \RR^{\dim M -2}$) and has boundary
$$\partial \mathcal{N}_{C,n} = \partial_+ \mathcal{N}_{C,n} \cup \partial_- \mathcal{N}_{C,n}$$
where $\partial_\pm \mathcal{N}_{C,n}$ is a compact submanifold of $M$ with boundaries and  the corners of $\mathcal{N}_{C,n}$ are given by
$$ \partial(\partial_+\mathcal{N}_{C,n}) \, = \, \partial(\partial_- \mathcal{N}_{C,n}) \, = \, (\partial_+ \mathcal{N}_{C,n}) \cap  (\partial_- \mathcal{N}_{C,n}) \, = \, f^{-1}(c) \cap \partial \mathcal{N}_{C,n};$$

(c) the gradient vector field $\grad(f)$ on $M$ associated to its Riemannian metric $g$  satisfies\\
(i) $(\partial_+ \mathcal{N}_{C,n})^\circ = \partial _+ \mathcal{N}_{C,n} \setminus  \partial(\partial_+\mathcal{N}_{C,n})  \subseteq f^{-1}(c,\infty)$ with the restriction of $\grad (f)$ to $(\partial_+ \mathcal{N}_{C,n})^\circ$ pointing inside $\mathcal{N}_{C,n}$, and  \\
(ii) $(\partial_- \mathcal{N}_{C,n})^\circ = \partial _- \mathcal{N}_{C,n} \setminus  \partial(\partial_-\mathcal{N}_{C,n})  \subseteq f^{-1}(-\infty,c)$ with the restriction of $\grad (f)$ to $(\partial_- \mathcal{N}_{C,n})^\circ$ pointing outside $\mathcal{N}_{C,n}$.

\end{defn}

\begin{rem}
Analysis on manifolds with corners is not straightforward (cf. \cite{Grieser,Joyce,Joyce1,MM,Melrose}) but we will not need the subtleties of this theory.
\end{rem}

\begin{defn} \label{smoothMn}
We will refer to a system of strict Morse neighbourhoods such that, for all $(C,n)$, the boundary $\partial \mathcal{N}_{C,n} = \partial_+ \mathcal{N}_{C,n} \cup \partial_- \mathcal{N}_{C,n}$ is a submanifold of $M$ (and hence $\mathcal{N}_{C,n}$ is a submanifold with boundary), as a {\it system of smooth Morse neighbourhoods}.
\end{defn}

\begin{figure}[t]
\includegraphics[width = 0.4\linewidth]{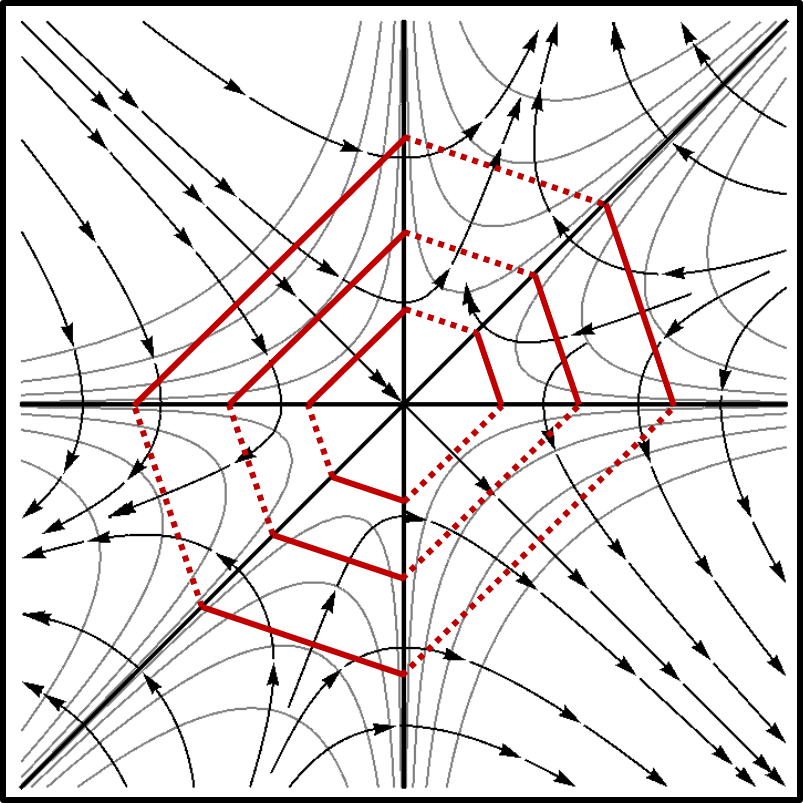}
\centering
\caption{Strict Morse neighbourhoods for $f$ given locally by $f(x,y) = x y\, (x-y)$.}
\label{fig:strict}
\end{figure}

Suppose that we have a system of strict Morse neighbourhoods for $f$.
Using the gradient flow $\grad(f)$, downwards on $f^{-1}(c,\infty)$ and upwards on $f^{-1}(-\infty, c)$ until $\partial \mathcal{N}_{C,n+1} \cup f^{-1}(c)$ is reached, we obtain a retraction from $\mathcal{N}_{C,n}$ to $\mathcal{N}_{C,n+1} \cup (f^{-1}(c) \cap \mathcal{N}_{C,n})$ which takes $\partial_- \mathcal{N}_{C,n}$ to $\partial_- \mathcal{N}_{C,n+1} \cup (f^{-1}(c) \cap (\mathcal{N}_{C,n}\setminus \mathcal{N}_{C,n+1}))$, and so induces an isomorphism of relative homology groups
$$H_i(\mathcal{N}_{C,n},\partial_- \mathcal{N}_{C,n};\FF) \to H_i(\mathcal{N}_{C,n+1} \cup (f^{-1}(c) \cap \mathcal{N}_{C,n}),\partial_- \mathcal{N}_{C,n+1}\cup (f^{-1}(c) \cap (\mathcal{N}_{C,n}\setminus \mathcal{N}_{C,n+1}));\FF).$$
Excision (\cite{Hatcher} Thm 2.20), together with the fact that $f^{-1}(c) \cap \mathcal{N}_{C,n}$ is diffeomorphic near $f^{-1}(c) \cap \partial \mathcal{N}_{C,n}$ to the product of $f^{-1}(c) \cap \partial \mathcal{N}_{C,n}$ with an interval,
 then gives isomorphisms from
$$H_i(\mathcal{N}_{C,n+1} \cup (f^{-1}(c) \cap \mathcal{N}_{C,n}),\partial_- \mathcal{N}_{C,n+1}\cup (f^{-1}(c) \cap (\mathcal{N}_{C,n}\setminus \mathcal{N}_{C,n+1}));\FF) $$ to $ H_i(\mathcal{N}_{C,n+1},\partial_- \mathcal{N}_{C,n+1};\FF),$
and by composition we obtain isomorphisms of relative homology
\begin{equation} \label{phistrict} \phi_i^{m,n}: H_i(\mathcal{N}_{C,n},\partial_- \mathcal{N}_{C,n};\FF) \to H_i(\mathcal{N}_{C m},\partial_- \mathcal{N}_{C m};\FF)
\end{equation} when $m>n$. 

Next we will define a system of Morse neighbourhoods without the strictness condition. This allows the boundary of a Morse neighbourhood to decompose into three submanifolds instead of just two; in addition to $\partial_+ \mathcal{N}_{C,n}$ and  $\partial_- \mathcal{N}_{C,n}$ which are transverse to the gradient flow of $f$, another submanifold $\partial_\perp \mathcal{N}_{C,n}$ is allowed which is invariant under the gradient flow.

\begin{defn} \label{Mn}
A {\it system of Morse neighbourhoods} for $f$ is given by
$ \{ \mathcal{N}_{C,n} : C \in \mathcal{D} , \,\,\, n \geqslant 0 \} $
such that 
if $ n \geqslant 0$  and $f$ takes value $c$ on $C \in \mathcal{D}$ then 

(a) $\mathcal{N}_{C,n}$ is a neighbourhood of $C $ in $M$ containing no points of $\text{Crit}(f) \setminus \mathcal{D}$, with $\mathcal{N}_{C,n+1} \subseteq (\mathcal{N}_{C,n})^{\circ}$ 
 and $\bigcap_{m \geqslant 0} \mathcal{N}_{C,m} = C $;

(b) $\mathcal{N}_{C,n}$ is a compact submanifold of $M$ with corners (locally modelled on $[0,\infty) ^2 \times \RR^{\dim M -2}$) and has boundary
$$\partial \mathcal{N}_{C,n} = \partial_+ \mathcal{N}_{C,n} \cup \partial_- \mathcal{N}_{C,n} \cup \partial_\perp \mathcal{N}_{C,n},
\,\,\, \mbox{ with } \,\,\,  \partial_+ \mathcal{N}_{C,n} \cap \partial_- \mathcal{N}_{C,n} \cap \partial_\perp \mathcal{N}_{C,n} = \emptyset,$$
where $\partial_+ \mathcal{N}_{C,n}$ , $\partial_- \mathcal{N}_{C,n}$ and $\partial_\perp \mathcal{N}_{C,n}$ are compact submanifolds of $M$ with boundaries forming the corners of $\mathcal{N}_{C,n}$ while $f$ is constant on
$   (\partial_+ \mathcal{N}_{C,n}) \cap  (\partial_- \mathcal{N}_{C,n}) \, = \, \partial(\partial_+\mathcal{N}_{C,n}) \, \cap \, \partial(\partial_- \mathcal{N}_{C,n}), $ 
with value $c$, and on each of
$ \partial(\partial_\pm \mathcal{N}_{C,n}) \, \cap \, \partial(\partial_\perp \mathcal{N}_{C,n}) \, = \, (\partial_\pm \mathcal{N}_{C,n}) \cap  (\partial_\perp \mathcal{N}_{C,n})  
;$

(c)  the gradient vector field $\grad(f)$ on $M$ associated to its Riemannian metric $g$  satisfies\\
(i) $(\partial_+ \mathcal{N}_{C,n})^\circ = \partial _+ \mathcal{N}_{C,n} \setminus  \partial(\partial_+\mathcal{N}_{C,n})  \subseteq f^{-1}(c,\infty)$ with the restriction of $\grad (f)$ to $(\partial_+ \mathcal{N}_{C,n})^\circ$ pointing inside $\mathcal{N}_{C,n}$;  \\
(ii) $(\partial_- \mathcal{N}_{C,n})^\circ = \partial _- \mathcal{N}_{C,n} \setminus  \partial(\partial_-\mathcal{N}_{C,n})  \subseteq f^{-1}(-\infty,c)$ with the restriction of $\grad (f)$ to $(\partial_- \mathcal{N}_{C,n})^\circ$ pointing outside $\mathcal{N}_{C,n}$;
\\
(iii) $\partial_\perp \mathcal{N}_{C,n}$ is contained in the union of the trajectories under the gradient flow $\grad (f)$ of $f^{-1}(c) \cap \partial \mathcal{N}_{C,n}$, which is a submanifold of codimension one in (the smooth part of) $f^{-1}(c)$.

\end{defn}

\begin{rem} A system of Morse neighbourhoods $\{ \mathcal{N}_{C,n}: C \in \mathcal{D},  n \geqslant 0 \} $ is strict if and only if $\partial_\perp \mathcal{N}_{C,n}$ is empty, for each $C$ and $n$.
\end{rem}

\begin{rem}
We can, and usually will, assume that each Morse neighbourhood $\mathcal{N}_{C,n}$ is connected, by replacing $\mathcal{N}_{C,n}$ with its connected component containing $C$.
\end{rem}

Finally it is useful to define another special type of Morse neighbourhoods; these will be called {\em cylindrical} Morse neighbourhoods. We will see that systems of cylindrical Morse neighbourhoods always exist, and that they allow us to build systems of strict Morse neighbourhoods. 

\begin{defn} \label{cylindricalMn}
 A {\it system of cylindrical Morse neighbourhoods} for $f$ is given by
$$ \{ \mathcal{N}_{C,n} : C \in \mathcal{D} , \,\,\, n \geqslant 0 \} $$
such that 
if $f$ takes the value $c$ on $C \in \mathcal{D}$ and if $ n \geqslant 0$ then 

(a) $\mathcal{N}_{C,n}$ is a neighbourhood of $C $ in $M$ containing no other critical points for $f$, with $\mathcal{N}_{C,n+1} \subseteq (\mathcal{N}_{C,n})^{\circ}$ 
 and $\bigcap_{m \geqslant 0} \mathcal{N}_{C,m} = C $;

(b) $\mathcal{N}_{C,n}$ is a compact submanifold of $M$ with corners (locally modelled on $[0,\infty) ^2 \times \RR^{\dim M -2}$) and has boundary
$$\partial \mathcal{N}_{C,n} = \partial_+ \mathcal{N}_{C,n} \cup \partial_- \mathcal{N}_{C,n} \cup \partial_\perp \mathcal{N}_{C,n}$$
where $\partial_+ \mathcal{N}_{C,n}$ , $\partial_- \mathcal{N}_{C,n}$ and $\partial_\perp \mathcal{N}_{C,n}$ are compact submanifolds of $M$ with boundary, and

(c)  for some $\delta_+ = \delta_+ (c,n) >0$ and $\delta_- = \delta_-(c,n) >0$ \\
(i)
$\partial_+ \mathcal{N}_{C,n} \subseteq f^{-1}(c + \delta_+)$  and $\partial_- \mathcal{N}_{C,n} \subseteq f^{-1}(c - \delta_-)$;
\\
(ii) $\partial_\perp \mathcal{N}_{C,n}$ is the intersection of $f^{-1}[c-\delta_-, c+\delta_+]$ with the trajectories under the gradient flow $\grad (f)$ of $f^{-1}(c) \cap \partial \mathcal{N}_{C,n}$, which is a submanifold of codimension one in (the smooth part of) $f^{-1}(c)$.

\end{defn}

\begin{figure}[t]
\includegraphics[width = 0.4\linewidth]{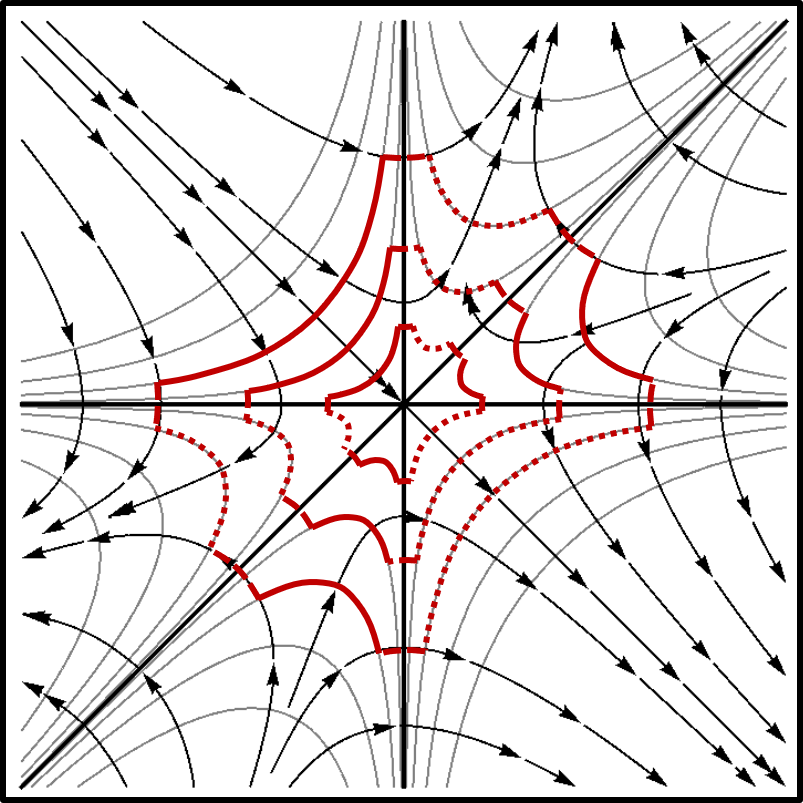}
\centering
\caption{Cylindrical Morse neighbourhoods for $f$ given locally by $f(x,y) = x y\, (x-y)$.}
\label{fig:cylindrical}
\end{figure}

Now suppose that we have a system of cylindrical Morse neighbourhoods for $f$ as above. By combining the gradient flow $\grad(f)$ downwards on $f^{-1}[c+\delta_+(c,n+\!1),\infty)$ and upwards on $f^{-1}(-\infty, c-\delta_-(c,n+\!1))$, 
we obtain a retraction from $\mathcal{N}_{C,n}$ onto 
$$\mathcal{N}_{C,n} \cap f^{-1}([c-\delta_-(c,n+\!1), c+\delta_+(c,n+\!1])$$
which takes $\partial_- \mathcal{N}_{C,n}\cup (\partial_\perp \mathcal{N}_{C,n} \cap f^{-1}[c,\infty))$ to 
its intersection with $f^{-1}([c-\delta_-(c,n+\!1), c+\delta_+(c,n+\!1])$. 
Similarly the closure of the complement of $\mathcal{N}_{C,n+\!1}$ in $\mathcal{N}_{C,n} \cap f^{-1}([c-\delta_-(c,n+\!1), c+\delta_+(c,n+\!1)])$ is diffeomorphic via the gradient flow to the product with the interval $[c-\delta_-(c,n+\!1), c+\delta_+(c,n+\!1)]$ of the complement of $\mathcal{N}_{C,n+\!1}$ in $\mathcal{N}_{C,n} \cap f^{-1}(c)$, which in turn is diffeomorphic near $\partial \mathcal{N}_{C,n+\!1} \cap f^{-1}(c)$ to the product of $\partial \mathcal{N}_{C,n+\!1} \cap f^{-1}(c)$ with an interval.
The retraction induces an isomorphism of relative homology groups from 
$ H_i(\mathcal{N}_{C,n},\partial_{-\perp} \mathcal{N}_{C,n};\FF),$ where $$ \partial_{-\perp} \mathcal{N}_{C,n} = \partial_- \mathcal{N}_{C,n} \cup (\partial_\perp \mathcal{N}_{C,n} \cap f^{-1}(-\infty,c]),$$ to $$H_i(
\mathcal{N}_{C,n} \cap f^{-1}([c-\delta_-(c,n+\!1), c+\delta_+(c,n+\!1)]), \mathcal{N}_{C,n} \cap f^{-1}(c-\delta_-(c,n+\!1
)\cup (\partial_\perp \mathcal{N}_{C,n} \cap f^{-1}[c-\delta_+(c,n+\!1),c])
;\FF)$$
which is isomorphic to $H_i(\mathcal{N}_{C,n+1} \cup (\mathcal{N}_{C,n} \cap f^{-1}(c)), \partial_{-\perp}\mathcal{N}_{C,n+1} \cup (\mathcal{N}_{C,n} \cap f^{-1}(c)\setminus \mathcal{N}_{C,n+1})
;\FF)$. Using excision (\cite{Hatcher} Thm 2.20) this is isomorphic to $ H_i(\mathcal{N}_{C,n+1},\partial_{-\perp} \mathcal{N}_{C,n+1});\FF) $.
Thus by composition there are induced isomorphisms
\begin{equation} \label{phicylindrical} \phi_i^{m,n}: H_i(\mathcal{N}_{C,n},\partial_{-\perp} \mathcal{N}_{C,n};\FF) \to H_i(\mathcal{N}_{C m},\partial_{-\perp} \mathcal{N}_{C m};\FF) \end{equation} when $m>n$.  
When $\mathcal{N} = \{ \mathcal{N}_{C,n} : C \in \mathcal{D}, n \geqslant 0 \} $ is any system of Morse neighbourhoods,
then combining this construction with that of (\ref{phistrict}) gives us isomorphisms of relative homology
\begin{equation} \label{phi} \phi_i^{m,n}: H_i(\mathcal{N}_{C,n},\partial_\pm \mathcal{N}_{C,n};\FF) \to H_i(\mathcal{N}_{C m},\partial_\pm \mathcal{N}_{C m};\FF)
\end{equation} when $m>n$. 

\begin{rem} \label{compare}
Suppose that $ \{ \mathcal{N}_{C,n} : C \in \mathcal{D}, n \geqslant 0 \}$ is a system of Morse neighbourhoods for $f$ with respect to a Riemannian metric $g$ on $M$, and that $ \{ \mathcal{N}_{C,n} : C \in \mathcal{D}, n \geqslant 0 \}$ is a system of Morse neighbourhoods for $f$ with respect to a Riemannian metric $g'$ on $M$. We have $\bigcap_{n \geqslant 0} \mathcal{N}_{C,n} = C  = \bigcap_{n \geqslant 1} \mathcal{N}'_{C,n}$ and so, by compactness, for each $n$ and $C$ there exists $m(C,n)$ such that 
$$\mathcal{N}'_{C,m(C,n)} \subseteq (\mathcal{N}_{C,n})^{\circ}.$$
If $m \geqslant m(C,n)$, then by constructing a metric which coincides with $g$ in a neighbourhood of $\partial \mathcal{N}_{C,n}$ and with $g'$ in a neighbourhood of $\partial \mathcal{N}'_{C,m}$ and using the construction of the isomorphisms given at (\ref{phi}), we obtain  isomorphisms of relative homology
$$\psi_i^{m,n}: H_i(\mathcal{N}_{C,n},\partial_- \mathcal{N}_{C,n};\FF) \to H_i(\mathcal{N}'_{C,m},\partial_{-\perp} \mathcal{N}'_{C,m};\FF).
$$
These induce isomorphisms 
$$ H_i(\mathcal{N}_{C,\infty},\partial_- \mathcal{N}_{C,\infty};\FF) \to H_i(\mathcal{N}'_{C,\infty},\partial_{-\perp} \mathcal{N}'_{C \infty};\FF)
$$
when, as at (\ref{Hlimit}), 
we define $H_i(\mathcal{N}_{C,\infty},\partial_\pm \mathcal{N}_{C,\infty};\FF)$ with isomorphisms 
$\phi^{\infty,n}_i : H_i(\mathcal{N}_{C,n},\partial_\pm \mathcal{N}_{C,n};\FF) \to H_i(\mathcal{N}_{C,\infty},\partial_\pm \mathcal{N}_{C,\infty};\FF)$
compatible with $\phi^{m,n}_i$ for all $m,n,i$ by taking the limit as $m \to \infty$ in (\ref{phi}).
\end{rem}

\begin{rem} \label{modify}
Given any system of cylindrical Morse neighbourhoods $ \{ \mathcal{N}_{C,n} : C \in \mathcal{D}, n \geqslant 0 \}$, we can modify each $\mathcal{N}_{C,n}$ near $\partial_\perp \mathcal{N}_{C,n}$ to obtain a system of strict Morse neighbourhoods $$ \{ \mathcal{N}'_{C,n} : C \in \mathcal{D}, n \geqslant 0 \}$$ 
such that $\mathcal{N}'_{C,n} \subseteq \mathcal{N}_{C,n}$ and $\mathcal{N}'_{C,n} \cap f^{-1}(c) = \mathcal{N}_{C,n} \cap f^{-1}(c)$, 
with a homotopy equivalence taking the pair $(\mathcal{N}_{C,n},\partial_{-\perp} \mathcal{N}_{C,n})$ to the pair
$({N}'_{C,n},\partial_- {N}'_{C,n})$  for each $c$ and $n$, and fixing $f^{-1}(c)$ pointwise. These induce isomorphisms
from $H_i({N}'_{c,n},\partial_- {N}'_{c,n};\FF)$ to  $H_i(\mathcal{N}_{C,n},\partial_{-\perp} \mathcal{N}_{C,n};\FF)$. Combining these and their inverses with  the isomorphisms $\phi_i^{m,n}$ defined at (\ref{phicylindrical}) above, we obtain the maps (\ref{phistrict})
for $ \{ \mathcal{N}'_{C,n} : C \in \mathcal{C}, n \geqslant 1 \}$, which are therefore isomorphisms. 

 If we wish, we can construct $\mathcal{N}'_{C,n}$ in such a way that the angle between $\partial_+ \mathcal{N}'_{C,n}$ and $\partial_- \mathcal{N}'_{C,n}$ is $\pi$; then $\partial \mathcal{N}'_{C,n} = \partial_+ \mathcal{N}'_{C,n} \cup \partial_- \mathcal{N}'_{C,n}$ is smooth and $\mathcal{N}'_{C,n}$ can be regarded as a submanifold of $M$ with boundary, rather than a submanifold with corners. We can thus obtain a system of smooth Morse neighbourhoods (Definition \ref{smoothMn}). In addition we can choose $ \{ \mathcal{N}'_{C,n} : C \in \mathcal{C}, n \geqslant 0 \}$ so that $f|_{\partial_{\pm}\mathcal{N}'_{C,n}}$ is a Morse function on the interior of $\partial_{\pm}\mathcal{N}'_{C,n}$ (with minimum/maximum on the boundary $\partial_-\mathcal{N}'_{C,n} \cap \partial_+\mathcal{N}'_{C,n}$). 
 
 We can also find a homotopy equivalence from $\mathcal{N}_{C,n}$ to itself taking $\partial_{-\perp} \mathcal{N}_{C,n}$ to 
 $\partial_{-} \mathcal{N}_{C,n}$, so that
 \begin{equation} \label{noperp} H_i(\mathcal{N}_{C,n}, \partial_{-\perp} \mathcal{N}_{C,n};\FF) \cong H_i(\mathcal{N}_{C,n}, \partial_- \mathcal{N}_{C,n};\FF).
 \end{equation} 
\end{rem}

\begin{rem}\label{Sardrem}
By applying Sard's theorem (\cite{Sard} Thm II.3.1) to the smooth function $|\!| \grad(f)   |\!|^2$ on the submanifold $f^{-1}(c) \setminus \text{Crit}(f) $ of $M$ for each $c \in \text{Critval}(f)$, we can find a sequence $(\epsilon_n(c))_{n \geqslant 1}$ of strictly positive real numbers which are regular values of $|\!| \grad(f)   |\!|^2$ on this submanifold and which tend to 0 as $n \to \infty$. We can also choose disjoint open neighbourhoods $\{U_C : C \in \mathcal{C}_c \}$ in $f^{-1}(c)$ of the critical sets $C \in \mathcal{C}_c$ contained in $f^{-1}(c)$.
There is some $n_0 = n_0(\underline{\epsilon}) >0$ such that for each $c \in \text{Critval}(f)$
$$  \{ x \in f^{-1}(c) : |\!| \grad(f)  (x) |\!|^2 \leqslant \epsilon_{n_0}(c) \} \subseteq \bigcup_{C \in \mathcal{C}_c} U_C.$$
We can then construct a system of cylindrical Morse neighbourhoods $\{ \mathcal{N}^{\underline{\epsilon}}_{C,n}: C \in \mathcal{C}, n\geqslant 0 \}$ such that
$$ \mathcal{N}^{\underline{\epsilon}}_{C,n} \cap f^{-1}(c) = \{ x \in f^{-1}(c) : |\!| \grad(f)  (x) |\!|^2 \leqslant \epsilon_{n_0 + n}(c) \} \cap U_C.$$
\end{rem}

The following proposition follows from Remarks \ref{compare}, \ref{modify} and \ref{Sardrem}.

\begin{prop} \label{Mnexist}
Any  smooth function $f \colon M \to \RR$ on a Riemannian manifold $M$  whose critical locus $\text{Crit}(f)$ has finitely many connected components has a system of strict Morse neighbourhoods. Moreover,
if $\mathcal{N} =  \{ \mathcal{N}_{C,n} : C \in \mathcal{C} , \,\,\, n \geqslant 0 \} $ is any system of  Morse neighbourhoods for $f$, then 
the vector spaces 
$H_i(\mathcal{N}_{C,n},\partial_\pm \mathcal{N}_{C,n};\RR)$, up to canonical isomorphism,
 and thus the Poincar\'e polynomials $P_t(\mathcal{N}_{C, \infty},\partial_- \mathcal{N}_{C, \infty})$ and
$P_t(\mathcal{N}_{C, \infty},\partial_+ \mathcal{N}_{C, \infty})$ defined as
 $$P_t(\mathcal{N}_{C, \infty},\partial_\pm \mathcal{N}_{C, \infty}) =  \sum_{i \geqslant 0} t^i \dim_\FF H_i(\mathcal{N}_{C,n},\partial_\pm \mathcal{N}_{C,n};\FF),$$ are independent of $n$, and also independent of the
choice of system of Morse neighbourhoods, and of the Riemannian metric on $M$. 
\end{prop}

We can now state our generalised version of the Morse inequalities.

\begin{thm} \label{mainthm}
Let $M$ be a compact  Riemannian manifold without boundary, and suppose that  $f \colon M \to \RR$ is a smooth function 
 whose critical locus $\text{Crit}(f)$ has finitely many connected components.  Suppose also that  $ \{ \mathcal{N}_{C,n} : C \in \mathcal{C} , n \geqslant 0 \} $ is a system of Morse neighbourhoods for $f \colon M \to \RR$. Then the Betti numbers of $M$ satisfy the descending Morse inequalities
$$  P_t(M) = \sum_{C \in \mathcal{C}}  P_t(\mathcal{N}_{C,\infty},\partial_- \mathcal{N}_{C,\infty})  \, - \, (1+t)R_{\downarrow}(t)\,\,\,\, \,\,\,  \mbox{ where } R_{\downarrow}(t) \geqslant 0
$$
and 
 the ascending Morse inequalities
$$  P_t(M) = \sum_{C \in \mathcal{C}}  P_t(\mathcal{N}_{C,\infty},\partial_+\mathcal{N}_{C,\infty})  \, - \, (1+t)R_{\uparrow}(t)\,\,\,\, \,\,\,  \mbox{ where } R_{\uparrow}(t) \geqslant 0.
$$
\end{thm} 
\begin{proof}
The first proof we will give of these Morse inequalities follows the approach in the classical case given by attaching handles. By using the gradient flow we see that the topology of $f^{-1}(-\infty, a]$ is unchanged as $a$ increases, except when $a$ passes through a critical value $C \in \mathcal{D}$, and then disjoint Morse neighbourhoods $\mathcal{N}_{C,n}$ for $C \in \mathcal{D}_c$ are attached along $\partial_-\mathcal{N}_{C,n}$. Thus if $c_- = c - \delta_-$ and $c_+ = c + \delta_+$ where $\delta_\pm >0$ are sufficiently small, there is an isomorphism
$$ H_i(f^{-1}(-\infty,c_+],f^{-1}(-\infty,c_-];\FF)  \cong \bigoplus_{C \in \mathcal{D}_{c}} H_i(\mathcal{N}_{C,\infty},\partial_- \mathcal{N}_{C,\infty};\FF)
$$
induced by the gradient flow and excision (cf. Remark \ref{compare}), and therefore a long exact sequence
$$ \cdots \to \bigoplus_{C \in \mathcal{D}_{c}} H_{i+1}(\mathcal{N}_{C,\infty}, \partial_-\mathcal{N}_{C,\infty};\FF) \xrightarrow{\delta_{i,c}}
 H_i(f^{-1}(-\infty,c_-];\FF) \to H_i(f^{-1}(-\infty,c_+];\FF) $$ $$\to \bigoplus_{C \in \mathcal{D}_{c}} H_{i}(\mathcal{N}_{C,\infty}, \partial_-\mathcal{N}_{C,\infty};\FF) \to \cdots $$
which tells us that
$ \dim H_i(f^{-1}(-\infty,c_+];\FF)$ is equal to $$\dim H_i(f^{-1}(-\infty,c_-];\FF) - \dim \text{im} \delta_{i,c}  - \dim \text{im} \delta_{i-1,c}  + \sum_{C \in \mathcal{D}_{c}} \dim H_{i}(\mathcal{N}_{C,\infty}, \partial_-\mathcal{N}_{C,\infty};\FF) 
$$
for each $i$ and $c$.
By combining these long exact sequences for $c \in \text{Critval}(f)$ we obtain the descending Morse inequalities
$$  P_t(M) = \sum_{C \in \mathcal{D}}  P_t(\mathcal{N}_{C,\infty},\partial_- \mathcal{N}_{C,\infty})  \, - \, (1+t)R_{\downarrow}(t)\,\,\,\, \,\,\,  \mbox{ where } R_{\downarrow}(t) \geqslant 0;
$$
here $R_{\downarrow}(t)$ is given by $\sum_{i\geqslant 0} t^i \sum_{c \in \text{Critval}(f)}  \dim \text{im} \delta_{i,c}$. 
The proof of the ascending Morse inequalities is similar.
\end{proof}

\begin{rem}
As was noted in the introduction, when $M$ is oriented the descending Morse inequalities are equivalent to 
 the ascending Morse inequalities
$$  P_t(M) = \sum_{C \in \mathcal{D}}  P_t(\mathcal{N}_{C \infty},\partial_+\mathcal{N}_{C \infty})  \, - \, (1+t)R_{\uparrow}(t)\,\,\,\, \,\,\,  \mbox{ where } R_{\uparrow}(t) \geqslant 0
$$
since by Poincar\'e duality
$P_t(M) = t^{\dim M} P_{(1/t)}(M)$, and by Alexander-Spanier
duality (cf. \cite{Hatcher} Theorem 3.43) we have $P_t(\mathcal{N}_{C \infty},\partial_- \mathcal{N}_{C \infty}) = t^{\dim M}  P_{(1/t)}(\mathcal{N}_{C \infty},\partial_+ \mathcal{N}_{C \infty})$. 
\end{rem}

\section{Morse inequalities via Witten's deformation technique} \label{sec:witten}

In this section we will give another proof of our main result, Theorem \ref{mainthm}, using the alternative approach to proving the classical Morse inequalities (with real coefficients) pioneered by Witten \cite{Witten82}. He made use of supersymmetric quantum mechanics, specifically a supersymmetric non-linear sigma model with target space $M$, which as before we assume to be a compact Riemannian manifold (without boundary). The Hilbert space of this theory is canonically isomorphic to the space of differential forms $\Omega^*(M)$ with the supercharges corresponding to the exterior derivative $d$ and codifferential $d^*$. The Hamiltonian is therefore the Laplacian (or Laplace-Beltrami operator) $\Delta = dd^* + d^* d$  on the space of differential forms. This is a positive, essentially self-adjoint operator on the $L^2$ closure of the space of differential forms on $M$. The zero energy states are harmonic forms, and so, by the usual arguments from Hodge theory, the zero energy subspace is canonically isomorphic to the de Rham cohomology of $M$. 

By adding a superpotential $t f$ (for a Morse function $f$ and constant $t>0$), Witten deformed the theory, while preserving the supersymmetry, with the supercharges now $d_t^W = e^{-tf} d e^{tf}$ and $d_t^{W*} = e^{tf} d^* e^{-tf}$ respectively. Since the map $\omega \mapsto e^{tf}\omega$ is invertible, the cohomology, and hence the number of zero energy states, is unchanged by this deformation. The new Hamiltonian is given by the deformed Laplacian $\Delta_t = d_t^W d_t^{W*} + d_t^{W*} d_t^W$ and contains a potential term $t^2  |\!| \grad(f)   |\!|^2$ which means that, for $t \gg 1$, low energy states must be localised near critical points of $f$. As a result, the Hamiltonian can be approximated for $t \gg 1$ by a direct sum over supersymmetric harmonic oscillators associated to each critical point. Within this approximation, there exists a single zero energy state for each oscillator, which will be a $p$-form if the Hessian of the associated critical point has $p$ negative eigenvalues. Since the exact zero energy states must form a subspace of these approximate zero energy states, the Morse inequalities follow.

In this section, we outline a similar approach to rederive Theorem \ref{mainthm}. This mirrors the strategy used in \cite{BS} to prove Novikov inequalities in the presence of \lq minimal degeneracy' (cf.\! \cite{K}) using a deformed Laplacian. We construct extended Morse neighbourhoods $\tilde{\mathcal{N}}_{C}$  by attaching cylindrical ends (on which the function grows quadratically) to smooth Morse neighbourhoods $\mathcal{N}_{C,n}$. It is enough to consider a single smooth Morse neighbourhood $\mathcal{N}_{C,n_C}$ with $n_C$ fixed for each connected component $C \subseteq \text{Crit}(f)$.

As we shall see, the deformed $L^2$ cohomology of the extended Morse neighbourhoods is isomorphic to $H^*(\mathcal{N}_{C,n_C},\partial_- \mathcal{N}_{C,n_C};\RR)$. At low energies and large $t$, a deformed Laplacian $\Delta_t(M)$ acting on the manifold $M$ can be modelled by a direct sum of the same deformed Laplacian $\oplus_{C \in \mathcal C} \Delta_t(\tilde{\mathcal{N}}_{C,n_C}) $ on a set of extended Morse neighbourhoods. As in Witten's original argument, the zero energy states of the model Hamiltonian give an upper bound on the number of zero energy states of the deformed Laplacian and hence we will obtain another proof of Theorem \ref{mainthm}.

\begin{rem} \label{boundaryregionpropn}  The gradient flow of $f$ combined with local analysis near the corners $f^{-1}(c) \cap \partial \mathcal{N}_{C,n_C}$ of  $\mathcal{N}_{C,n_C}$ 
can be used to show that  there exists a smooth embedding
$$ \phi : \partial \mathcal{N}_{C,n_C} \times (1-\varepsilon, 1] \to \mathcal{N}_{C,n_C} $$
for sufficiently small $\varepsilon > 0$, such that 

(a) $\phi\left(\partial \mathcal{N}_{C,n_C} \times \{1\}\right) = \partial \mathcal{N}_{C,n_C}$ and

(b) for all $x \in \partial \mathcal{N}_{C,n_C}$ and $s \in (1-\varepsilon, 1]$, then $f(\phi(x,s)) = c + s^2 f_0 (x)$, where $c = f(C)$ and $c + f_0 (x)$ is the restriction of $f$ to  $\partial \mathcal{N}_{C,n_C}$.
\end{rem}

\begin{defn} \label{extendedMn}
Let $\mathcal{N}_{C,n_C}$ be a smooth Morse neighbourhood (Definition \ref{smoothMn}) of some connected component $C \in \mathcal{D}$ of the critical set $\text{Crit}(f)$. Then a corresponding {\em extended Morse neighbourhood} is 
$$\tilde{\mathcal{N}}_{C,n_C} = \mathcal{N}_{C,n_C} \cup (\partial \mathcal{N}_{C,n_C} \times (1-\varepsilon,\infty))$$
where we identify $\partial \mathcal{N}_{C,n_C} \times (1-\varepsilon,1]$ with its image in $\mathcal{N}_{C,n_C}$ under a smooth embedding $\phi$ as in Remark \ref{boundaryregionpropn}. We shall refer to 
$\mathcal{T}_C = \phi(\partial \mathcal{N}_{C,n_C} \times (1-\varepsilon,1])$ as a boundary region of $\mathcal{N}_{C,n_C}$
and to $\tilde {\mathcal {T}}_{C} = (\partial \mathcal{N}_{C,n_C} \times (1-\varepsilon,\infty))$ as the extended boundary region of 
 $\tilde{\mathcal{N}}_{C,n_C}$.

The smooth function $f$ extends to this extended Morse neighbourhood by defining
$$f(x,s) = c + s^2 f_0(x)$$
where $c = f(C)$ for all $x \in \partial \mathcal{N}_{C,n_C}$ and $s \in (1-\varepsilon,\infty)$. As before $c + f_0(x)$ is the restriction of $f$ to $\partial \mathcal{N}_{C,n_C}$. 
\end{defn}
We can now construct a Riemannian metric $\tilde g^{C}$ on $\tilde{\mathcal{N}}_{C,n_C}$ 
 such that
\begin{enumerate} [label=(\roman*)]
\item  $\tilde g^{C}$ agrees with the metric $g$ upon restriction to $\mathcal{N}_{C,n_C}\textbackslash \mathcal T_{C}$;
\item on $\partial \mathcal{N}_{C,n_C} \times  [1 - 3\varepsilon / 4, \infty)$, we have
$$\tilde g^{C}  = s^2 g |_{ \partial \mathcal{N}_{C,n_C}} \oplus g_{\RR},$$
where $s \in  [1 - 3\varepsilon / 4, \infty)$ and $g_{\RR} = ds^2$ is the standard  metric on $\RR$.
\end{enumerate}
We can also choose a smooth metric $\tilde g$ on $M$ which agrees with $\tilde g^{C}$ within each Morse neighbourhood  $\mathcal{N}_{C,n_C}$, by modifying $g$ close to the Morse neighbourhoods $\mathcal{N}_{C,n_C}$.

\begin{remark}
 The Morse neighbourhoods $\{\mathcal{N}_{C,n_C} : C \in \mathcal C\}$ continue to satisfy the properties of Morse neighbourhoods with respect to this new metric $\tilde g$.
\end{remark}
\begin{defn} Following \cite{BS}
let
$$\tilde \Omega^*_f ( \tilde{\mathcal{N}}_{C}) = \{ \zeta \in L^2 \Omega^* (\tilde{\mathcal{N}}_{C}) : d_f \zeta \in L^2 \Omega^* (\tilde{\mathcal{N}}_{C})\},$$
where $d_f = e^{-f} d e^{f}$ and $L^2 \Omega^* (\tilde{\mathcal{N}}_{C})$ is the space of square integrable differential forms on $\tilde{\mathcal{N}}_C$ with square-integrability defined using the standard inner product of differential forms induced by the metric $\tilde g_C$.

Then the deformed $L^2$ cohomology $H^* (\tilde \Omega^*_f ( \tilde{\mathcal{N}}_{C,n}) ,d_f)$ is defined to be the cohomology of the complex
$$ 0 \rightarrow \tilde \Omega^0_f ( \tilde{\mathcal{N}}_{C}) \overset{d_f}{\longrightarrow} \tilde \Omega^1_f ( \tilde{\mathcal{N}}_{C}) \overset{d_f}{\longrightarrow} \dots \overset{d_f}{\longrightarrow} \tilde \Omega^n_f ( \tilde{\mathcal{N}}_{C}) \longrightarrow 0.
$$
\end{defn}

\begin{lemma} \label{lemma:L2}
The deformed $L^2$ cohomology $H^* (\tilde \Omega^*_f ( \tilde{\mathcal{N}}_{C}) ,d_f)$ 
 is isomorphic to the relative de Rham cohomology $H^i(\mathcal{N}_{C,n_C}, \partial_- \mathcal{N}_{C,n_C})$. 
\end{lemma}

\begin{proof}
The extended Morse neighbourhood $\tilde{\mathcal{N}}_{C}$ is a manifold with cylindrical end such that the closed one-form $df$ and the metric $\tilde g^C$ are both homogeneous of degree 2 at infinity. Hence, by Proposition 5.3 of \cite{BS}, the deformed $L^2$ cohomology
$$ H^* (\tilde \Omega^*_f ( \tilde{\mathcal{N}}_{C,n}) ,d_f) \cong H^*( \tilde{\mathcal{N}}_{C,n},f^{-1}((-\infty,b])$$
for any $b$ just less than $f(C)$. But by retraction under gradient flow and excision this is in turn isomorphic to $H^i(\mathcal{N}_{C,n}, \partial_- \mathcal{N}_{C,n})$.
\end{proof}
Our next step is to construct a one-parameter family of deformed Laplacians $\Delta_t(\tilde{\mathcal{N}}_{C})$ on each of the extended Morse neighbourhoods $\tilde{\mathcal{N}}_{C}$, together with a similar one-parameter family of Laplacians $\Delta_t(M)$ on the entire manifold $M$ and show that eigenstates of $\Delta_t(M)$ whose energy vanish if $t \to \infty$ are in one-to-one correspondence with zero energy states of $\oplus_C \Delta_t({\tilde{\mathcal{N}}_{C}})$.

This will require us to prove that there do \emph{not} exist any non-zero energy states of $\oplus_C \Delta_t({\tilde{\mathcal{N}}_{C}})$ whose energy vanishes in the $t \to \infty$ limit. We therefore define the deformed Laplacians based not only on a deformed exterior derivative $d_t$, but also on a $t$-dependent metric $\tilde g^C_t$. By doing so, we will be able to show that the spectrum of $\Delta_t(\tilde{\mathcal{N}}_{C})$ is independent of $t$, and hence any eigenstate of $\Delta_t(\tilde{\mathcal{N}}_{C})$ that has zero energy in the $t \to \infty$ limit will also have zero energy at any finite $t$. Using basic Hodge theory combined with Lemma \ref{lemma:L2}, the space of these zero energy states will be isomorphic to $H^i(\mathcal{N}_{C,n_C}, \partial_- \mathcal{N}_{C,n_C})$.

Let $\tau_t: \tilde{\mathcal{N}}_{C,n_C} \to \tilde{\mathcal{N}}_{C,n_C}$ form a one parameter family of diffeomorphisms for $t> 0$ with
$
\tau_t (p) = p
$
 for all $p \in \tilde{\mathcal{N}}_{C,n_C} \textbackslash\tilde{\mathcal{T}}_{C}$, while
$$
\tau_t (x,s)  = (x, \xi_t(s)),
$$
 for all $(x,s) \in \tilde{\mathcal{T}}_{C}$, where the smooth monotonically-increasing function $\xi_t: \RR \to \RR$ satisfies $\xi_t(s) = s$ for $s\leq 1 - \varepsilon/2$ and $\xi_t(s) = \sqrt{t} s$ for $s \geq 1 - \varepsilon / 4$. 

For 
$s > 1 - \varepsilon/4$ we have 
\begin{align} \label{eq:pullback}
\frac{1}{t}\tau_t^* \tilde g =  \tilde g \,\,\,\,\text{and} \,\,\,\,\tau_t^* f = t (f -c) + c
\end{align}
where $c = f(C)$. 
\begin{defn} \label{defn:laplaceNC} Let $d_t = e^{- \tau_t^* f} d e^{\tau_t^* f}$. The $t$-dependent Riemannian metric $\tilde g^C_t = \tau_t^* \tilde g^C$ induces a Hodge star operator $\star_t$ and hence a co-differential
$$
d_t^* = (-1)^k e^{\tau_t^* f} \star_t^{-1} d \star_t e^{-\tau_t^* f},
$$
which is the adjoint of $d_t$ with respect to the inner product on $\Omega^*(\tilde{\mathcal{N}}_C)$ induced by the metric $\tilde g^C_t$. We can then define the deformed Laplacian $$\Delta_t(\tilde{\mathcal{N}}_{C,n_C}) = d_t^* d_t + d_t d_t^*.$$

Let $\tilde \Omega^*_t = \{ \omega \in L^2\Omega(\tilde{\mathcal{N}}_C) : \Delta_t \omega \in L^2\Omega(\tilde{\mathcal{N}}_C)\}$. 
 Let $\bar \Delta_t(\tilde{\mathcal{N}}_C)$  be the closure of the restriction of $\Delta_t(\tilde{\mathcal{N}}_{C,n_C})$ to $\tilde \Omega^*_t$ in the $L^2$ completion $\overline{L^2\Omega^*(\tilde{\mathcal{N}}_C)}$ of $L^2\Omega^*(\tilde{\mathcal{N}}_C)$.
\end{defn}

\begin{rem}
Here square integrability with respect to the metric $\tilde g^C_t$ is equivalent to
 square integrability with respect to $\tilde g^C$, since these metrics are bounded by constant positive scalar multiples of each other.
\end{rem}

Using the diffeomorphism invariance of the exterior derivative, we see that 
$$ 
\tau_t^* (\Delta_1(\tilde{\mathcal{N}}_{C,n_C}) f) = \Delta_t(\tilde{\mathcal{N}}_{C,n_C}) (\tau_t^* f)
$$ 
and so the spectrum of $\Delta_t(\tilde{\mathcal{N}}_{C,n_C})$ is independent of $t$. Since $\Delta_t(\tilde{\mathcal{N}}_{C,n_C})$ is an elliptic operator with discrete spectrum (Proposition 4.5 of \cite{BS}), then by standard Hodge theory arguments (Proposition 5.2 of \cite{BS}) and Lemma \ref{lemma:L2} we have 
$$ 
 \label{eq:dimker}
\text{Ker}(\Delta_t(\tilde{\mathcal{N}}_{C,n_C})) \cap \Omega^p(\tilde{\mathcal{N}}_{C,n_C})= \text{Ker}(\Delta_1(\tilde{\mathcal{N}}_{C,n_C}))  \cap \Omega^p(\tilde{\mathcal{N}}_{C,n_C}) = H^p (\tilde \Omega^*_f ( \tilde{\mathcal{N}}_{C,n_C}) ,d_f) = H^p(\mathcal{N}_{C,n_C}, \partial_- \mathcal{N}_{C,n_C}).
$$ 
We can then define a deformed Laplacian on the manifold $M$ as follows. 
\begin{defn}
Let the smooth function $f_t: M \to \RR$ satisfy $f_t(p) = t f(p)$ for $p \in M \textbackslash \bigcup_C \mathcal{N}_{C,n_C}$, while $f_t (x) = (t - 1) c + \tau_t^* f$  for $x \in \mathcal{N}_{C,n_C}$. 
Let the metric $\tilde g_t$ satisfy $$\tilde g_t(p) = t \tilde g(p)$$ for $p \in M \textbackslash \bigcup_C \mathcal{N}_{C,n_C}$, while
$\tilde g_t = \tau_t^* \tilde g$ on  $\mathcal{N}_{C,n}$.
Smoothness at $\partial \mathcal{N}_{C,n_C}$ follows from \eqref{eq:pullback}.
Define  $d_t = e^{- f_t} d e^{f_t}$, with the codifferential $d_t^*$ defined, analogously to Definition \ref{defn:laplaceNC}, as the adjoint of $d_t$ with respect to the inner product on $\Omega^*(M)$ induced by $\tilde g_t$. Explicitly
$$
d_t^* = (-1)^k e^{f_t}\star^{-1}_t d \star e^{-f_t},
$$
where $\star_t$ is the Hodge star operator induced by the metric $\tilde g_t$. We therefore define
$$
\Delta_t(M) = d_t d_t^* + d_t^* d_t.
$$
As in Definition \ref{defn:laplaceNC}, we can also define $\bar \Delta_t(M)$ to be the closure of $\Delta_t(M)$ in the $L^2$ completion $\overline{ \Omega^*(M)}$ of $\Omega^*(M)$.
\end{defn}
On both $M$ and $\tilde{\mathcal{N}}_{C,n_C}$
$ 
(\phi, \Delta_t \phi) = (d_t \phi, d_t \phi) + (d_t^* \phi, d_t^* \phi) \geq 0,
$ 
 for all states $\phi$, so
both $\Delta_t(M)$ and $\Delta_t(\tilde{\mathcal{N}}_{C,n_C})$ are positive, densely-defined symmetric operators. It is well-known that their closures $\bar \Delta_t(M)$ and $\bar \Delta_t(\tilde{\mathcal{N}}_{C,n_C})$ are self-adjoint \cite{BMS}.
\begin{remark}
By the elliptic regularity theorem, if $\bar \Delta_t(M) \phi = \lambda \phi$ (respectively  $\bar \Delta_t^{\tilde{\mathcal N}_C} \phi = \lambda \phi$) for $\phi \in \overline{ \Omega^*(M)}$ (respectively  $\phi \in \overline{ L^2\Omega^*(\tilde{\mathcal{N}}_C)}$), then $\phi \in \Omega^*(M)$ (respectively  $\phi \in  L^2\Omega^*(\tilde{\mathcal{N}}_C)$).
\end{remark}

\begin{defn}
For each $C \in \mathcal{D}$, let $J_C: \mathcal{N}_{C,n_C} \to [0,1]$ be  a smooth function such that
for all $x \in \mathcal{N}_{C,n_C} \textbackslash  \mathcal{T}_C 
$ 
and
for all $x = \phi(y,s) \in \mathcal{T}_C$ with $s \leq 1 - \varepsilon/4$ we have $J_C(x) = 1,
$
but
for all $x = \phi(y,s) \in \mathcal{T}_C$ with $s \geq 1 - \varepsilon/8$ we have $J_C(x) = 0.
$ 
We shall also use $J_C$ to denote the functions $J_C: \tilde{\mathcal{N}}_{C,n_C} \to [0,1]$ and  $J_C: M \to [0,1]$ that agree with $J_C : \mathcal{N}_{C,n_C} \to [0,1]$ on $\mathcal{N}_{C,n_C}$ and are zero elsewhere. Let $J_0 : M \to [0,1]$ satisfy
\begin{align}
J_0 = \sqrt{ 1 - \sum_C J_C^2 },
\end{align}
while for all $C \in \mathcal{D}$ we define $\bar J_{C} : \tilde{\mathcal{N}}_{C,n_C} \to [0,1]$ by
\begin{align}
\bar J_C = \sqrt{ 1 - J_C^2}.
\end{align}
Then $\{ J_C^2 , \bar J_C^2\}$ and $\{J_0^2\} \cup \{ J_C^2 : C \in \mathcal{D}\}$ form partitions of unity for $\tilde{\mathcal{N}}_{C,n_C}$ and $M$ respectively.
\end{defn}
\begin{remark}
For all $\psi, \phi \in \Omega^*(\mathcal{N}_{C,n_C})$ and all $C \in \mathcal{D}$ 
$$
(J_C \phi, \Delta_t(M) J_C \psi) =( J_C \phi, \Delta_t(\tilde{\mathcal{N}}_{C,n_C}) J_C \psi),
$$
where on the left (respectively right) hand side $J_C \psi$ and $J_C \phi$ are treated as differential forms on $M$ (respectively $\tilde{\mathcal{N}}_{C,n_C}$) with support only in $\mathcal{N}_{C,n_C}$.
\end{remark}
To complete the proof of Theorem \ref{mainthm}, we need to show that the operator $\oplus_C \Delta_t^{\tilde{\mathcal{N}}_{C,n}}$ approximates the operator $\Delta_t(M)$ at large $t$ in the following sense:
\begin{lemma} \label{lem:hamiltonianestimate}
Let $\bar \Delta_{t,p}({\tilde{\mathcal{N}}_{C,n_C}})$ and $\bar \Delta_{t,p}(M)$ be the restriction of $\bar \Delta_{t}({\tilde{\mathcal{N}}_{C,n_C}})$ and $\bar \Delta_{t}(M)$ to $p$-forms. Moreover, let
$$
\bar \Delta_{t,p}^{\mathcal{D}} = \oplus_{C \in \mathcal{D}} \bar \Delta_{t,p}({\tilde{\mathcal{N}}_{C,n_C}}).
$$
Then
for sufficiently large $t$
$$\text{dim}(\text{Ker}(\bar \Delta_{t,p}^{\mathcal{D}}) = n_{p}$$
where $n_{p}$ is the number of eigenvalues of $\bar \Delta_{t,p}^M$ (counting multiplicities) with eigenvalue less than $1/\sqrt{t}$.
\end{lemma}
\begin{proof}
Let $\psi_C \in \text{Ker}(\bar \Delta_{t,p}({\tilde{\mathcal{N}}_{C,n_C}}))$ be normalised such that the inner product $(\psi_C , \psi_C)$ is 1.
Given a Hamiltonian H which is the sum of a Laplace-Beltrami operator plus any first order differential operator and a set of functions $\{J_i\}$ such that $\sum_i J_i^2 = 1$, the INS localisation formula (cf.\! \cite{BS} Lemma 8.2, \cite{CFKS,Shubin} Lemma 3.1) states that
$$
H = \sum_i J_i H J_i  + \frac{1}{2} \sum_i [J,[J,H]]= \sum_i J_i H J_i - \sum_i \lVert dJ_C \rVert^2.
$$
From this we see that
$$ (\bar J_C \psi_C, \bar\Delta_t(\tilde{\mathcal{N}}_{C,n_C})\bar J_C \psi_C) + (J_C \psi_C, \bar\Delta_t(\tilde{\mathcal{N}}_{C,n_C}) J_C \psi_C) = 2(\psi_C, \lVert dJ_C \rVert_t^2 \psi_C) = O(1/t), $$
where we have used the subscript $t$ to indicate that the norm $\lVert dJ_C \rVert_t$ is defined using the metric $\tilde g^C_t$. The last estimate follows because $dJ_C$ is independent of $t$ and bounded and $\tilde g_t = t \tilde g$ everywhere on the support of $dJ_C$. Hence
$$
 (J_C \psi_C, \bar\Delta_t(\tilde{\mathcal{N}}_{C,n_C}) J_C \psi_C)  = O(1/t).
 $$
Since $df \neq 0$ and $\tilde g_t = t \tilde g$ everywhere in the support of $\bar J_C$, we have
\begin{align}
\bar J_C \bar\Delta_t(\tilde{\mathcal{N}}_{C,n_C}) \bar J_C= t \bar J_C \tilde g^{-1}(df,df) \bar J_C + O(1) \geq \varepsilon t \bar J_C^2,
\end{align}
where the last inequality is true for sufficiently large $t$, given any fixed $\varepsilon < \inf_{\text{supp}(\bar J_C)} \tilde g^{-1}(df,df)$. Hence
\begin{align}
\varepsilon t (\bar J_C \psi_C, \bar J_C \psi_C) \leq (\bar J_C \psi_C, \Delta_t(\tilde{\mathcal{N}}_{C,n_C})\bar J_C \psi_C) = O(1/t)
\end{align}
and 
\begin{align}
(J_C \psi_C, J_C \psi_C) = 1 - (\bar J_C \psi_C, \bar J_C \psi_C) = 1 - O(1/t^2).
\end{align}
 
However the support of $J_C$ lies entirely within $\mathcal{N}_C$. Hence
\begin{align}
\frac{(J_C \psi_C, \bar\Delta_t^{M} J_C \psi_C)}{(J_C \psi_C, J_C \psi_C)} = \frac{(J_C \psi_C, \bar\Delta_t(\tilde{\mathcal{N}}_{C,n_C}) J_C \psi_C)}{(J_C \psi, J_C \psi_C)} = O(1/t).
\end{align}
Let 
$k_{t,p} = \text{dim}(\text{Ker}(\bar \Delta_{t,p}^{\mathcal{D}}) = \sum_{C \in \mathcal{D}} \text{dim}( \text{Ker}(\Delta_{t,p}({\tilde{\mathcal{N}}_{C,n_C}}))).$ 
We showed that $k_{t,p}$ was finite in \eqref{eq:dimker}. Let $\mu_n(\bar\Delta_{t,p}(M))$ for $n \in \ZZ_{>0}$ be defined by the following minimax formula
$$
\mu_n(\bar\Delta_{t,p}(M)) = \sup_{\zeta_1, \zeta_2, \dots\zeta_{n-1}} Q(\zeta_1,\zeta_2,\dots \zeta_{n-1};\bar\Delta_{t,p}(M)),
$$
where
$
Q(\zeta_1,\zeta_2,\dots \zeta_{n-1};\bar\Delta_{t,p}(M)) = \inf_\psi \{ (\psi, \bar\Delta_{t,p}(M) \psi) | \,\,\psi \in D(\bar\Delta_{t,p}(M)), \,\lVert \psi \rVert = 1,\,\forall i \,\,(\psi, \zeta_i) = 0 \}.
$
By the spectral theory of self-adjoint operators,
$$
\mu_n = \min (E_n, \inf \sigma_{ess}(\bar\Delta_{t,p}(M)),
$$
where $E_n$ is the $n$th eigenvalue (counting multiplicities) of $\bar\Delta_{t,p}(M)$ (or $E_n =\infty$ if there are fewer than $n$ eigenvalues) and $\sigma_{ess}(\bar\Delta_{t,p}(M))$ is the essential spectrum of $\bar\Delta_{t,p}(M)$. (In fact, since $M$ is compact, $\bar\Delta_{t,p}(M)$ has discrete spectrum). Since  the space
$$
V = \spn\{J_C \psi_C \,:\, \psi_C \in \text{Ker}(\bar \Delta_{t,p}({\tilde{\mathcal{N}}_{C,n_C}})) \,\, C \in \mathcal{D}\}
$$ is $k_{t,p}$-dimensional and satisfies
$$
\sup_{\psi \in V,\,\,\lVert \psi\rVert = 1} (\psi, \bar\Delta_{t,p}(M) \psi) \,\, =\,\, \max_{C \in \mathcal{D}}\,\,\,  \sup_{\psi \in \text{Ker}(\bar \Delta_{t,p}({\tilde{\mathcal{N}}_{C,n_C}}))} \frac{(J_C \psi, \bar \Delta_t(M) J_C \psi)}{(J_C \psi, J_C \psi)}\,\, = \,\, O(1/t),
$$
it follows that 
$
\mu_{k_{t,p}} = O(1/t).
$

We now show that $\mu_{k_{t,p}+1} > 1/\sqrt{t}$  for sufficiently large $t$. Let $V_{[0,2/\sqrt{t}]} \subseteq \overline{\Omega^p(M)}$ be the spectral subspace of $[0,2/\sqrt{t}]$ for the self-adjoint positive operator $\bar\Delta_{t,p}(M)$. We assume that $\text{dim}(V_{[0,2/\sqrt{t}]}) \geq k+1$ and derive a contradiction for sufficiently large $t$.

Let $\phi \in V_{[0,2/\sqrt{t}]}$ have $(\psi,\psi) = 1$. By almost identical arguments to the ones above
$$
 \sum_C (J_C \phi, \bar\Delta_t({M}) J_C \phi) + (J_0 \phi, \bar\Delta_t({M}) J_0 \phi) - 2\sum_C (\phi, \lVert dJ_C \rVert_t^2 \phi) \leq \frac{2}{\sqrt{t}}.
$$
Because $J_0 \Delta_t({M}) J_0 \geq \varepsilon t$ for sufficiently large $t$ at fixed sufficiently small $\varepsilon$,
$$
\sum_C( J_C \phi, J_C \phi) = 1 - (J_0 \phi, J_0 \phi) = 1 - O(1/t^{3/2})
$$
and
$$
\frac{(\oplus_C J_C \phi, \bar \Delta_{t,p}^{\mathcal{D}} \oplus_C J_C \phi)}{(\oplus_C J_C \phi, \oplus_C J_C \phi)} = \frac{\sum_C( J_C \phi, \Delta_t({M}) J_C \phi)}{\sum_C( J_C \phi, J_C \phi)} = O(1/t^{1/2}).
$$
However the spectrum of $\Delta_{t,p}^{\mathcal{D}}$ is discrete and independent of $t$. Hence, if $t$ is sufficiently large, then $2 \sqrt{t}$ will be less than the minimum non-zero eigenvalue of $\Delta_{t,p}^{\mathcal{D}}$. The only possibility would be to have
$$
\text{dim}(\text{Ker}(\bar \Delta_{t,p}^{\mathcal{D}})) \geq k_{t,p}+1,
$$
giving our desired contradiction. It follows that $\mu_{k_{t,p}+1} \geq 2\sqrt{t}$ for sufficiently large $t$ and $\bar \Delta_{t,p}^{\mathcal{D}}$ must have a discrete spectrum with exactly $k_{t,p}$ eigenvalues below $1/\sqrt{t}$, which completes the proof.
\end{proof}

\textit{Proof of Theorem \ref{mainthm}}
With Lemma \ref{lem:hamiltonianestimate} in hand, a proof of Theorem \ref{mainthm} follows straightfowardly. By standard Hodge-theoretic arguments, there is one zero energy state of $\bar \Delta_{t,p}^M$ for each cohomology class for the deformed exterior derivative $d_t$. However multiplication by $e^{-tf}$ gives an isomorphism between this deformed cohomology and the ordinary de Rham cohomology. From this, together with Lemmas \ref{lemma:L2} and \ref{lem:hamiltonianestimate}, we immediately obtain weak Morse inequalities
\begin{align} \label{eq:weak}
P_t(M) = \sum_{C \in \mathcal{D}}  P_t(\mathcal{N}_{C,\infty},\partial_- \mathcal{N}_{C,\infty}) - S(t),
\end{align}
where the coefficients of $P_t(M)$ count the zero energy states of $\bar \Delta_{t,p}(M)$, while $P_t(\mathcal{N}_{C,\infty},\partial_- \mathcal{N}_{C,\infty})$ counts the low energy states of $\bar \Delta_{t,p}(M)$, so $S(t)$ has non-negative coefficients.

The strong Morse inequalities (Theorem \ref{mainthm}) also follow by standard arguments, which we sketch here. From a physics perspective, the strong inequalities arise because non-zero energy states in a supersymmetric system always come in equal energy pairs, one bosonic and one fermionic \cite{Witten82a}. More specifically, let $\phi \in \Omega^*(M)$ satisfy $\Delta_t(M) \phi = E \phi$ for some $E > 0$. We can always write
$$
\phi = \frac{1}{E}(d_t^* d_t \phi + d_t d_t^* \phi).
$$
Since
$$[d_t, \Delta_t(M)] = [d_t^*, \Delta_t(M)] = 0,
$$
$d_t d_t^* \phi$ and $d_t^* d_t \phi$ are respectively exact and co-exact eigenstates with the same eigenvalue. Hence we can always choose an eigenbasis for the low energy states  of $\Delta_t(M)$ such that every non-zero energy eigenstate is either exact or co-exact. Given an $d_t$-exact $p$-form eigenstate $d\psi$, the $(p-1)$-form $d_t^*d_t\psi$ is a co-exact eigenstate with the same eigenvalue. Similarly given a co-exact $p$-form eigenstate $d_t^*\chi$, the $(p+1)$-form $d_t d_t^* \chi$ is an exact eigenstate with the same eigenvalue.

We therefore obtain isomorphisms between the co-exact non-zero energy eigenspaces of $p$-forms and the exact non-zero energy eigenspaces of $(p+1)$-forms. It follows that we can rewrite \eqref{eq:weak} with $S(t) = (1+t) R(t)$ where $R(t) = r_p t^p$ has the non-negative coefficients 
$$
r_p = \text{dim}\left(\spn\left\{ \phi \,:\, \Delta_t(M) \phi  = E \phi,\,d_t^* \phi = 0,\,0<E\leq 1/\sqrt{t}\right\}\right). 
$$
This gives us an alternative route to Theorem \ref{mainthm}. \hfill\qedsymbol

\section{Morse stratifications and Morse covers} \label{sec:stratifications}

This section 
 generalises the construction of Morse stratifications, and the resulting proof of the Morse inequalities, to the situation where $f \colon M \to \RR$ is any smooth real-valued function on a compact Riemannian manifold $M$  whose critical locus $\text{Crit}(f)$ has finitely many connected components, providing a third proof of Theorem \ref{mainthm}. It also associates  to suitable systems of Morse neighbourhoods open covers of $M$ (see Definition \ref{Ms}) and decompositions of $M$ into submanifolds with corners (see Remark \ref{decomp}). This will lead to yet another proof of Theorem \ref{mainthm} in $\S$4, and  in $\S$5 and $\S$6 to a refined version of this theorem which generalises the Morse-Witten complex.

As before let $\mathcal{D} = \bigsqcup_{c \in \text{Critval}(f)} \mathcal{D}_c$ be the finite set of connected components of $\text{Crit}(f)$, where
$$ \mathcal{D}_c = \{ C \in \mathcal{D} : C \subseteq f^{-1}(c)\}.$$
Let $\mathcal{N} = \{\mathcal{N}_{C,n} : C \in \mathcal{D} ,  n \geqslant 0 \} $ be a  system of strict Morse neighbourhoods for $f$ satisfying $\mathcal{N}_{C,n} \cap \mathcal{N}_{C',n'} = \emptyset $ unless $C=C'$.

\begin{defn} \label{defn3.1}
If $C \in \mathcal{D}$ and $j\geqslant 0$ we will say that the downwards gradient flow $\{ \psi_t(x): t \geqslant 0\}$ (where $\psi_0(x) = x$) for $f$ from $x$ meets $\mathcal{N}_{C,j} $ if there is some $t_0 \geqslant 0$ such that $\psi_{t_0}(x) \in \mathcal{N}_{C,j}$. 
 Let
$$\mathcal{W^+\!N}_{C,j} = \{ x \in M: \mbox{ the downwards gradient flow for $f$ from $x$ meets $\mathcal{N}_{C,j}$} \}$$
and 
$$\mathcal{W^-\!N}_{C,j} = \{ x \in M: \mbox{ the upwards gradient flow for $f$ from $x$ meets $\mathcal{N}_{C,j}$} \}.$$
Similarly let 
$$\mathcal{W^+\! N}^\circ_{C,j} = \{ x \in M: \mbox{ the downwards gradient flow for $f$ from $x$ meets $\mathcal{N}^\circ_{C,j}$} \}$$
and 
$$\mathcal{W^-\!N}^\circ_{C,j} = \{ x \in M: \mbox{ the upwards gradient flow for $f$ from $x$ meets $\mathcal{N}_{C,j}^\circ$} \}.$$
\end{defn}

\begin{lemma}  \label{lem3.1}
If $C \in \mathcal{D}$ and $j \geq 0$ then \\ (i)
$\mathcal{W^+\!N}_{C,j}$ is a locally closed submanifold of $M$ with corners, of codimension 0, having interior $\mathcal{W^+\! N}^\circ_{C,j} $, boundary $$\partial_-\mathcal{N}_{C,j} \cup \partial _\perp \mathcal{W^+\!N}_{C,j},
$$ where $ \partial _\perp \mathcal{W^+\!N}_{C,j}
$ is the union of the upwards trajectories under the gradient flow for $f$ of $\partial _- \mathcal{N}_{C,j} \cap \partial _+ \mathcal{N}_{C,j}
$, and corners given,  as for $\mathcal{N}_{C,n}$ itself, by  (the connected components of) the intersection $\partial _- \mathcal{N}_{C,j} \cap \partial _\perp \mathcal{W^+\!N}_{C,j}
= \partial _- \mathcal{N}_{C,j} \cap 
\partial _+ \mathcal{N}_{C,j}
$.
\\
(ii) The downwards Morse flow induces a retraction of 
$\mathcal{W^+\!N}_{C,j}$ onto $\mathcal{N}_{C,j}$.
\end{lemma}

\begin{rem} To extend Definition \ref{defn3.1} and Lemma \ref{lem3.1} to systems of Morse neighbourhoods which are not necessarily strict, we need to take 
 $ \partial _\perp \mathcal{W^+\!N}_{C,j}
$ to be the union of the upwards trajectories under the gradient flow for $f$ of the corners
$$(\partial _- \mathcal{N}_{C,j} \cap \partial _+ \mathcal{N}_{C,j}) \cup (\partial _- \mathcal{N}_{C,j} \cap \partial _\perp \mathcal{N}_{C,j}) \cup (\partial _\perp \mathcal{N}_{C,j} \cap \partial _+ \mathcal{N}_{C,j})$$ of $\mathcal{N}_{C,n}$ (or equivalently  the upwards trajectories under the gradient flow of 
$(\partial _- \mathcal{N}_{C,j} \cap \partial _+ \mathcal{N}_{C,j}) \cup
\partial _\perp \mathcal{N}_{C,j}$). Then the corners of $\mathcal{W^+\!N}_{C,j}$ are (the connected components of)
$\partial _- \mathcal{N}_{C,j} \cap \partial _\perp \mathcal{W^+\!N}_{C,j}
= (\partial _- \mathcal{N}_{C,j} \cap \partial _+ \mathcal{N}_{C,j}) \cup (\partial _- \mathcal{N}_{C,j} \cap \partial _\perp \mathcal{N}_{C,j}) 
$: the lower corners for  $\mathcal{N}_{C,n}$. 
\end{rem}

Now if $x \in M$ then since $M$ is compact the downwards  gradient flow $\{ \psi_t(x): t \geqslant 0\}$ for $f$ from $x$ has a limit point in $\text{Crit}(f)$, and so there is some $C \in \mathcal{D}$ such that $\{ \psi_t(x): t \geqslant 0\}$ meets every Morse neighbourhood $\mathcal{N}_{C,n}$ of $C$. It follows from the definition of a system of Morse neighbourhoods that if $x \in \mathcal{N}_{C,n}$ then $\{ \psi_t(x): t \geqslant 0\}$ leaves $\mathcal{N}_{C,n}$ if and only if it meets the open subset $f^{-1}(-\infty, c)$ of $M$, where $c = f(C)$, and this happens if and only if it has no limit point in $C$. Thus for each $x \in M$ there is a unique $C \in \mathcal{D}$ such that for every Morse neighbourhood $\mathcal{N}_{C,n}$ of $C$ the downwards gradient flow for $f$ from $x$ enters and never leaves $\mathcal{N}_{C,n}$.

\begin{defn} \label{defnwc} 
If $C \in \mathcal{D}$ let
$$ W^+_C = \{ x \in M : \mbox{ for every $n \geqslant 0$ the downwards gradient flow for $f$ from $x$ enters } $$
$$ \mbox{ and never leaves the Morse neighbourhood } \mathcal{N}_{C,n} \}. $$
Similarly let 
$$ W^-_C = \{ x \in M : \mbox{ for every $n \geqslant 0$ the upwards gradient flow for $f$ from $x$ enters } $$
$$ \mbox{
and never leaves the Morse neighbourhood } \mathcal{N}_{C,n} \}. $$
\end{defn} 

\begin{lemma}\label{lemdec}  (i) $ W^+_C = \bigcap_{j\geqslant 0} \mathcal{W^+\!N}_{C,j}$ 
and for any $\mathbb{j} : \mathcal{D} \to \NN$ we have $$M = \bigcup_{C \in \mathcal{D}}   \mathcal{W^+\!N}_{C,
\mathbb{j}(C)}^\circ.$$
(ii) M can be expressed as disjoint unions
$$ M = \bigsqcup_{C \in \mathcal{D}} W^+_C = \bigsqcup_{C \in \mathcal{D}} W^-_C = \bigsqcup_{C, \tilde{C} \in \mathcal{D}} W^+_C \cap W^-_{\tilde{C}} 
 $$
where $W^+_C \cap W^-_{{C}} = C$ and $W^+_C \cap W^-_{\tilde{C}} $ is empty unless $\tilde{C} = C$ or $f(\tilde{C}) > f(C)$, and $W^+_C$ and  $W^-_{{C}}$ are independent of the choice of system of Morse neighbourhoods.\\
(iii) For each $C, \tilde{C} \in \mathcal{D}$ the subsets $W^+_C$, $W^-_C$ and $W^+_C \cap W^-_{\tilde{C}} $ of $M$ are locally closed  with
$$ \overline{W^+_C} \subseteq W^+_C \cup \bigcup_{\tilde{C} \in \mathcal{D}, f(\tilde{C}) > f(C)} W^+_{\tilde{C}} \,\,\,\,\, \mbox{ and }  \,\,\,\,\, \overline{W^-_C} \subseteq W^-_C \cup \bigcup_{\tilde{C} \in \mathcal{D}, f(\tilde{C}) < f(C)} W^-_{\tilde{C}}.
$$
\end{lemma}
\begin{proof}
(i) and (ii) follow directly from the argument above. (iii) also follows since if $\psi_{t_0}(x) \in \mathcal{N}_{C,n+1} \subseteq \mathcal{N}_{C,n}^\circ$ for some $t_0 \in \RR$ then $\psi_{t_0}(y) \in \mathcal{N}_{C,n}^{\circ}  \subseteq \mathcal{N}_{C,n}$ for all $y$ in a neighbourhood of $x$.
\end{proof}

\begin{defn} \label{Ms}
We will call $\{W_C^+:C \in \mathcal{D}\}$ and $\{W_C^-:C \in \mathcal{D}\}$ {\it Morse stratifications} of $M$. 

We will also call 
$\{ \mathcal{W^+\!N}^\circ_{C,
\mathbb{j}(C)} :C \in \mathcal{D}\}$ and  $\{ 
 \mathcal{W^-\!N}^\circ_{{C},
{\mathbb{j}}({C})}:C \in \mathcal{D}\}$ and 
$$\{ \mathcal{W^+\!N}^\circ_{C,
\mathbb{j}(C)} 
\cap \mathcal{W^-\!N}^\circ_{\tilde{C},
\tilde{\mathbb{j}}(\tilde{C})}:C, \tilde{C} \in \mathcal{D}\}$$  {\it Morse covers} of $M$ when 
$\mathbb{j} : \mathcal{D} \to \NN$ 
and $\tilde{\mathbb{j}} : \mathcal{D} \to \NN$ with $\mathbb{j}(C) >\!> 1$ and  $\tilde{\mathbb{j}}(C) >\!> 1$ for all $C \in \mathcal{D}$.
\end{defn}

Now define a strict partial order on $\mathcal{D}$ by $C<\tilde{C}$ if $f(C) < f(\tilde{C})$, and extend it to a total order $>$ on $\mathcal{D}$. For each $C \in \mathcal{D} $ the subset 
$$U_C = M \setminus \bigcup_{\tilde{C}>C} W^+_{\tilde{C}}$$
is open in $M$ and contains $W^+_C$ as a closed subset. There is then a long exact sequence of homology
\begin{equation} \label{les}  \cdots \to H_{i +1}(U_C, U_C \setminus W^+_C ;\FF) \to
 H_i(U_C \setminus W^+_C;\FF)  \to  H_i(U_C; \FF) \to H_{i  }(U_C, U_C \setminus W^+_C;\FF) \to  \cdots .
\end{equation}
Moreover if $n$ is any natural number then $\mathcal{W^+\!N}_{C,n}$ is a neighbourhood of $W^+_C$ in $U_C$ and is a closed submanifold of $U_C$ with corners. Thus by excision
$$ H_{*}(U_C, U_C \setminus W^+_C;\FF) \cong H_{*}(\mathcal{W^+\!N}_{C,n}, \mathcal{W\!^+N}_{C,n} \setminus W^+_C;\FF) .$$
The downwards gradient flow for $f$ induces a retraction from $\mathcal{W^+\!N}_{C,n}$ to $\mathcal{N}_{C,n}$ and also a retraction from $\mathcal{W^+\!N}_{C,n} \setminus W^+_C$ to
$\partial_-\mathcal{N}_{C,n}$. These induce an isomorphism 
$$ H_{*}(\mathcal{W^+\!N}_{C,n}, \mathcal{W\!^+N}_{C,n} \setminus W^+_C;\FF) \cong
H_{*}(\mathcal{N}_{C,n}, \partial_-\mathcal{N}_{C,n};\FF) .$$
The Morse inequalities (Theorem \ref{mainthm}) now follow from the long exact sequence (\ref{les}), as in the first proof given in $\S$\ref{sec:neighbourhoods}. 

\begin{rem} \label{decomp}
Recall from Remark \ref{Sardrem} that we can choose a system of cylindrical Morse neighbourhoods  $\{ \mathcal{N}^{\underline{\epsilon}}_{C,n}: C \in \mathcal{D}, n\geqslant 0 \}$ of the critical sets $C \in \mathcal{D}$ such that if $c = f(C)$ then
$$ \mathcal{N}^{\underline{\epsilon}}_{C,n} \cap f^{-1}(c) = \{ x \in f^{-1}(c)\cap \tilde{U}_C : |\!| \grad(f)  (x) |\!|^2 \leqslant \epsilon(C)_{n} \} $$
where $\tilde{U}_C$ is a fixed neighbourhood of $C$ in $M$ and $\underline{\epsilon}$ is a function from $\mathcal{D}$ to the set of strictly decreasing sequences $(\epsilon_1, \epsilon_2, \ldots )$ of strictly positive real numbers converging to 0. This is possible by Sard's theorem  (\cite{Sard} Thm II.3.1) applied to  $|\!| \grad(f)   |\!|^2$ on  $f^{-1}(c) \setminus \text{Crit}(f)$, which allows us to choose such sequences $\underline{\epsilon}(C)$ consisting of regular values of the smooth function $|\!| \grad(f)   |\!|^2|_{f^{-1}(c) \setminus \text{Crit}(f)}$. We then define $ \mathcal{N}^{\underline{\epsilon}}_{C,n}$ to consist of those $x \in f^{-1}[c-\delta,c+\delta]$ (for $c = f(C)$ and $\delta>0$ sufficiently small) such that the gradient flow for $f$ from $x$ meets $ \mathcal{N}^{\underline{\epsilon}}_{C,n} \cap f^{-1}(c)$. As in Remark \ref{modify} we can further choose a system $\{\mathcal{N}_{C,n} : C \in \mathcal{D} ,  n \geqslant 0 \} $ of Morse neighbourhoods such that 
$$ \mathcal{N}_{C,n} \subseteq  \mathcal{N}^{\underline{\epsilon}}_{C,n} \,\,\, \mbox{ and } \,\,\,  \mathcal{N}_{C,n} \cap f^{-1}(c) =  \mathcal{N}^{\underline{\epsilon}}_{C,n} \cap f^{-1}(c)$$
for all $c \in \text{Critval}(f)$, all $C \in \mathcal{D}_c$ and all $n \geqslant 0$.

The gradient flow defines a diffeomorphism $g_{c,c'}$ from $f^{-1}(c)  \setminus \text{Crit}(f) $ to an open subset of $f^{-1}(c')$ for any $c'$ sufficiently close to $c$; if $c' \geqslant c$ (respectively $c' \leqslant c$) then this open subset is 
$$ f^{-1}(c') \setminus \bigcup_{C \in \mathcal{D}_c} W_C^+$$
(respectively $ f^{-1}(c') \setminus \bigcup_{C \in \mathcal{D}_c} W_C^-$). Composing $|\!| \grad(f)   |\!|^2$ on  $f^{-1}(c) \setminus \text{Crit}(f) $  with the diffeomorphism $g_{c,c'}$ defines a smooth function $|\!| \grad(f)   |\!|^2_{c,c'}$ on this open subset of $f^{-1}(c')$ (which extends to a continuous function on $f^{-1}(c')$ by assigning the value 0 on $ \bigcup_{C \in \mathcal{D}_c} W_C^\pm$). Then 
$$ \mathcal{N}^{\underline{\epsilon}}_{C,n} = \{x \in f^{-1}[c-\delta,c+\delta]: |\!| \grad(f)   |\!|^2_{c,f(x)} \leqslant \epsilon(C)_n \}.$$
Let $\text{Critval}(f) = \{c_1,\ldots,c_P\}$ where $c_1 < c_2 < \cdots < c_P$, 
and  for any subset $I$ of $\{1, \ldots ,P \}$ let
$$ U_{I,c'} =  f^{-1}(c') \setminus ( \bigcup_{C \in \mathcal{D}: c' \geqslant f(C) \geqslant \min\{c_i: i \in I\}} W^+_C \,\,\,  \cup 
\bigcup_{C \in \mathcal{D}: \max\{c_i: i \in I\} \geqslant f(C) \geqslant c'} W^-_C).
$$
Then for any $c' \in f(M)$ and $k \in \{1, \ldots , P\}$ there is a smooth function 
$$ |\!| \grad(f)   |\!|^2_{c_k,c'}: U_{\{k\},c'} \,\,\,  \to \, \RR
$$
obtained by composing $|\!| \grad(f)   |\!|^2|_{f^{-1}(c_k)}$ with the gradient flow from $U_{\{k\},c'}$ to $f^{-1}(c_k)$, which is a diffeomorphism onto an open subset of $f^{-1}(c_k)$. 
Similarly for any subset $I$ of $\{1, \ldots ,P \}$ such that the open subset $U_{I,c'}$ of $f^{-1}(c')$ is nonempty, there is a smooth function
$$F_{I,c'} = ( |\!| \grad(f)   |\!|^2_{c_i,c'})_{i \in I} : U_{I,c'} \to \RR^{|I|},$$
and by Sard's theorem the image in $\RR^{|I|}$ under $F_{I,c'}$ of its critical set $\text{Crit}(F_{I,c'})$ has Lebesgue measure 0. 
When $c'$ and $c'\!'$ lie in the same connected component of $\RR \setminus \text{Critval}(f)$ then $F_{I,c'}$ is the composition of $F_{I,c'\!'}$ with a diffeomorphism and so
$F_{I,c'}(\text{Crit}(F_{I,c'}))$ equals $F_{I,c'\!'}(\text{Crit}(F_{I,c'\!'}))$. 
\begin{figure}[t]
\includegraphics[width = 0.4\linewidth]{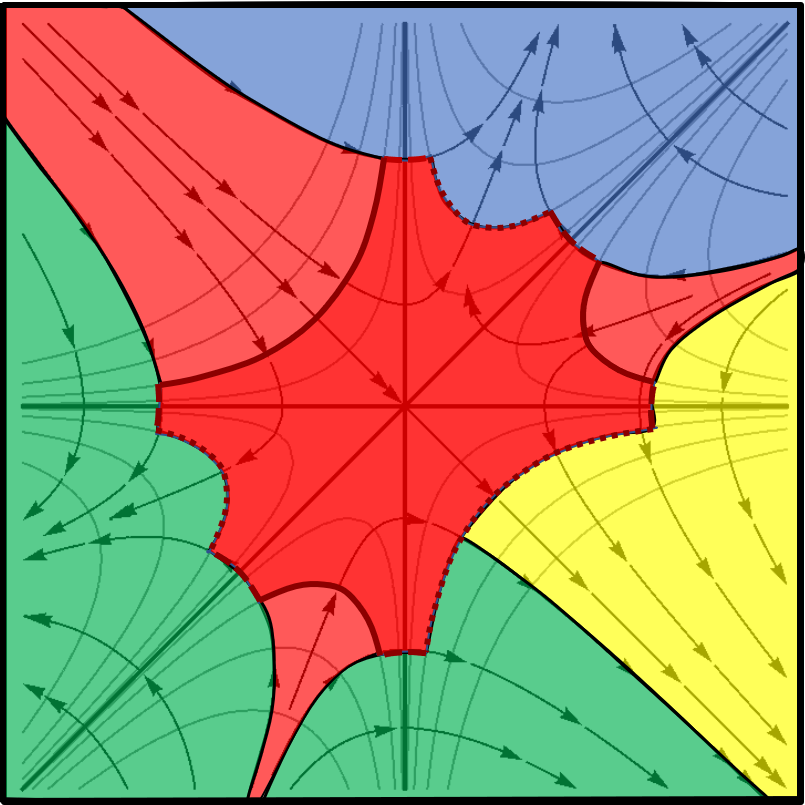}
\centering
\caption{A Morse decomposition for  $f(x,y) = x y\, (x-y)$ near 0.}
\label{fig:cover}
\end{figure}
Since the union of finitely many subsets of Lebesgue measure 0 has Lebesgue measure 0, we can choose a function
 $\underline{\epsilon}$ from $\mathcal{D}$ to the set of strictly decreasing sequences $(\epsilon_1, \epsilon_2, \ldots )$ of strictly positive real numbers converging to 0
 such that whenever $f(C_i) = c_i$ for $i \in I$ then if $(\epsilon(C_i)_{n_i})_{i \in I}$ lies in the image of $F_{I,c'}$ it is a regular value of $F_{I,c'}$. We can then choose  a system $\{\mathcal{N}_{C,n} : C \in \mathcal{D} ,  n \geqslant 0 \} $ of Morse neighbourhoods, as above, such that 
  near any $f^{-1}(c')$ with $c' \not\in \text{Critval}(f)$ the subsets $\mathcal{W^+\!N}_{C_i,n_i}$ are submanifolds (of codimension 0) with boundary which are invariant under the gradient flow and whose boundaries intersect transversely. 
  
  Suppose we are  given any  $\mathbb{j} : \mathcal{D} \to \NN$ with $\mathbb{j}(C) >\!> 1$ for all $C \in \mathcal{D}$. 
  Then by Lemma \ref{lemdec} we get a decomposition of $M$ as a union of closed submanifolds $$V^+_C = \mathcal{W^+\!N}_{C,
\mathbb{j}(C)} \setminus \bigcup_{\tilde{C} \in \mathcal{D}, f(\tilde{C}) > f(C)} 
 \mathcal{W^+\!N}_{\tilde{C},
\mathbb{j}(\tilde{C})}^\circ $$ with corners, all having the same dimension as $M$ and meeting along submanifolds (with corners) of their common boundaries, which can be regarded as approximations to the Morse stratifications. These submanifolds are invariant under the gradient flow of $f$ except where they meet the boundaries of the Morse neighbourhoods $\mathcal{N}_{C,n} $, and if we choose the Morse neighbourhoods to be strict or cylindrical then this decomposition restricts to a decomposition of any level set of $f$ into submanifolds with corners, meeting along submanifolds of their common boundaries. 
Similarly 
we get a decomposition of $M$ as a union of closed submanifolds $V^-_C = \mathcal{W^-\!N}_{C,
{\mathbb{j}} (C)} \setminus \bigcup_{\tilde{C} \in \mathcal{D}, f(\tilde{C}) < f(C)} 
 \mathcal{W^-\!N}_{\tilde{C},
{\mathbb{j}} (\tilde{C})}^\circ $ with corners, meeting along submanifolds of their common boundaries. The decompositions $M = \bigcup_{C \in \mathcal{D}} V^+_C$ and $M = \bigcup_{C \in \mathcal{D}} V^+_C$ are compatible in that the intersections $V_{C,\tilde{C}} = V^+_C \cap V^-_{\tilde{C}}$ are closed submanifolds of $M$ with corners, all having the same dimension as $M$ and meeting along submanifolds of their common boundaries, and they give a decomposition
$M = \bigcup_{C \in \mathcal{D}} \mathcal{N}_{C,{\mathbb{j}}(C)} \,\,\,\, \cup \,\,\, \bigcup_{C,\tilde{C} \in \mathcal{D}, f(\tilde{C}) < f(C)}  V_{C,\tilde{C}}.$
Each submanifold with corners $V_C^+$ retracts via the gradient flow onto the corresponding Morse neighbourhood $\mathcal{N}_{C,\mathbb{j}(C)}$, inducing isomorphisms of relative homology groups
$$H_k(V^+_{C}, \partial_- V_{C}^+;\FF) \cong H_k(\mathcal{N}_{C,\infty}, \partial_- \mathcal{N}_{C,\infty};\FF)$$
where $\partial_- V^+_{C} = \partial_- \mathcal{N}_{C,\infty}$ is the outflowing boundary of $V^+_C$ for  $C\in \mathcal{D}$; corresponding statements hold for $V_C^-$.
We will call such decompositions $\mathcal{V} = \{V_{C}^+: C \in \mathcal{D}\}$ of $M$ {\it Morse decompositions} for $f:M \to \RR$; a local picture of a Morse decomposition near one connected component $C$ of the critical locus is given in Figure \ref{fig:cover}. 
\end{rem}

\section{Double complexes and spectral sequences} \label{sec:spectral}

In this section, given a smooth function $f:M \to \RR$ on a compact Riemannian manifold $M$  whose critical locus $\text{Crit}(f)$ has finitely many connected components, 
we will obtain further proofs of the Morse inequalities (Theorem \ref{mainthm}) from spectral sequences, one by 
using a Morse cover of $M$, as defined in $\S$3, to obtain a double complex with a filtration and associated spectral sequence abutting to the homology of $M$. 

Recall that a (homological) spectral sequence $E$ of bigraded vector spaces over $\FF$ starting at $r_0 \in \NN$ is given by three sequences:\\
(i) for all integers $r\geqslant r_0$, a bigraded  vector space $E_r = \bigoplus_{p,q \in \ZZ} E^r_{p,q}$, called the $r$th  page of the spectral sequence,\\
(ii)  linear maps $d_r : E_r \to E_r$ of bidegree $(-r,r - 1)$ satisfying $d_r \circ d_r = 0$, called boundary maps or differentials, and \\
(iii) identifications of $E_{r+1}$ with the homology of $E_r$ with respect to $d_r$.\\

Recall also that a 
 double complex (respectively a first quadrant double complex) $N$ over $\FF$ is given by a bigraded  vector space $N = \bigoplus_{p,q \in \ZZ} N_{p,q}$ (respectively $N = \bigoplus_{p,q \in \NN} N_{p,q}$) over $\FF$
with two differentials $\partial':N \to N$ of bidegree $(-1,0)$ and $\partial'\!':N \to N$ of bidegree $(0,-1)$ satisfying $(\partial')^2 = 0 = (\partial'\!')^2$ and $\partial' \partial'\!' + \partial'\!' \partial' = 0$  (cf. \cite{Boardman, Hurtubise,
Wall}). Then $\partial = \partial' + \partial'\!'$ satisfies $\partial^2 = 0$, and the total complex $\text{Tot}(N)$ of $N$ is given by
$$\text{Tot}(N)_k = \bigoplus_{p+q = k} N_{p,q}$$
with differential $\partial$. A first quadrant double complex is the special case when $\partial_j = 0$ for $j \geqslant 2$ of a
first quadrant multicomplex $N$ over $\FF$ which is given by a bigraded  vector space  $N = \bigoplus_{p,q \in \NN} N_{p,q}$
with linear maps $\partial_i: N \to N$ of bidegree $(-i,i-1)$ for $i \geqslant 0$ satisfying $$ \sum_{i+j = n} \partial_i \partial_j = 0$$
for all $n$. Then $\partial = \sum_{j \geqslant 0} \partial_j$ satisfies $\partial^2 = 0$, and the total complex $\text{Tot}(N)$ of $N$ is given by
$$\text{Tot}(N)_k = \bigoplus_{p+q = k} N_{p,q}$$
with differential $\partial$. 

Let  $C= \bigoplus_{k \in \NN} C_k$ be a complex with differential $d:C \to C$ of degree $-1$ and a filtration $F$ of $C$ by subcomplexes
$$0 = F_0  C \subseteq F_1 C \subseteq \cdots \subseteq F_L C = C.$$
Then each $C_k$ has a filtration  $0 = F_0  C_k \subseteq F_1 C_k \subseteq \cdots \subseteq F_L C_k = C_k$ where $F_\ell C_k = F_\ell C \cap C_k$ for $0 \leqslant \ell \leqslant L$, and there is also an induced filtration on the homology $H(C) = \bigoplus_{k \in \NN} H_k(C)$ of $C$, where $F_\ell H_k(C)$ is the image of the $k$th homology of $F_\ell C$ under the map  induced by the inclusion of $F_\ell C$ in $C$. Recall (from for example \cite{Maclane} Ch 16, Thm 5.4) that $F$ determines a spectral sequence $E_r = \bigoplus_{p,q \in \NN} E^r_{p,q}$ with isomorphisms
$$E^1_{p,q} \cong H_{p+q}(F_p C/F_{p-1}C)$$
which abuts to the homology of $C$ in the sense that there is some $r_\infty \geqslant r_0$ such that the differentials $d_r$ are all zero when $r \geqslant r_\infty$, giving natural isomorphisms $E_{r_\infty} \cong E_{r_\infty +1} \cong \cdots \cong E_\infty$ with
$$ E^\infty_{p,q} \cong F_p(H_{p+q}C)/F_{p-1}(H_{p+q}C).$$

\begin{defn} \label{defnbees}
Let $\text{Critval}(f) = \{c_1,\ldots,c_\rho\}$ where $c_1 < c_2 < \cdots < c_\rho$. Let $b_k = (c_k + c_{k+1})/2$  for $1 \leq k < \rho$ with $b_0=c_1 -1$ and $b_\rho = c_\rho +1$, and let
$$A_k = f^{-1}(-\infty,b_k]$$
 for $0 \leq k \leq \rho$, so that $A_+0 = \emptyset$ and $A_\rho = M$.  
\end{defn}

The first proof of the Morse inequalities for $f$ (Theorem \ref{mainthm}) used the existence of an  isomorphism
$$ H_i(A_k,A_{k-1};\FF)  \cong \bigoplus_{C \in \mathcal{D}_{c_k}} H_i(N_{C,\infty},\partial_- N_{C,\infty};\FF)
$$
induced by the gradient flow of $f$ and excision. Using the filtration of the chain complex $C_*(M)$  of singular simplices on $M$ by the sub-complexes $C_*(A_k)$, we see that there is a spectral sequence $E^r_{p,q}$ abutting to $H_*(M;\FF)$ with $E^1_{p,q} = H_{p+q}(A_p, A_{p-1};\FF)$ and boundary map  taking $[\xi] \in H_{p+q}(A_p, A_{p-1};\FF)$, where $\xi$ is a chain  in $A_p$ with boundary in $A_{p-1}$, to $[\partial(\xi)] \in  H_{p+q-1}(A_{p-1}, A_{p-2};\FF)$.
The existence of this spectral sequence implies the Morse inequalities. 
We will see that there is a related spectral sequence arising from a double complex which also abuts to $H_*(M;\FF)$ and is easier to describe in terms of the spaces $H_i(N_{C,\infty},\partial_- N_{C,\infty};\FF)$.

Given a first quadrant double complex $N = \bigoplus_{p,q \in \NN} N_{p,q}$ as above, its total complex $\text{Tot}(N)$ has two filtrations $F^{(1)}$ and $F^{(2)}$ defined by 
\begin{equation} \label{ffilt} F^{(1)}_p \text{Tot}(N)_k = \bigoplus_{h \leqslant p} N_{h,k-h}  \,\,\, \mbox{ and } \,\,\, F^{(2)}_p \text{Tot}(N)_k = \bigoplus_{h \leqslant p} N_{k-h,h}.\end{equation}
These give us the \lq first and second spectral sequences' $E(1)_r = \bigoplus_{p,q \in \NN} E(1)^r_{p,q}$ and $E(2)_r = \bigoplus_{p,q \in \NN} E(2)^r_{p,q}$ associated to the double complex, which both abut to the homology of the total complex $\text{Tot}(N)$. 
The 0th page of the first spectral sequence is given by 
$$E(1)^0_{p,q} =   \bigoplus_{h \leqslant p} N_{h,p+q-h}/\bigoplus_{h < p} N_{h,p+q-h} = N_{p,q}$$
where the differential is the map on the quotient induced by  $\partial = \partial' + \partial'\!'$ or equivalently by  $\partial'\!'$. Thus 
$$E(1)^1_{p,q} =  H'\!'_q( N_{p,*})$$
where $H'\!'$ denotes homology with respect to the differential $\partial'\!'$. Then the map induced by $\partial'\!'$ on $E(1)^1$ is zero, so the map induced by $\partial = \partial' + \partial'\!'$ is the same as that induced by  $\partial'$, and
the second page of the first spectral sequence is given by 
$$E(1)^2_{p,q} = H'_p H'\!'_q(N)$$
where $H'\!'(N)$ is the homology of $N$ with respect to the differential $\partial'\!'$ and $H' H'\!'(N)$ is the homology of $H'\!'(N)$ with respect to the differential induced by $\partial'$. The first few pages of the second spectral sequence have a similar description. 

We can also consider a situation when the first quadrant double complex $N = \bigoplus_{p,q 
} N_{p,q}$ 
 has a filtration $F$ by double complexes
$$0 = F_0  N \subseteq F_1 N \subseteq \cdots \subseteq F_L N = N.$$
Then the  total complex $\text{Tot}(N)$ has an induced filtration 
 given on $ \text{Tot}(N)_k = \bigoplus_{p+q = k} N_{p,q}$ by
$$ F_\ell \text{Tot}(N)_k = \text{Tot}(F_\ell N)_k =
 \bigoplus_{p+q = k} F_\ell N_{p,q}.$$
This filtration determines a spectral sequence $E_r = \bigoplus_{\ell, k} E^r_{\ell,k}$ abutting to the homology of $\text{Tot}(N)$ and satisfying
$$ 
E^1_{\ell,k} =  H_{\ell + k}  ( \bigoplus_{p+q=k} F_\ell N_{p,q} \,  / F_{\ell -1} N_{p,q})$$
where the homology is taken with respect to the differential induced by $\partial = \partial' + \partial'\!'$ on the quotient $F_\ell \text{Tot}(N)/F_{\ell -1} \text{Tot}(N)$.

We will apply this construction to the Mayer--Vietoris double complex associated to a Morse cover  $\mathcal{U} = \{ U_C: C \in \mathcal{D} \}$ of $M$ where $U_C = \mathcal{W}^+_C(\mathcal{N}^\circ_{C,\mathbb{j}(C)})$ as at Definition \ref{Ms}, or to a Morse decomposition as described at Remark \ref{decomp}.
Let $C_*(M)$ be the chain complex of singular simplices on $M$.
The Mayer--Vietoris double complex $N$ of a cover $\mathcal{U} = \{ U_C: C \in \mathcal{D} \}$ of $M$ has total complex $\text{Tot}(N)$  whose homology is isomorphic to  the homology of $M$,  and $N$  is given by
$$N_{p,q} = \bigoplus_{\sigma \in \mathrm{Nerve}_p(\mathcal{U})} C_q(\bigcap_{C \in \sigma}U_C)$$
where $\mathrm{Nerve}_p(\mathcal{U}) = \{ \sigma \subseteq \mathcal{D}: |\sigma| = p+1 \mbox{ and } \bigcap_{C \in \sigma}U_C \neq \emptyset \}$ defines the nerve of the cover $\mathcal{U}$
(see \cite{Brown} VII $\S$4, \cite{Stafa}). If $\sigma = \{C_0, \ldots,C_p\}$ where $C_0 < \cdots < C_p$ for some fixed total order on $\mathcal{D}$, the inclusions of
$\bigcap_{C \in \sigma}U_C$ in $\bigcap_{C \in \sigma \setminus \{C_i\}}U_C$ induce chain maps $\partial_{(i)}$ from $N_{p,q}$  to $N_{p-1,q}$ and we define the differential $\partial'$ on $N$ by
$\sum_{i=0}^p (-1)^i \partial_{(i)}$. The differential $\partial'\!'$ is given by the usual boundary map $N_{p,q} \to N_{p,q-1}$. 

\begin{rem} \label{rem:nerve}
Note that if we choose the Morse cover $\mathcal{U} = \{ U_C: C \in \mathcal{D} \}$ 
where $U_C = \mathcal{W}^+_C(\mathcal{N}^\circ_{C,\mathbb{j}(C)})$ as at Definition \ref{Ms}, using a construction of Morse neighbourhoods chosen as in
 Remark \ref{decomp} for $\mathbb{j}(\tilde{C}) >\!> \mathbb{j}(C) >\!> 1$ when $f(\tilde{C}) > f(C)$, then the nerve $\mathrm{Nerve}_p(\mathcal{U}) $ is controlled by a combinatorial invariant of the smooth function $f$ on the compact Riemannian manifold $M$. Indeed by Lemma \ref{nervequiver} below, if
$\bigcap_{C \in \sigma}U_C \neq \emptyset$ then
 $\sigma = \{C_0, \ldots,C_p\}$ where
$f(C_0) < \cdots < f(C_p)$ and there exist  \lq broken flow-lines' from $C_i$ to $C_{i-1}$; more precisely there are connected components
$$C_0 = C_{0,0}, C_{0,1},\ldots C_{0,m_1} = C_1 = C_{1,0}, \ldots C_{p-1,m_p} = C_p$$
of $\mathrm{Crit}(f)$ such that if $0\leqslant i < p$ and $0 < j \leqslant m_{i+1}$ then $f(C_{i,j-1}) < f(C_{i,j})$ and 
 the intersection
$W^-_{C_{i,j}} \cap W^+_{C_{i,j-1}} $  given by $$\{ x \in M : \{ \psi_{t}(x): t \leqslant 0\} \mbox{  has a limit point in $C_{i,j}$ and } \{ \psi_{t}(x): t \geqslant 0\} \mbox{  has a limit point in $C_{i,j-1}$} \}$$
 is nonempty and closed in $f^{-1}(f(C_{i,j-1}),f(C_{i,j}))$. Here as before $\psi_t(x)$  describes the downwards gradient flow for $f$ from $x$ at time $t$. 
This  can be expressed succinctly in terms of the quiver $\Gamma$ defined below:
whenever
$\sigma = \{C_0, \ldots,C_p\} \in \mathrm{Nerve}_p(\mathcal{U})$
there is a (directed) path in this quiver $\Gamma$ through all the vertices $v_{C_0},\ldots,v_{C_p}$.
 \end{rem}

\begin{defn} \label{6.1below} \label{defnquiver}
Let $\Gamma$ be the quiver with vertices $\{ v_C: C \in \mathcal{D}\}$ corresponding to the connected components $C$ of $\mathrm{Crit}(f)$, and with an arrow from $v_C$ to $v_{\tilde{C}}$ for each 
 connected component $A$ of 
$$W^-_C \cap W^+_{\tilde{C}} = \{ x \in M : \{ \psi_{t}(x): t\! \leqslant \!0\} \mbox{  has a limit point in $C$, $\{ \psi_{t}(x)\!: \!t \geqslant 0\}$ has a limit point in $\tilde{C}$} \}$$
whenever $W^-_C \cap W^+_{\tilde{C}}$ 
 is closed in $f^{-1}(f(\tilde{C}),f(C))$. 
\end{defn}

\begin{lemma}
If $\sigma = \{C_0, \ldots,C_p\}$ is a $p$-simplex in the nerve $\mathrm{Nerve}_p(\mathcal{U})$ of a Morse cover $\mathcal{U}$ chosen as in Remark \ref{rem:nerve}, then
there is a path in the quiver $\Gamma$ through all the vertices $v_{C_0},\ldots,v_{C_p}$.
\end{lemma}

\begin{proof}
We have $\mathcal{U} = \{ U_C: C \in \mathcal{D} \}$ 
where $U_C = \mathcal{W}^+(\mathcal{N}^\circ_{C,\mathbb{j}(C)})$ and $\mathbb{j}(\tilde{C}) >\!> \mathbb{j}(C) >\!> 1$ when $f(\tilde{C}) > f(C)$.  

Suppose that $C, \tilde{C} \in \mathcal{D}$ and $f(\tilde{C}) = c_j > c_i = f(C)$. Then $U_C \cap U_{\tilde{C}} \neq \emptyset$ if and only if
$U_C \cap U_{\tilde{C}} \cap f^{-1}[c_j - \epsilon,c_j] \neq \emptyset$ for any $\epsilon > 0$, since $U_C \cap U_{\tilde{C}}$ is contained in the (upwards and downwards) flow of its intersection with $f^{-1}[c_j - \epsilon,c_j]$, and this implies that 
$$\overline{U_C} \cap \mathcal{W}^+(\mathcal{N}_{\tilde{C},\mathbb{j}(\tilde{C})}) \cap f^{-1}[c_j - \epsilon,c_j] \neq \emptyset.$$
If $\epsilon >0$ is sufficiently small then $$ \mathcal{W}^+(\mathcal{N}_{\tilde{C},\mathbb{j}(\tilde{C})}) \cap f^{-1}[c_j - \epsilon,c_j] =  \mathcal{N}_{\tilde{C},\mathbb{j}(\tilde{C})}  \cap f^{-1}[c_j - \epsilon,c_j] $$
is a closed submanifold with corners of $M$. Moreover these submanifolds are nested as $\epsilon \to 0$ and $\mathbb{j}(\tilde{C}) \to \infty$ for fixed $\mathbb{j}(C)$, and their intersection is $\tilde{C}$.
So if $U_C \cap U_{\tilde{C}} \neq \emptyset$ for $\mathbb{j}(\tilde{C}) >\!> \mathbb{j}(C) >\!> 1$ then 
$$\overline{ \mathcal{W}^+(\mathcal{N}_{{C},\mathbb{j}({C})})} \cap \tilde{C} \neq \emptyset,$$
or equivalently
$$\overline{ \mathcal{W}^+(\mathcal{N}_{{C},\mathbb{j}({C})})} \cap W^-_{\tilde{C}} \neq \emptyset.$$

Now suppose that $\overline{ \mathcal{W}^+(\mathcal{N}_{{C},\mathbb{j}({C})})} \cap W^-_{\tilde{C}} \neq \emptyset$ where $\mathbb{j}(\tilde{C}) >\!> \mathbb{j}(C) >\!> 1$ when $f(\tilde{C}) > f(C)$. We will prove by induction on $j-i$ where $f(\tilde{C}) = c_j \geqslant c_i = f(C)$ that there is then a path in the quiver $\Gamma$ from $v_{\tilde{C}}$ to $v_C$. If $j=i$ then $\tilde{C} = C$, and if $j-i=1$ then $W^+_C \cap W^-_{\tilde{C}}$ is nonempty and closed in $f^{-1}(f(\tilde{C}),f(C))$ so there is an arrow in $\Gamma$ from $v_{\tilde{C}}$ to $v_C$. 

So suppose that $f(\tilde{C}) = c_j > c_k > c_i = f(C)$. Consider 
$$ S_k = \overline{ \mathcal{W}^+\mathcal{N}_{{C},\mathbb{j}({C})}} \cap W^-_{\tilde{C}} \cap f^{-1}(c_k).$$
If $S_k$ meets some $\hat{C} \in  \mathcal{D}$ with $f(\hat{C}) = c_k$ then  there exists
$$x \in \overline{ \mathcal{W}^+\mathcal{N}_{{C},\mathbb{j}({C})}} \cap W^-_{\hat{C}} \cap f^{-1}(b_k)
$$
and 
$$ y \in W^-_{\tilde{C}} \cap W^-+{\hat{C}} \cap f^{-1}(b_{k+1}) \subseteq \overline{ \mathcal{W}^+\mathcal{N}_{{\hat{C}},\mathbb{j}({\hat{C}})}} \cap W^-_{\tilde{C}} \cap f^{-1}(b_{k+1}),
$$
so by induction there is a path in $\Gamma$ from $v_{\tilde{C}}$ to $v_C$.
If on the other hand $S_k$ meets no $\hat{C} \in  \mathcal{D}$ with $f(\hat{C}) = c_k$ then 
$$ \overline{ \mathcal{W}^+\mathcal{N}_{{C},\mathbb{j}({C})}} \cap W^-_{\tilde{C}} \cap f^{-1}[b_k,b_{k+1}] \cong (\overline{ \mathcal{W}^+\mathcal{N}_{{C},\mathbb{j}({C})}} \cap W^-_{\tilde{C}} \cap f^{-1}(b_k)) \times [b_k,b_{k+1}]$$
via the gradient flow, and if this is true for all $k$ such that $j>k>i$ then 
$$ \overline{ \mathcal{W}^+\mathcal{N}_{{C},\mathbb{j}({C})}} \cap W^-_{\tilde{C}} \cap f^{-1}(c_i,c_j) = { \mathcal{W}^+\mathcal{N}_{{C},\mathbb{j}({C})}} \cap W^-_{\tilde{C}} \cap f^{-1}(c_i,c_j) \cong ({ \mathcal{W}^+\mathcal{N}_{{C},\mathbb{j}({C})}} \cap W^-_{\tilde{C}} \cap f^{-1}(b)) \times (c_i,c_j) 
$$
for any $b \in (c_i,c_j)$ via the gradient flow. Moreover taking the intersection over $\mathbb{j}(C) \to \infty$ we find that 
$$W_C^+  \cap W^-_{\tilde{C}}  = {W}^+_C \cap W^-_{\tilde{C}} \cap f^{-1}(c_i,c_j) \cong (W^+_C \cap W^-_{\tilde{C}} \cap f^{-1}(b)) \times (c_i,c_j) 
$$
is closed in $f^{-1}(c_i,c_j)$, so there is an arrow in $\Gamma$ from $v_{\tilde{C}}$ to $v_C$.

Finally observe that if  $\sigma = \{C_0, \ldots,C_p\} \in \mathrm{Nerve}_p(\mathcal{U})$ then $f(C_i) \neq f(C_j)$ if $i \neq j$, so without loss of generality
$f(C_0) < \cdots < f(C_p)$. Since $U_{C_i} \cap U_{C_j} \neq \emptyset$ for all $i,j$, there must be a path in $\Gamma$ from $v_{C_i}$ to $v_{C_{i-1}}$ for $1 \leqslant i \leqslant p$ and the result follows.
\end{proof}

Consider the spectral sequence associated to the filtration of this Mayer--Vietoris double complex $N$ given by 
$$F_\ell N_{p,q} = \bigoplus_{\sigma \in \mathrm{Nerve}_p(\mathcal{U})} C_q(f^{-1}(-\infty, b_\ell ] \cap \bigcap_{C \in \sigma}U_C)$$
where $b_0,\ldots,b_{\rho}$ are defined as at Definition \ref{defnbees} and $U_C = \mathcal{W}^+_C(\mathcal{N}^\circ_{C,\mathbb{j}(C)})$ as in Remark \ref{rem:nerve}. We have

\begin{lemma} \label{rellemma} The relative homology
$H_k( f^{-1}(-\infty, b_\ell ] \cap \bigcap_{C \in \sigma}U_C, f^{-1}(-\infty, b_{\ell -1} ] \cap \bigcap_{C \in \sigma}U_C;\FF)$ is given by $H_k(\mathcal{N}_{C,\infty}, \partial_- \mathcal{N}_{C,\infty};\FF)$ if $\sigma = \{C\}$ where $C \in \mathcal{D}_{c_\ell}$, and is 0 otherwise.
\end{lemma}
\begin{proof}
By Remark \ref{rem:nerve} we can let $\sigma = \{C_0, \ldots,C_p\}$ where
$f(C_0) < \cdots < f(C_p)$. By assumption the Morse neighbourhoods $\mathcal{N}{C,\mathbb{j}(C)}$ are chosen as in Remark \ref{decomp} so that the open subsets $U_C = \mathcal{W}^+_C(\mathcal{N}^\circ_{C,\mathbb{j}(C)})$ and their intersections are the interiors of closed submanifolds of $M$ with corners, and so are the intersections of these with the intervals $[b_{\ell -1},b_\ell]$. If $c_\ell > f(C_j)$ for some $j \in \{0,\ldots,p\}$ then the downwards gradient flow of $f$ induces a retraction 
 $$f^{-1}[b_{\ell -1}, b_\ell ] \cap \bigcap_{C \in \sigma}U_C \to f^{-1}(b_{\ell -1}) \cap \bigcap_{C \in \sigma}U_C$$ 
 and so by excision $H_k( f^{-1}(-\infty, b_\ell ] \cap \bigcap_{C \in \sigma}U_C, f^{-1}(-\infty, b_{\ell -1} ] \cap \bigcap_{C \in \sigma}U_C;\FF
)$ is zero. If $c_\ell < f(C_j)$ for some $j \in \{0,\ldots,p\}$ then the upwards gradient flow of $f$ induces a retraction 
$$f^{-1}[b_{\ell -1}, b_\ell ] \cap \bigcap_{C \in \sigma}U_C \to f^{-1}(b_{\ell}) \cap \bigcap_{C \in \sigma}U_C$$ 
which extends to a retraction of closed submanifolds with corners 
$$f^{-1}[b_{\ell -1}, b_\ell ] \cap \overline{\bigcap_{C \in \sigma}U_C} \to f^{-1}(b_{\ell}) \cap \overline{\bigcap_{C \in \sigma}U_C}.$$
The boundary of $f^{-1}[b_{\ell -1}, b_\ell ] \cap \overline{\bigcap_{C \in \sigma}U_C}$ is the union $\partial_- \cup \partial_\perp \cup \partial_+$ of its intersection $\partial_-$ with $f^{-1}(b_{\ell -1})$, its 
intersection $\partial_+$ with $f^{-1}(b_{\ell})$ and the closure $\partial_\perp$ of its intersection  with $f^{-1}(b_{\ell -1},b_\ell)$, and by Alexander-Spanier duality 
$$ H_k( f^{-1}[b_{\ell - 1}, b_\ell ] \cap \overline{\bigcap_{C \in \sigma}U_C}, \partial_-) \cong
H_{\dim M - k} ( f^{-1}[b_{\ell - 1}, b_\ell ] \cap \overline{\bigcap_{C \in \sigma}U_C}, \partial_\perp \cup \partial_+)^*.$$
The retraction 
 $$f^{-1}[b_{\ell -1}, b_\ell ] \cap \overline{\bigcap_{C \in \sigma}U_C} \to f^{-1}(b_{\ell}) \cap \overline{\bigcap_{C \in \sigma}U_C} = \partial_+$$
 restricts to a retraction of $\partial_\perp \cup \partial_+$ to $\partial_+$, and so 
 $H_k( f^{-1}[b_{\ell - 1}, b_\ell ] \cap \overline{\bigcap_{C \in \sigma}U_C}, \partial_-) \cong 0.
$
Since $f^{-1}[b_{\ell - 1}, b_\ell ] \cap {\bigcap_{C \in \sigma}U_C}$ is the interior of the submanifold with corners $f^{-1}[b_{\ell - 1}, b_\ell ] \cap \overline{\bigcap_{C \in \sigma}U_C}$, it follows by excision that 
$$ H_k( f^{-1}(-\infty, b_\ell ] \cap \bigcap_{C \in \sigma}U_C, f^{-1}(-\infty, b_{\ell -1} ] \cap \bigcap_{C \in \sigma}U_C;\FF) \cong 0$$
in this case too.
 Therefore $H_k( f^{-1}(-\infty, b_\ell ] \cap \bigcap_{C \in \sigma}U_C, f^{-1}(-\infty, b_{\ell -1} ] \cap \bigcap_{C \in \sigma}U_C;\FF)$ is zero 
 unless $\sigma = \{C\}$ where $C \in \mathcal{D}_{c_\ell}$. In the latter case the gradient flow combined with excision gives the required result.
\end{proof}

\begin{corollary} \label{relcor}
The $E_1$ page of the spectral sequence associated to the filtration of the Mayer--Vietoris double complex $N$ for the Morse cover $\mathcal{U} = \{U_C: C \in \mathcal{D} \}$ given by 
$F_\ell N_{p,q} = \bigoplus_{\sigma \in \mathrm{Nerve}_p(\mathcal{U})} C_q(f^{-1}(-\infty, b_\ell ] \cap \bigcap_{C \in \sigma}U_C)$ is given by 
$$E^1_{k,\ell} = \bigoplus_{C \in \mathcal{D}_{c_\ell}} H_{k+\ell}(\mathcal{N}_{C,\infty}, \partial_- \mathcal{N}_{C,\infty};\FF).$$
\end{corollary} 
\begin{proof}
This follows from Lemma \ref{rellemma} using the filtrations $F^{(1)}(F_\ell N/F_{\ell -1} N)$ of the double complexes $F_\ell N/F_{\ell -1} N$ and the associated \lq first spectral sequences' as at (\ref{ffilt}).
\end{proof}

\begin{rem} \label{relrem}
We obtain the same results by replacing $f^{-1}(-\infty,b_\ell]$ with $\bigcap_{f(C) \leqslant b_\ell} \mathcal{W}^-_C(\mathcal{N}_{C,\mathbb{j}(C)})$ chosen as in Remark \ref{rem:nerve}.
\end{rem}

Since this spectral sequence abuts to the homology of $M$, we obtain yet another proof of the generalised Morse inequalities (Theorem \ref{mainthm}). In addition this approach allows us to try to refine this theorem by studying the maps in the spectral sequence.

Recall that the Mayer-Vietoris double complex of the Morse cover $\mathcal{U} = \{ U_C: C \in \mathcal{D} \}$
has differential $d = \partial' + \partial'\!'$ where $\partial'\!'$ is the usual boundary map $N_{p,q} \to N_{p,q-1}$, and $\partial':N_{p,q} \to N_{p-1,q}$ is defined by fixing some total order on $\mathcal{D}$ and defining 
 $$ \partial': 
 C_q( \bigcap_{C \in \sigma}U_C)  \to \bigoplus_{i=0}^p C_q( \bigcap_{C \in \sigma \setminus\{C_i\}}U_C)$$
 for $\sigma \in \mathrm{Nerve}_p(\mathcal{U})$  by
$\partial' = \sum_{i=0}^p (-1)^i \partial_{(i)}$ if $\sigma = \{C_0, \ldots,C_p\}$ where $C_0 < \cdots < C_p$ for this total order, and 
$\partial_{(i)}$ is the chain map induced by 
the inclusion of
$\bigcap_{C \in \sigma}U_C$ in $\bigcap_{C \in \sigma \setminus \{C_i\}}U_C$. The signs are chosen to ensure that $(\partial')^2 = 0 = \partial'\partial'\!' + \partial'\!' \partial$, so that  $d = \partial' + \partial'\!'$ is a differential on $N$; a different choice of total order only affects the resulting double complex up to isomorphism. 
In our situation we can use a different choice of signs obtained by using Remark \ref{rem:nerve} and ordering the elements of $\sigma = \{C_0, \ldots,C_p\}$ according to the values taken by $f$ on $C_0, \ldots,C_p$.

As before, let $C_*(M)$ be the chain complex of singular simplices on $M$ with coefficients in $\FF$, and let $\mathcal{U} = \{U_C: C \in \mathcal{D}\}$ be a Morse cover as above, where
$$U_C = \mathcal{W}^+\mathcal{N}^\circ_{C,\mathbb{j}(C)} \,\,\, \mbox{ if $C \in \mathcal{D}$}$$ for some $\mathbb{j}:\mathcal{D} \to \NN$ such that $\mathbb{j}(C) >\!> 1$ when $C \in \mathcal{D}$ and the Morse neighbourhoods $\mathcal{N}_{C,\mathbb{j}(C)}$ are chosen as in Remark \ref{decomp}. Let
$$C^{\mathcal{U}
}_* (M) = \{ \sum_i n_i \sigma_i \in C_*(M) : \sigma_i \mbox{ has image contained in some element of $\mathcal{U}$} \} . $$
By \cite{Hatcher} Prop 2.21, the inclusion 
$C^{\mathcal{U}}_* (M) \to C_*(M)$ is a chain homotopy equivalence. Moreover if $A \subseteq M$ then there is an induced chain homotopy equivalence
$$ \frac{C^{\mathcal{U}}_* (M) }{C^{\mathcal{U}}_* (A)} \hookrightarrow 
\frac{C_*(M)}{C_*(A)} 
$$
giving an isomorphism on relative homology $H_i^{\mathcal{U}}(M,A;\FF) \cong H_i(M,A; \FF)
$. Similarly we obtain isomorphisms on relative homology $H_i^{\mathcal{U}}(\mathcal{N}_{C,j},\partial_- \mathcal{N}_{C,j};\FF) \cong H_i(\mathcal{N}_{C,j},\partial_- \mathcal{N}_{C,j}; \FF)
$ for each $C \in \mathcal{D}$ and $j$,
giving
\begin{equation} \label{equationu} 
 H_i(\mathcal{N}_{C,j},\partial_- \mathcal{N}_{C,j};\FF) \,\,\,\,\cong \,\,\,\,
 \frac{ \{\xi \in C^{\mathcal{U}}_i (\mathcal{N}_{C,j}): \partial \xi \in C^{\mathcal{U}}_{i-1} (\partial_- \mathcal{N}_{C,j}) \}
}{ \{ \partial \alpha + \beta: \alpha \in C^{\mathcal{U}^{\mathbb{j}}(\mathcal{N})}_{i+1} (\mathcal{N}_{C,j})
, \beta \in C^{\mathcal{U}^{\mathbb{j}}(\mathcal{N})}_i (\partial_- \mathcal{N}_{C,j})
\} }.
\end{equation}

 We would like to define boundary maps $\partial_{i\ell}^k$ from the sum $\bigoplus_{\tilde{C} \in \mathcal{D}_{c_k}} H_i(\mathcal{N}_{\tilde{C},\infty},\partial_- \mathcal{N}_{\tilde{C},\infty};\FF)$
to the sum  $
 \bigoplus_{C \in \mathcal{D}_{c_\ell}} H_{i-1}(\mathcal{N}_{C,\infty},\partial_- \mathcal{N}_{C,\infty};\FF)$ 
to be zero  if $\ell \geqslant k$ and for $\ell < k$ to be given as follows. Suppose  
 $\xi \in C^{\mathcal{U}}_i (\mathcal{N}_{\tilde{C},\mathbb{j}(\tilde{C})})$ with $\partial \xi \in C^{\mathcal{U}}_{i-1} (\partial_- \mathcal{N}_{\tilde{C},\mathbb{j}(\tilde{C})})$. Then we can write
$$ \partial \xi = \sum_{\ell=1}^{k-1} \sum_{C \in \mathcal{D}_{c_{\ell}}}  \zeta_{C} $$
where the image of $\zeta_{C}$ for $C \in \mathcal{D}_{c_{\ell}}$     
lies in the intersection of $\partial_- \mathcal{N}_{\tilde{C}, \mathbb{j}(\tilde{C})}$ with $U_C$, and its boundary lies in the intersection of this with $U_{\hat{C}}$ for $\hat{C} \in \mathcal{D}_{c_{j}}$ with $j \neq \ell$ and $0<j<k$.
We would like to set
$$ \partial_{i\ell} ^k([\xi]) = 
\sum_{C \in \mathcal{D}_{c_{\ell}}}  [\zeta_{C}] ,
$$
 where $[\zeta_{C}] \in  H_{i-1}(\mathcal{N}_{C,\infty},\partial_- \mathcal{N}_{C,\infty};\FF)$ is represented by the element of $ H_{i-1}(\mathcal{N}_{C,\mathbb{j}(C)},\partial_- \mathcal{N}_{C,\mathbb{j}(C)};\FF)$ 
 given by flowing $\zeta_C$ from $\partial_- \mathcal{N}_{\tilde{C}, \mathbb{j}(\tilde{C})} \cap U_C 
$ until it meets $
\mathcal{N}_{C,\, j_\ell}$.

When $\ell = k-1$ this gives us a well defined map from 
 $\bigoplus_{\tilde{C} \in \mathcal{D}_{c_k}} H_i(\mathcal{N}_{\tilde{C},\infty},\partial_- \mathcal{N}_{\tilde{C},\infty};\FF)
$ to the sum 
$ \bigoplus_{C \in \mathcal{D}_{c_\ell}} H_{i-1}(\mathcal{N}_{C,\infty},\partial_- \mathcal{N}_{C,\infty};\FF)$. 
In general we do not get well defined maps $\partial_{i\ell}^k$ from
 $\bigoplus_{\tilde{C} \in \mathcal{D}_{c_k}} H_i(\mathcal{N}_{\tilde{C},\infty},\partial_- \mathcal{N}_{\tilde{C},\infty};\FF)
$ to
$ \bigoplus_{C \in \mathcal{D}_{c_\ell}} H_{i-1}(\mathcal{N}_{C,\infty},\partial_- \mathcal{N}_{C,\infty};\FF)$.
However, with appropriate signs, 
 for $k-\ell = r$ they do define the $E_r$-page of the spectral sequence with 
 $$E^1_{p, q} = \bigoplus_{C \in \mathcal{D}_q} H_{p+q}(\mathcal{N}_{C,\infty}, \partial_- \mathcal{N}_{C,\infty};\FF)$$
 associated to the Mayer--Vietoris double complex of the cover $\mathcal{U} = \{U_c:C \in \mathcal{D}\}$ filtered as above. 
 
 \begin{rem} \label{rem4.5}   We can describe the spectral sequence for the Morse cover $\mathcal{U} = \{ U_C: C \in \mathcal{D} \}$ by relating it to the corresponding spectral sequence for a Morse decomposition 
$\mathcal{V} = \{ V_C: C \in \mathcal{D} \}$
where $V_C = \mathcal{W^+\!N}_{C,
\mathbb{j}(C)} \setminus \bigcup_{\tilde{C} \in \mathcal{D}, f(\tilde{C}) > f(C)} 
 \mathcal{W^+\!N}_{\tilde{C},
\mathbb{j}(\tilde{C})}^\circ $, as described in Remark \ref{decomp}. Here the system of Morse neighbourhoods has been chosen so that each $V_C$ is a closed submanifold of $M$ with corners, meeting along submanifolds with corners of their common boundaries. So we can choose a triangulation $T$ of $M$ which is compatible with triangulations of the submanifolds $V_C$ and unions of intersections $V_{C_1} \cap \cdots \cap V_{C_\ell}$  and their boundaries and corners, as well as intersections of these with $f^{-1}(-\infty, b_\ell ] $ and $f^{-1}(b_\ell)$ for $0<\ell<\rho$. For suitable choices of the Morse decomposition the chain complex $C_*^T(M)$ defined by this triangulation $T$ is contained in the chain complex $C_*^\mathcal{U}(M)$ for the Morse cover $\mathcal{U}$, and the inclusions 
$C^T_*(M) \to C^{\mathcal{U}}_* (M) \to C_*(M)$ are chain homotopy equivalences, with induced chain homotopy equivalences
$$ \frac{C^{T}_* (M) }{C^{T}_* (A)} \to \frac{C^{\mathcal{U}}_* (M) }{C^{\mathcal{U}}_* (A)} \to 
\frac{C_*(M)}{C_*(A)} 
$$
for subsets $A$ of $M$ compatible with the triangulation $T$. There is a corresponding homotopy equivalence given by inclusion into the Mayer--Vietoris double complex of the Morse cover $\mathcal{U}$ of $M$ from the Mayer--Vietoris double complex of the Morse decomposition 
$\mathcal{V} = \{ V_C: C \in \mathcal{D} \}$ with respect to the triangulation $T$ given by 
$$N^{\mathcal{V},T}_{p,q} = \bigoplus_{\sigma \in \mathrm{Nerve}_p(\mathcal{U})} C^T_q(\bigcap_{C \in \sigma}V_C).$$

Note that (cf. Remark \ref{rem:nerve}) 
when $f(\tilde{C}) > f(C)$ then 
$V_{\tilde{C}}$ only meets  $V_C$ when there are sequences  $C_0 = C, C_1,\ldots, C_m = \tilde{C}$ in $\mathcal{D}$ and
$\ell_0 = \ell < \ell_1 < \cdots < \ell_m = k$
with $f(C_p) = \{c_{\ell_p} \}$  for $0\leqslant p \leqslant m$, such that if $0<p\leqslant m$ 
$$\{ x \in M : \{ \psi_{t}(x): t \leqslant 0\} \mbox{  has a limit point in $C_p$ and } \{ \psi_{t}(x): t \geqslant 0\} \mbox{  has a limit point in $C_{p-1}$} \}$$
is nonempty and closed in $f^{-1}(f(C_{p-1}),f(C_p))$; moreover if $\ell_{p-1} = \ell_p - 1$ 
 then this closedness condition is always satisfied. 
When this closedness condition is satisfied 
 then the map $\partial^{p}_{i{p-1}}$ described as follows 
 on $\xi \in C_i^T(V_{{C}_p})$ with $\partial \xi \in C_{i-1}^T(\partial_-V_{{C}_{p}})$
 induces 
 a well defined map from
  $ H_i(V_{C_p}, \partial_-V_{C_p};\FF) \cong H_i(\mathcal{N}_{C_p,\infty},\partial_- \mathcal{N}_{C_p,\infty};\FF) 
$ to 
$ H_{i-1}(\mathcal{N}_{C_{p-1},\infty},\partial_- \mathcal{N}_{C_{p-1},\infty};\FF)$. Here we  write
$$ \partial \xi = \sum_{\ell=1}^{k-1} \sum_{C \in \mathcal{D}_{c_{\ell}}}  \zeta_{C} $$
where  for each $C \in \mathcal{D}_{c_{\ell}}$     we have
 $\zeta_{C} \in C_{i-1}^T(\partial_-V_{{C}_{p}} \cap V_C)$ and
 $$\partial \zeta_{C} \in C_{i-1}^T(\partial_-V_{{C}_{p}} \cap V_C \cap \bigcup_{\hat{C} \neq C} V_{\hat{C}} ).$$
 
 We set
$$ \partial_{i\ell} ^k([\xi]) = 
\sum_{C \in \mathcal{D}_{c_{\ell}}}  [\zeta_{C}]
$$
 where $[\zeta_{C}] \in  H_{i-1}(\mathcal{N}_{C,\infty},\partial_- \mathcal{N}_{C,\infty};\FF)$ is represented by the element of $ H_{i-1}(\mathcal{N}_{C,\mathbb{j}(C)},\partial_- \mathcal{N}_{C,\mathbb{j}(C)};\FF)$ 
 given by flowing $\zeta_C$ from $\partial_- \mathcal{N}_{\tilde{C}, \mathbb{j}(\tilde{C})} \cap U_C 
$ until it meets $
\mathcal{N}_{C,\, j_\ell}$.

Note also that such a connected component must be contained in a connected component of the Morse stratum  $W^-_{C_p}$ for $C_p$   (as defined at Definition \ref{defnwc}) and a  connected component of the Morse stratum $W^+_{C_{p-1}}$ for $C_{p-1}$. Thus we can decompose the maps $ \partial_{i\ell} ^k$ according to these connected components
(cf. Definition \ref{defnquiverrefined} below).  
 \end{rem}

 \begin{rem} In the Morse--Smale situation when $C$ is an isolated critical point we have 
 $$H_j(\mathcal{N}_{C,\infty}, \partial_- \mathcal{N}_{C,\infty};\FF) =  \FF$$ if $j$ is the index of $C$, and is zero otherwise. Moreover $W^-_C \cap W^+_{\tilde{C}} = \emptyset$ unless $\mathrm{index}(C) > \mathrm{index}(\tilde{C})$. This means that the spectral sequence only has nonzero maps between critical points with indices differing by one. 
 \end{rem}

 \begin{rem} \label{remcomplexMorse} When $M$ is oriented 
 the components $H_i(\mathcal{N}_{\tilde{C},\infty},\partial_- \mathcal{N}_{\tilde{C},\infty};\FF)
\to  H_{i-1}(\mathcal{N}_{C,\infty},\partial_- \mathcal{N}_{C,\infty};\FF)$ of the maps
 $\partial_{i\ell}^k : \bigoplus_{\tilde{C} \in \mathcal{D}_{c_k}} H_i(\mathcal{N}_{\tilde{C},\infty},\partial_- \mathcal{N}_{\tilde{C},\infty};\FF)
\to 
 \bigoplus_{C \in \mathcal{D}_{c_\ell}} H_{i-1}(\mathcal{N}_{C,\infty},\partial_- \mathcal{N}_{C,\infty};\FF)$ (when well defined) can be described using the isomorphism 
 $$(H_{i}(\mathcal{N}_{C,\infty},\partial_- \mathcal{N}_{C,\infty};\FF))^* \cong 
H_{\dim M - i}(\mathcal{N}_{C,\infty},\partial_+ \mathcal{N}_{C,\infty};\FF) $$
given by the intersection pairing between chains with boundaries in $\partial_\pm \mathcal{N}_{C,n}$ meeting transversally (and therefore not meeting on $\partial \mathcal{N}_{C,n}$). From this viewpoint the component mapping $H_i(\mathcal{N}_{\tilde{C},\infty},\partial_- \mathcal{N}_{\tilde{C},\infty};\FF)
 $ to  $H_{i-1}(\mathcal{N}_{C,\infty},\partial_- \mathcal{N}_{C,\infty};\FF)$ is given by the bilinear pairing
$$H_i(\mathcal{N}_{\tilde{C},\infty},\partial_- \mathcal{N}_{\tilde{C},\infty};\FF)
\otimes  H_{\dim M - i+1}(\mathcal{N}_{C,\infty},\partial_- \mathcal{N}_{C,\infty};\FF) \to \FF$$
which takes a pair $(\xi,\eta)$ such that $\xi \in C_i (\mathcal{N}_{\tilde{C},\mathbb{j}(\tilde{C})})$ with $\partial \xi \in C_{i-1} (\partial_- \mathcal{N}_{\tilde{C},\mathbb{j}(\tilde{C})})$
and  $\eta \in C_{\dim M -i +1} (\mathcal{N}_{{C},{j}_\ell})$ with $\partial \eta \in C_{\dim M - i} (\partial_- \mathcal{N}_{{C},{j}_\ell})$, transports the boundary $\partial \xi$ of $\xi$ under the gradient flow until it meets $
\mathcal{N}_{C,\, j_\ell + 1}$ and takes the intersection pairing of this with $\eta \in C_{\dim M -i +1} (\mathcal{N}_{{C},{j}_\ell})$. 
Equivalently this is given by an intersection pairing in a submanifold with corners of a moduli space of unparametrised flows. 
 When 
$$\{ x \in M : \{ \psi_{t}(x): t \leqslant 0\} \mbox{  has a limit point in $\tilde{C}$ and } \{ \psi_{t}(x): t \geqslant 0\} \mbox{  has a limit point in $C$} \}$$
 is closed in $f^{-1}(f(C),f(\tilde{C}))$ this intersection pairing is well defined, 
  but in general it is only well defined on a suitable page of the spectral sequence. In the Morse--Smale situation, when the vector spaces concerned are nonzero then this subset  is always closed in $f^{-1}(f(C),f(\tilde{C}))$, and the intersection pairing counts flow lines in the usual way between critical points whose indices differ by one. Thus in this case the information encoded in the spectral sequence is essentially the same as in the Morse--Witten complex. However in general we have a more complicated picture, which can be described in terms of multicomplexes. 
 \end{rem}
 
 \section{Multicomplexes supported on acyclic quivers} \label{sec:quivers}
 
 In this section we will modify the usual definition of a multicomplex given in $\S$4 (cf. for example \cite{LWZ}) to define multicomplexes supported on acyclic quivers (Definition \ref{defngamma} below). These will be used to generalise the Morse-Witten complex to the situation where $f:M \to \RR$ is any smooth function on a compact Riemannian manifold whose critical locus has finitely many connected components.
 
 Let $\Gamma = (Q_0,Q_1, h, t)$ be an acyclic quiver (or equivalently a directed graph without oriented cycles). Here $Q_0$ and $Q_1$ are finite sets (of vertices and arrows respectively) and $h:Q_1 \to Q_0$ and $t:Q_1 \to Q_0$ are maps (determining the head and tail of any arrow). 
 
 \begin{defn}
 The vertex span $R = \FF^{Q_0}$ and arrow span $A=\FF^{Q_1}$ of $\Gamma$ are the vector spaces of $\FF$-valued functions on $Q_0$ and $Q_1$, with bases identified with $Q_0$ and $Q_1$ via Kronecker delta functions.
 $R$ is a commutative algebra over $\FF$ under pointwise multiplication of functions, with the basis elements in $Q_0$ as idempotents, 
while $A$ is an $R$-bimodule via
 $$(e\cdot f)(a) = e(ha) f(a) \,\,\, \mbox{ and } \,\,\, (f \cdot e)(a) = f(a) e(ta)$$
 for all $e \in R$, $f \in A$ and $a \in Q_1$. 
Any $R$-bimodule $M$ can be decomposed into sub-$R$-bimodules $$M=\bigoplus_{e,\tilde{e} \in Q_0} M_{e,\tilde{e}} \,\, \mbox{ where } \,\, M_{e,\tilde{e}} = eM\tilde{e}.$$
 The path algebra of $\Gamma$ is the graded algebra
 $$ \mathcal{A} = \bigoplus_{d \geqslant 0} A^d$$
 where $A^d = A \otimes_R A \otimes_R \cdots \otimes_R A$ and $A^0 = R$. Here
  $A^d$ has a basis given by the paths
 $$\{ a_1\ldots a_d: a_1, \ldots ,a_d \in Q_1 \mbox{ and } t(a_k)=h(a_{k+1}) \mbox{ for } 1 \leqslant k < d \}
 $$ of length $d$ in $\Gamma$, with multiplication given by concatenation where defined and 0 otherwise.
 Our assumption that $\Gamma$ is acyclic implies that $A^d=0$ when $d$ is large enough, and therefore $\dim_\FF \mathcal{A} < \infty$.  \end{defn}
 
Recall (from for example \cite{DWZ}) that a representation $\rho$ of $\Gamma$ over $\FF$ is given by vector spaces $V_w$ for $w \in Q_0$ and linear maps $\rho(a): V_{h(a)} \to V_{t(a)}$ for $a \in Q_1$. Any such representation $\rho$ induces a representation $V = \bigoplus_{w \in Q_0} V_w$ of the path algebra $\mathcal{A}$ with a path $a_1,\ldots a_d \in A^d$ acting as the composition of the linear maps $\rho(a_1), \ldots, \rho(a_d)$. Conversely any representation of the algebra $\mathcal{A}$ comes from a representation of $\Gamma$.  Let
 $$ e_\Gamma = \sum_{e \in Q_0} e \in R$$
 so that $e_\Gamma b = b = be_\Gamma $ for all $b \in \mathcal{A}$. If $V = \bigoplus_{e \in Q_0} V_e$ is a representation of $\Gamma$ then a linear subspace $W$ of $V$ is a subrepresentation of $\Gamma$ if and only if $e_\Gamma W= W$.
 
\begin{rem}  \label{remdg} If we also assume that $\Gamma$ has no multiple arrows between the same two vertices, then we can make $\mathcal{A}$ into a differential graded algebra where the differential $\delta a$ of an arrow from $v$ to $w$ is the sum of all the paths of length 2 from $v$ to $w$, and $\delta$ is extended to paths via the Leibniz rule (cf. \cite{GLMY}).
Let
 $$D_\Gamma = \sum_{a \in Q_1} a \,\, \in \mathcal{A}.$$
 Then $\delta D_\Gamma = (D_\Gamma)^2$ is given by the sum of all paths of length 2 in $\Gamma$; indeed $(D_\Gamma)^r$ is given by the sum of all paths of length r in $\Gamma$ for any $r\geqslant 1$ and $\delta((D_\Gamma)^r)$ is 0 if $r$ is even and $(D_\Gamma)^{r+1}$ if $r$ is odd.  If we let $\Gamma^r$ be the quiver with set of vertices $Q_0$ and arrows given by paths of length $r$ in $\Gamma$, then its path algebra is $ \bigoplus_{d \geqslant 0} A^{rd}$ and 
 $D_{\Gamma^r} = (D_\Gamma)^r$.
\end{rem}

\begin{defn}
We will say that a subset $Q_0'$ of $Q_0$ defines a final subquiver 
(respectively an initial subquiver) $$\Gamma' = (Q_0',Q_1', h|_{Q_1'},t|_{Q_1'}) \,\, \mbox{ where } \,\, Q_1' = h^{-1}(Q_0') \cap t^{-1}(Q_0')$$
 of $\Gamma$ if it satisfies the condition that $v \in Q_0'$ whenever there exists $w \in Q_0'$ and $a \in Q_1$ with $t(a) = w$ and $h(a) = v$ (respectively $t(a) = v$ and $h(a) = w$).
\end{defn}

Note that if $\Gamma_1$ and $\Gamma_2$ are final subquivers of $\Gamma$ then so are $\Gamma_1 \cap \Gamma_2$ and $\Gamma_1 \cup \Gamma_2$; the same is true for initial subquivers.

\begin{exit}
If $j \geqslant 0$ then the subset 
 of $Q_0$ consisting of those $v \in Q_0$ such that every path from $v$ (respectively to $v$) in $\Gamma$ has length at most $j$ defines a final subquiver (respectively an initial subquiver) of $\Gamma$.\end{exit}

\begin{exit}  \label{rem5.5o}
If $d \in \RR$ and $f$ is any real-valued function on $Q_0$ such that $f(v) > f(w)$ whenever there is an arrow in $\Gamma$ from $v$ to $w$,  then  the subset 
 of $Q_0$ consisting of those $v \in Q_0$ such that $f(v) < d$ (respectively $f(v) > d$) defines a final subquiver (respectively an initial subquiver) of $\Gamma$.
\end{exit}

\begin{defn} \label{defnrquiver}
An $\RR$-quiver $(\Gamma,f)$ is given by a quiver $\Gamma= (Q_0,Q_1,h,t)$ and a function $f:Q_) \to \RR$ which grades $\Gamma$ in the sense that if $a \in Q_1$ is an arrow in $\Gamma$ then $f(h(a)) > f(t(a))$ (cf.\cite{HKKP} $\S$2.1).
\end{defn}

\begin{rem}
If $(\Gamma,f)$ is an $\RR$-quiver then the quiver $\Gamma$ is acyclic.
\end{rem}

\begin{rem}
Example \ref{rem5.5o} can be rephrased to say that if $(\Gamma,f)$ is an $\RR$-quiver and $d \in \RR$ then $\{ v \in Q_0: f(v) < d\}$ defines a final subquiver of $\Gamma$. Conversely if $\Gamma$ is an acyclic quiver and has no multiple arrows and if $\Gamma'$ is a final subquiver of $\Gamma$, then there is some $d \in \RR$ and $f:Q_0 \to \RR$ such that $(\Gamma,f)$ is an $\RR$-quiver and $\Gamma'$ is the subquiver defined by $\{ v \in Q_0: f(v) < d\}$.
\end{rem}

\begin{exit}
If $v_0 \in Q_0$ then the subset $Q_0^{[v_0>]}$ of $Q_0$ consisting of those $v \in Q_0$ such that there is a path from $v_0$ to $v$
 in $\Gamma$ defines a final subquiver $\Gamma^{[v_0>]}$ of $\Gamma$, and $Q_0^{[v_0>]} \cup \{ v_0 \}$  defines a final subquiver $\Gamma^{[v_0 \geqslant ]}$. More generally for any $r \in \NN$ the subset $Q_0^{[v_0\geqslant (r)]}=Q_0^{[v_0>(r-1)]}$ of $Q_0$ consisting of those $v \in Q_0$ such that there is a path from $v_0$ to $v$ of length at least $r$
 in $\Gamma$ defines a final subquiver $\Gamma^{[v_0\geqslant (r)]}=\Gamma^{[v_0>(r-1)]}$ of $\Gamma$. Similarly  the subset $Q_0^{[\leqslant (r) v_0]}$ of $Q_0$ consisting of those $v \in Q_0$ such that there is a path from $v$ to $v_0$ of length at most $r$
 in $\Gamma$ defines an initial subquiver $\Gamma^{[\leqslant (r)v_0]}$ of $\Gamma$. 
\end{exit}

\begin{defn}
Let $V$ be a vector space over $\FF$. A {\it final filtration} $\mathcal{F}$ of $V$ over the quiver $\Gamma$ is given by a subspace $V_{\Gamma'}^\mathcal{F}$ of $V$ for every final subquiver $\Gamma'$ of $\Gamma$ such that
$$V_\emptyset^\mathcal{F} = \{ 0 \} \,\,\, \mbox{ and } \,\,\, V_\Gamma^\mathcal{F} =  V,$$
while
$$ {\Gamma_1} \subseteq{\Gamma_2}  \,\,\, \mbox{ implies } \,\,\,    V_{\Gamma_1}^\mathcal{F} \subseteq  V_{\Gamma_2}^\mathcal{F} $$
for all final subquivers $\Gamma_1$ and $\Gamma_2$ of $\Gamma$. 
The final filtration is {\it strict} if in addition 
$$ V_{\Gamma_1 \cap \Gamma_2}^\mathcal{F} =  V_{\Gamma_1}^\mathcal{F} \cap V_{\Gamma_2}^\mathcal{F}  \,\,\, \mbox{ and } \,\,\,  V_{\Gamma_1 \cup \Gamma_2}^\mathcal{F} =  V_{\Gamma_1}^\mathcal{F} +  V_{\Gamma_2}^\mathcal{F} $$
for all final subquivers $\Gamma_1$ and $\Gamma_2$ of $\Gamma$. The associated $\Gamma$-graded vector space defined by the filtration is
$$\text{gr}_\mathcal{F}(V) = \bigoplus_{v \in Q_0} V_{\Gamma^{[v\geqslant]}}^\mathcal{F}/V_{\Gamma^{[v> ]}}^\mathcal{F}.$$
When the final filtration is strict then $\text{gr}_\mathcal{F}(V)$ is isomorphic as a vector space to $V$.
\end{defn}

\begin{rem}
Initial filtrations can also be defined similarly.
\end{rem}

\begin{rem} \label{filt}
A representation $V = \bigoplus_{w \in Q_0} V_w$ of the quiver $\Gamma$  (or equivalently of its path algebra $\mathcal{A}$) 
has a natural filtration over $\Gamma$, and the associated $\Gamma$-graded vector space can be thought of as the representation of $\Gamma$ given by $V$ but with all arrows represented by zero maps.
\end{rem}

 \begin{defn} \label{defnfiltcomplex}
 A chain complex $C = \oplus_{n \in \ZZ} C_n$ with differential $d:C \to C$ with $d(C_n) \subseteq C_{n-1}$ is {\it filtered over $\Gamma$} if it is equipped with a filtration $\mathcal{F}$ of $C$ such that,
for every final subquiver $\Gamma'$ of $\Gamma$, the differential $d$ maps the subspace $C_{\Gamma'}^\mathcal{F}$ of $C$ into itself.

A {\it complex supported on $\Gamma$} is a representation $\rho$ of $\Gamma$ such that $(\rho(D_{\Gamma}))^2 = 0$, where $D_{\Gamma}$ is as in Remark \ref{remdg}.
 \end{defn}

 \begin{defn} \label{defngamma}
 A {\it multicomplex $N$ supported on the quiver $\Gamma = (Q_0,Q_1,h,t)$} is given by representations $\rho_r$ of $\Gamma^r$ (defined as in Remark \ref{remdg}) for $r\geqslant 1$ with the same vector space  $N_w$ associated to each vertex $w$ of $\Gamma^r$ independent of $r$, 
 such that $\sum_{r\geqslant 1} \rho_r(D_{\Gamma^r})$ is a differential on $\sum_{w \in Q_0} N_w$. Here we can allow the quiver $\Gamma$ to have infinitely many vertices provided that the vector space $N_w$ is nonzero for only finitely many vertices $w$ of $\Gamma$. The multicomplex will be called homogeneous of degree $k$ if $\rho_r (D_{\Gamma^r}) = 0$ unless $r=k$.
 We will call $N$ a {\it strict multicomplex supported on $\Gamma$} if
 $$ \sum_{m+n=j} \rho_m(D_{\Gamma^m})\rho_n(D_{\Gamma^n}) = 0$$
 for every $j > 0$. 
 
  A {\it level 1 graded multicomplex $N$ supported on $\Gamma$} is given by 
    a multicomplex supported on the quiver whose set of vertices is $Q_0 \times \ZZ$ and whose set of arrows is $Q_1 \times \ZZ$ with $(a,n) \in Q_1 \times \ZZ$ representing an arrow to $(h(a),n-1)$ from $(t(a),n)$. Thus $N = \bigoplus_{w \in Q_0} \bigoplus_{n \in \ZZ} N_{w,n}$ has a differential
 $\sum_{r\geqslant 1} \rho_r(D_{\Gamma^r})$ where   $ \rho_r(D_{\Gamma^r})$ maps $N_{w,n}$ to the sum of the subspaces $N_{w',n-1}$ such that there is a path of length $r$ in $\Gamma$ from $w$ to $w'$. A level $j$ graded multicomplex supported on $\Gamma$ for $j>0$ is defined similarly.
  

 A {\it double-multicomplex $N$ supported on $\Gamma \times \ZZ$} is given by representations $\rho_{n,r}$ of $\Gamma^r$ for $n \in \ZZ$ and $r\geqslant 1$ such that $(\rho_{n,r})_{r \geqslant 1}$ for fixed $n \in \ZZ$ defines a multicomplex supported on $\Gamma$, together with a differential $d$ given by $d:\rho_{n,r} \to \rho_{n-1,r}$ satisfying $d\circ \rho_{n,r}(D_{\Gamma^r}) = - \rho_{n-1,r}(D_{\Gamma^r})  \circ d$, so that 
 $d + \sum_{n,r} \rho_{n,r}(D_{\Gamma^r})$ is a differential on $N = \bigoplus_{w \in Q_0} \bigoplus_{n \in \ZZ} N_{w,n}$.





 \end{defn}

Recall that the nerve $\mathrm{Nerve}(\mathcal{C})$ of a small category $\mathcal{C}$ is a simplicial set with 0-simplices the objects of the category, and for $k>0$ the set $\mathrm{Nerve}(\mathcal{C})_k$ of $k$-simplices consisting of $k$-tuples of composable morphisms $x_0 \to x_1 \to \cdots \to x_k$ in $\mathcal{C}$. The $i$th face map $\mathrm{Nerve}(\mathcal{C})_k \to \mathrm{Nerve}(\mathcal{C})_{k-1}$is given by composition of morphisms at the $i$th object when $0<i<k$ and removal of the $i$th object when $i=0$ or $k$, while the $i$th degeneracy map $s_i: \mathrm{Nerve}(\mathcal{C})_k \to \mathrm{Nerve}(\mathcal{C})_{k+1}$ inserts an extra morphism given by the identity at the $i$th object when $0 \leqslant i \leqslant k$. 

The nerve of an open cover $\mathcal{U} = \{U_i:i \in I\}$ of $M$ as considered in $\S$4 is then the nerve of the category $\mathcal{C}[\mathcal{U}]$ whose objects are nonempty finite intersections in $\mathcal{U}$; that is, they are given by finite $J \subseteq I$ such that $\bigcap_{i \in J} U_i \neq \emptyset$, and whose morphisms are inclusions.

An acyclic quiver $\Gamma = (Q_0,Q_1,h,t)$ generates a finite category $\mathcal{C}((\Gamma))$ with $\mathrm{Ob}(\mathcal{C}((\Gamma))) = Q_0$, morphisms given by paths in $\Gamma$ and composition given by concatenation. In fact we can make this into a strict 2-category such that there is a unique 2-morphism from a path $v_0\ldots v_p$ in $\Gamma$ to a path $w_0\ldots w_q$ in $\Gamma$ if there is a subset $\{j_0<j_1 < \cdots < j_p\}$ of $\{0,1,\cdots,q\}$ such that $v_i = w_{j_i}$ for $0 \leq i \leq p$. When $\Gamma$ is the quiver defined at Definition \ref{6.1below} using a Morse cover $\mathcal{U}$ of $M$ associated to a smooth function $f:M \to \RR$ which has a finite set $\mathcal{D}$ of connected components of $\mathrm{Crit}(f)$, then there is a surjection from $\mathrm{Nerve}(\mathcal{C}((\Gamma))$ to $\mathrm{Ob}(\mathcal{C}[\mathcal{U}]) = \mathrm{Nerve}(\mathcal{C}[\mathcal{U}])_0$ which takes a $k$-tuple of composable paths in $\Gamma$ to its set of endpoints.

Suppose that we have a double complex  $N = \bigoplus_{p,q \in \ZZ} N_{p,q}$  over $\FF$
with differentials $\partial':N \to N$ of bidegree $(-1,0)$ and $\partial'\!':N \to N$ of bidegree $(0,-1)$ satisfying the following two conditions with respect to the quiver $\Gamma$. Firstly we require
\begin{equation} \label{eqn:nerve} N_{p,q} = \bigoplus_{v_0\cdots v_p \in \Gamma_{(p)}} N_{v_0\cdots v_p,q} \end{equation}
where $\Gamma_{(p)}$ is the set of $(p+1)$-tuples $(v_0,\ldots, v_p)$ of vertices in $\Gamma$ such that there is a path in $\Gamma$ from $v_{i-1}$ to $v_i$ for $1 \leq i \leq p$; equivalently $\Gamma_{(p)} = \mathrm{Nerve}(\mathcal{C}((\Gamma)))_p/\sim$ where two $p$-tuples of composable paths in $\Gamma$ are $\sim$-equivalent if and only if the paths in the compositions have the same endpoints. Secondly we require that $\partial':N_{v_0\cdots v_p,q} \to N_{p-1,q}$ is given by $\sum_{i=0}^p (-1)^i\partial_{(i,v_0\cdots v_p,q)}$ where $\partial_{(i,v_0\cdots v_p,q)}:N_{v_0\cdots v_p,q}  \to N_{v_0\cdots \hat{v_i} \cdots v_p,q} $ satisfies \begin{equation} \label{eqn:nerve2}  
\partial_{(i,v_0\cdots \hat{v_j} \cdots v_p,q)} \circ \partial_{(j,v_0\cdots v_p,q)} = \partial_{(j,v_0\cdots \hat{v_i} \cdots v_p,q)} \circ \partial_{(i,v_0\cdots v_p,q)}\end{equation} if $0 \leq i < j \leq p$. 

\begin{exit} \label{exampleprelevelling}
We saw in $\S$4 that these conditions are satisfied by the double complex associated to a suitable Morse cover of a Riemannian manifold $M$ with smooth function $f:M\to \RR$ whose critical set $\mathrm{Crit}(f)$ has finitely many connected components.
\end{exit}

\begin{defn} \label{levelling}
In this situation we will call a function $\phi:\Gamma \times \ZZ \to \RR$ taking only finitely many distinct values $\phi_1< \ldots < \phi_L$ a {\it levelling function} for the double complex $N$ provided that

(i) $\phi(v,q-1) \leqslant \phi(v,q)$ for all vertices $v$ of $\Gamma$ and all $q \in \ZZ$;

(ii) when there is an arrow in $\Gamma$ from $w$ to $v$ then $\phi(v,q) < \phi(w,q)$;

(iii) if we filter the double complex by 
$$ \mathcal{F}_\ell N = \bigoplus_{v_0\cdots v_p \in \Gamma_{(p)}, q \in \ZZ: \phi(v_p,q) \leqslant \phi_\ell} :N_{v_0\cdots v_p,q} ,$$
then for any $v_0\cdots v_p \in \Gamma_{(p)}$
the relative homology 
 of the complex $(\bigoplus_{q: \phi(v_0,q) \leqslant  \phi_\ell} N_{v_0\cdots v_p,\cdot }, \partial'\!')$ and its subcomplex $(\bigoplus_{q: \phi(v_0,q) < \phi_\ell} N_{v_0\cdots v_p,\cdot }, \partial'\!')$
is zero unless $p=0$.
\end{defn}

\begin{rem}
Conditions (i) and (ii) tell us that $\phi$ makes the quiver $(Q_0 \times \ZZ, Q_1 \times \ZZ)$ associated to $\Gamma$ as in Definition \ref{defngamma} into an $\RR$-quiver in the sense of Definition \ref{defnrquiver}.
\end{rem} 

\begin{lemma}  Let $\phi:\Gamma \times \ZZ \to \RR$ be a {levelling function} for the double complex $N$.
The spectral sequence associated to the filtration $ \mathcal{F}_\ell N$ of $N$ has $(\ell,k)$-term in its $E_1$ page given by
$$ \bigoplus_{v \in Q_0: \phi(v,k+\ell) = \phi_{\ell} }  
H_{k+\ell} (N_{v,\cdot};
\partial'\!' ),$$
and it abuts to the homology of the total complex $N$.
\end{lemma}

\begin{proof}
We follow the proof of Corollary \ref{relcor}, using the \lq first spectral sequence' associated to the filtration $F^{(1)}(  \mathcal{F}_\ell N/ \mathcal{F}_{\ell -1} N)$ of the double complex $ \mathcal{F}_\ell N/ \mathcal{F}_{\ell -1} N$ defined as at (\ref{ffilt}). By Definition \ref{levelling} (ii) the $(p,q)$th term in the $E_1$ page of this spectral sequence is
$$\bigoplus_{v_0\cdots v_p \in \Gamma_{(p)}:\phi(v_0,p+q)= \phi_\ell} H_{p+q}(N_{v_0\cdots v_p,\cdot};\partial'\!' ),$$
and by (iii) this is zero unless $p=0$ in which case it is 
$$ \bigoplus_{v_0 \in Q_0: \phi(v_0,q) = \phi_{\ell} }  
H_{q} (N_{v_0,\cdot};
\partial'\!' ).$$
The result follows.
\end{proof}

\begin{rem} \label{e1}
The $E_1$ page 
$$ \bigoplus_{k,\ell} \bigoplus_{v \in Q_0: \phi(v,k+\ell) = \phi_{\ell} }  
H_{k+\ell} (N_{v,\cdot};
\partial'\!' )
=\bigoplus_{v \in Q_0, n \in \ZZ}  
H_{n} (N_{v,\cdot}; \partial'\!' )
$$
of the spectral sequence associated to the filtration $ \mathcal{F}_\ell N$ of $N$ becomes a level 1 graded multicomplex supported on $\Gamma$ in the sense of Definition \ref{defngamma} in a unique way determined by the requirement that the differential should be $\sum_{r\geqslant 1} \rho_r(D_{\Gamma^r})$ where   $ \rho_r(D_{\Gamma^r})$ maps $H_{n} (N_{w,\cdot}; \partial'\!' )$ to the sum of the subspaces $H_{n-1} (N_{w',\cdot}; \partial'\!' )$ such that there is a path of length $r$ in $\Gamma$ from $w$ to $w'$.
\end{rem}

\begin{exit}  \label{examplelevelling}
In the situation of Example \ref{exampleprelevelling} we can take $\phi(v_C,q) = f(C)$ and we recover the spectral sequence described in $\S$4; we will call this choice of levelling function the Morse levelling function. More generally for any $\phi$ we find that 
arrows from $C_k$ to $C_\ell$ are represented by the maps $\partial^k_{i\ell}$ as in Remark \ref{rem4.5}. 

When $f:M \to \RR$ is Morse--Smale we can choose $\phi$ so that $\phi(v_C,q) = q$ when $C$ is a critical point of index $q$ and recover the Morse--Witten complex; more precisely, the spectral sequence degenerates after the $E_1$-page and the latter can be identified with the Morse--Witten complex.

For any pair $(f,g)$ where $g$ is a Riemannian metric on $M$ and $f:M \to \RR$ is smooth with finitely many critical components, we get an $\RR$-quiver $\Gamma$  and a final filtration of the homology of $M$ over $\Gamma$, with spectral sequences  abutting to $H_*(M;\FF)$ whose $E_1$ pages are level 1 graded multicomplexes over $\Gamma$.
\end{exit}

\begin{rem} \label{splitquiver}
Suppose that we have another acyclic quiver $\tilde{\Gamma} = (\tilde{Q}_0,\tilde{Q}_1,\tilde{h},\tilde{t})$ (again with no multiple arrows) which  is a perturbation $\eta: \tilde{\Gamma} \to \Gamma$ of the quiver $\Gamma$, in the sense that each vertex of $\Gamma$ has split into finitely many vertices of $\tilde{\Gamma}$, and each path in $\Gamma$ has split into finitely many paths in $\tilde{\Gamma}$ in a compatible way. More precisely we require firstly a surjection $\eta_0: \tilde{Q}_0 \to Q_0$, which allows us to identify the vertex span $R$ for $\Gamma$ with the subalgebra of the vertex span $\tilde{R}$ for $\tilde{\Gamma}$ spanned by the idempotents
$$  \sum_{ w \in \eta_0^{-1}(v)}  w \,\,\,\,\,\, \,\,\, \mbox{ for } v \in Q_0,$$
and then $e_\Gamma = e_{\tilde{\Gamma}}$ and the path algebra $\tilde{\mathcal{A}}$ for $\tilde{\Gamma}$ becomes an $R$-bimodule. Secondly there should be a surjection
$$\eta_1: \{ \mbox{paths in } \tilde{\Gamma} \} \to  \{ \mbox{paths in } {\Gamma} \} $$ 
such that $\eta_0( \tilde{h}(a_d)) = h ( \eta_1 (a_1 \ldots a_d))$ and $\eta_0( \tilde{t}(a_1)) = t (\eta_1 (a_1 \ldots a_d))$ for every path $ a_1 \ldots a_d$ in $\tilde{\Gamma}$, inducing a surjection $\tilde{\mathcal{A}} \to \mathcal{A}$ which respects the multiplication and  $R$-bimodule structure but not necessarily the grading given by path-length. 

Now suppose that $N = \bigoplus_{w \in \tilde{Q}_0} 
N_{w}$ is a multicomplex supported over $\tilde{\Gamma}$. Then $N$ can be regarded as a multicomplex supported on $\Gamma$ by writing 
$$ N = \bigoplus_{v \in Q_0}\,\, \left(   \bigoplus_{w \in \eta_0^{-1}(v)}  
N_w \right). $$

\end{rem}

\section{Morse--Smale perturbations } \label{sec:multi}

As before let $f:M \to \RR$ be a smooth function on a compact Riemannian manifold  whose critical locus $\text{Crit}(f)$ has finitely many connected components, and let $\mathcal{D}$ be the set of connected components of $\text{Crit}(f)$. In this section we will study the spectral sequences introduced in $\S$5 with $E_1$ pages given by multicomplexes supported on quivers 
(which refine the information given by the vector spaces $H_i(\mathcal{N}_{C,\infty},\partial_- \mathcal{N}_{C,\infty};\FF)$ and in the Morse--Smale situation reduce to the Morse--Witten complex), by relating them to the Morse--Witten complexes of Morse--Smale perturbations of $f:M\to \RR$.
 
 When $(f,g)$ is Morse--Smale then $\text{Crit}(f) = \mathcal{D}$ and if $C \in \mathcal{D}$ then $H_i(\mathcal{N}_{C,\infty},\partial_- \mathcal{N}_{C,\infty};\FF)$ is one-dimensional if $i$ is the index of the critical point in $C$ and otherwise is 0. 
Recall (from for example \cite{CohenNorbury}) that in this situation   there is an associated graph $\Gamma$ (more precisely a multi-digraph or quiver) embedded in $M$, with vertices given by the critical points of $f$ and arrows joining a critical point $p$ of index $i$ to a critical point $q$ of index $i-1$ given by the (finitely many) gradient flow lines from $p$ to $q$. The Morse--Witten complex is the differential module with basis $\text{Crit}(f)$ and differential 
$$\partial p = \sum_{q \in \text{Crit}(f)} n(p,q) q $$
where $n(p,q)$ is the number of flow lines from $p$ to $q$, counted with signs determined (canonically) by suitable choices of local orientations. 
It is usually regarded as a complex via the grading by index (which is essentially the homological degree), but 
when we grade it using the homological degree and the critical value, then we have a multicomplex $N = \bigoplus_{p,q \in \NN} N_{p,q}$ with 
$N_{p,q} = \bigoplus_{C \in \mathcal{D}_q} H_{p+q}(\mathcal{N}_{C,\infty}, \partial_- \mathcal{N}_{C,\infty};\FF)$
which is supported on $\Gamma$, in the sense that 
the total differential $\partial$ is the sum of linear maps from $H_{p+q}(\mathcal{N}_{C,\infty}, \partial_- \mathcal{N}_{C,\infty};\FF)$ to $H_{k+\ell}(\mathcal{N}_{C',\infty}, \partial_- \mathcal{N}_{C',\infty};\FF)$ for $C \in \mathcal{D}_q$ and $C' \in \mathcal{D}_\ell$
 of bidegree $(-i,i-1)$, one for each arrow in $\Gamma$ from $C$ to $C'$. The spectral sequence described at the end of $\S4$ contains the information in this (multi)complex, but it is packaged more efficiently in the Morse--Witten complex, which appears as the $E_1$ page of an alternative spectral sequence defined in $\S$5 (see Example \ref{examplelevelling}).
 
In the Morse--Smale situation a flow line from a critical point $p$ of index $i$ to a critical point $q$ of index $i-1$ is a connected component of 
$$\{ x \in M : \{ \psi_{t}(x): t \leqslant 0\} \mbox{  has a limit point in $\{p\}$ and } \{ \psi_{t}(x): t \geqslant 0\} \mbox{  has a limit point in $\{q\}$} \},$$
which is always closed in $f^{-1}(f(q),f(p))$ (where $\psi_t(x)$  describes the downwards gradient flow for $f$ from $x$ at time $t$).
In our more general situation recall that we associated to the smooth function $f:M \to \RR$ on the Riemannian manifold $M$ a quiver $\Gamma$ (see Definition \ref{6.1below}) which has  vertices $\{v_C:C \in \mathcal{D}\}$ labelled by the connected components $C$ of $\text{Crit}(f)$, and an arrow from $v_{C_+}$ to $v_{C_-}$ for each connected component of 
$$\{ x \in M : \{ \psi_{t}(x): t \leqslant 0\} \mbox{  has a limit point in $C_+$ and } \{ \psi_{t}(x): t \geqslant 0\} \mbox{  has a limit point in $C_-$} \}$$
whenever this subset is closed in $f^{-1}(f(C_-),f(C_+))$. Moreover given a \lq levelling function' $\phi$ as at Definition \ref{levelling} we obtain a spectral sequence abutting to the homology of $M$ whose $E_1$ page is a  a level 1 graded multicomplex supported on $\Gamma$ in the sense of Definition \ref{defngamma} (see Remark \ref{e1}).
One way to study these spectral sequences is to use the multicomplexes given by the Morse--Witten complexes for  Morse--Smale perturbations of the smooth function $f:M \to \RR$. In order for this to work we need to describe the relative homology $H_*(\mathcal{N}_{C,\infty},\partial_- \mathcal{N}_{C,\infty};\FF)$ in Morse--Witten terms. 

There are many descriptions of the relative (co)homology of a manifold with boundary which are appropriate for Morse theory \cite{Akaho,BNR,
CDGM, ChangLiu,MM,Schwarz, Shonkwiler} and many of them can be adapted to apply to manifolds with corners, whose boundaries decompose as the Morse neighbourhoods do, to give a description of the homology of the manifold relative to part of the boundary. We will use the approach of Laudenbach \cite{Laudenbach} and will avoid having to deal with corners by choosing a system of smooth Morse neighbourhoods (Remark \ref{modify}).

Laudenbach considers a compact manifold $M$ with non-empty boundary, and defines a smooth real-valued function $f$ on $M$ to be Morse when its critical points lie in the interior of $M$ and are nondegenerate, and its restriction to the boundary is a Morse function in the usual sense; being Morse in this sense is generic among smooth functions on $M$. Then there are two types of critical points of the restriction of $f$ to $\partial M$: type N (for Neumann) when the gradient flow is pointing into $M$ and type D (for Dirichlet) when the gradient flow is pointing out of $M$. The homotopy type (and homology) of $\{x \in M:f(x) \leqslant a\}$ may change when $a$ crosses a critical value of $f$ on the interior of $M$ or the value of a critical point of type N of $f |_{\partial M}$ but not  when $a$ crosses  the value of a critical point of type D of $f |_{\partial M}$. Similarly the homotopy type of $M$ relative to its boundary (and the relative homology $H_*(M,\partial M;\FF)$ may change when $a$ crosses a critical value of $f$ on the interior of $M$ or the value of a critical point of type D of $f |_{\partial M}$ but not  when $a$ crosses  the value of a critical point of type N of $f |_{\partial M}$. 

\begin{defn}
Let $f:M \to \RR$ be a Morse function in Laudenbach's sense on a compact manifold $M$ with boundary. Let

$C_k$ be the set of critical points of $f$ on the interior of $M$ with index $k$;

$N_k$ be the set of critical points of $f$ on $\partial M$ of type N and index $k$;

$D_k$ be the set of critical points of $f$ on $\partial M$ of type D and index $k-1$.
\end{defn}

Laudenbach proves

\begin{thm} 
(i) There is a differential on the free graded $\ZZ$-module generated by $C_* \cup N_*$ such that it is a chain complex whose homology is isomorphic to the homology of $M$. 

(ii) There is a  co-differential on the free graded $\ZZ$-module generated by $C_* \cup D_*$ such that it is a cochain complex whose cohomology is isomorphic to the relative cohomology $H^*(M,\partial M;\ZZ^{\text{or}})$ with coefficients twisted by the local system of orientations on $M$. 
\end{thm}

His proof of (i) involves using a pseudo-gradient vector field $X$ for $f$ which is adapted to the boundary in the sense that 

a) $X.f < 0$ except at the critical points in the interior of $M$ and the critical points of type $N$ on $\partial M$;

b) $X$ points inwards along the boundary except near the critical points of type N on $\partial M$ where it is tangent to the boundary;

c) if $p$ is a critical point in the interior of $M$, then $X$ is hyperbolic at $p$ and the quadratic form $X^{\text{lin}}.d^2_p f$ (where $X^{\text{lin}}$ is the linear part of $X$ at $p$) is negative definite;

d) near any critical point $p$ of type N for $f |_{\partial M}$ there are coordinates $x = (y,z) \in \RR^{n-1} \times \RR$ on $M$ such that $M$ is given locally by $z \geqslant 0$ and $f$ is given locally by $f(x) = f(p) + z + q(y)$ where $q$ is a nondegenerate quadratic form and $X$ is tangent to the boundary, vanishing and hyperbolic at $p$ and $X^{\text{lin}}.(q(y) + z^2)$ is negative definite;

e) $X$ is Morse--Smale, meaning that the global unstable manifolds and local stable manifolds are mutually transverse.

It is shown in \cite{Laudenbach} that a pseudo-gradient $X$ for $f$ adapted to the boundary always exists, that there is a differential on the free graded $\ZZ$-module generated by $C_* \cup N_*$ given for $p \in C_k \cup N_k$ by
\begin{equation} \label{Lcomplex}  \partial p = \sum_{q \in C_{k-1} \cup N_{k-1} } m_{pq} q \end{equation}
where $m_{pq}$ is the number of flow lines from $p$ to $q$, counted with appropriate signs, and that the homology of the resulting complex is isomorphic to $H_*(M;\FF)$. The last statement is proved by 

(i) showing that if $p$ is a critical point on $\partial M$ of type N, then $f$ and $X$ can be modified in an arbitrarily small neighbourhood $U$ of $p$ to a Morse function $f_1$ with the same Morse complex but with no critical point in $U \cap \partial M$ of type N and instead one  critical point in $U \cap \partial M$ of type D (with the same index as $p$ as a critical point on $\partial M$)
together with one interior critical point in $U \setminus \partial M$ (with the same index as a critical point on the interior of $M$);

(ii) considering the case when there are no critical points of type N. In this case the Laudenbach complex for the homology of $M$ does not \lq see' the boundary at all, in the sense that every flow line appearing in the differential connects critical points in the interior. Then standard Morse theory arguments show that it can be assumed without loss of generality that $f$ is weakly self-indexing, in the sense that for critical points $p$ and $q$ we have $f(p) > f(q)$ if and only if the index of $p$ is strictly greater than the index of $q$; this follows as in \cite{Laudenbach} $\S$2.4 from

\begin{lemma} Let  $(f, X)$ be a Morse function and an adapted pseudo-gradient. Let $p$ and $q$ be two critical points with $f(p) > f(q)$. Assume that the open interval $((f(q),f(p))$ contains no critical value and that there are no connecting orbits from $p$ to $q$. Then there exists a path of
Morse functions $f_t$ for $t\in [0,1]$, with $f_0 = f$ and  $f_1(p) < f_1(q) = f(q)$, such that $X$ is a pseudo-gradient for every $f_t$.
\end{lemma}

It is observed in \cite{Laudenbach} $\S$2.4 that, when $f$ is weakly self-indexing and there are no critical points of type N on the boundary, $M$ has the homotopy type of a CW-complex with a $k$-cell for each critical point of index $k$, and Laudenbach's   complex (\ref{Lcomplex}) is the corresponding cellular chain complex, whose homology is $H_*(M;\FF)$ by the cellular homology theorem.

Now let us return to the situation where $f:M \to \RR$ is any smooth function whose critical locus has finitely many components, on the compact Riemannian manifold $M$. We can apply Laudenbach's results to a perturbation of $f$ (or rather $-f$) restricted to suitably chosen Morse neighbourhoods.  As at Remark \ref{modify}, we can choose a system  of strict Morse neighbourhoods $\{ \mathcal{N}_{C,n}:C \in \mathcal{D}, n\geqslant 0\}$ such that $\partial  \mathcal{N}_{C,n} = \partial_+  \mathcal{N}_{C,n} \cup \partial_-  \mathcal{N}_{C,n}$ is smooth, $f|_{\partial  \mathcal{N}_{C,n}}$ is a Morse function and has no critical points on $\partial_+  \mathcal{N}_{C,n} \cap \partial_-  \mathcal{N}_{C,n}$ and $f|_{\partial_+ \mathcal{N}_{C,n}}$ (respectively $f|_{\partial_-  \mathcal{N}_{C,n}}$) achieves its minimum value (respectively its maximum value) precisely on $\partial_+  \mathcal{N}_{C,n} \cap \partial_-  \mathcal{N}_{C,n}$. Then we can perturb $f$ in the interior of $\mathcal{N}_{C,n}$ so that it is a Morse function on $\mathcal{N}_{C,n}$ in Laudenbach's sense, and moreover it is weakly self-indexing with every critical value strictly greater than the value of $f$ at every critical point of its restriction to
$\partial_-  \mathcal{N}_{C,n}$ and strictly less than the value of $f$ at every critical point of its restriction to
$\partial_+  \mathcal{N}_{C,n}$. The critical points of type N for $f$ (respectively type D for $-f$) on $\partial  \mathcal{N}_{C,n}$ are precisely those in $\partial_+  \mathcal{N}_{C,n}$ while the critical points of type D for $f$ (respectively type N for $-f$) are those in $\partial_-  \mathcal{N}_{C,n}$.

We will say that a pseudo-gradient vector field $X$ for $f$ on $\mathcal{N}_{C,n}$ is doubly adapted to the boundary if

a) $X.f < 0$ except at the critical points in the interior of $\mathcal{N}_{C,n}$ and the critical points for the restriction of $f$ to $\partial \mathcal{N}_{C,n}$;

b) $X$ points inwards along $\partial_-  \mathcal{N}_{C,n}$ except near the critical points  on $\partial_-  \mathcal{N}_{C,n}$ where it is tangent to the boundary, and $X$ points outwards along $\partial_+  \mathcal{N}_{C,n}$ except near the critical points  on $\partial_+  \mathcal{N}_{C,n}$ where it is tangent to the boundary;

c) if $p$ is a critical point in the interior of $\mathcal{N}_{C,n}$, then $X$ is hyperbolic at $p$ and the quadratic form $X^{\text{lin}}.d^2_p f$ (where $X^{\text{lin}}$ is the linear part of $X$ at $p$) is negative definite;

d) near any critical point $p$ on $\partial_-  \mathcal{N}_{C,n}$ (respectively on $\partial_+  \mathcal{N}_{C,n}$)
 there are coordinates $x = (y,z) \in \RR^{n-1} \times \RR$ on $M$ such that $M$ is given locally by $z \geqslant 0$ (respectively $z \leqslant 0$) and $f$ is given locally by $f(x) = f(p) + z + q(y)$ where $q$ is a nondegenerate quadratic form and $X$ is tangent to the boundary, vanishing and hyperbolic at $p$ and $X^{\text{lin}}.(q(y) + z^2)$ is negative 
 definite;

e) $X$ is Morse--Smale.

By Laudenbach's argument,
there is a doubly adapted pseudo-gradient  on the Morse neighbourhood $\mathcal{N}_{C,n}$  and a  differential on the free graded $\ZZ$-module generated by $C_* \cup N_* \cup D_{*}$, which is given for $p \in C_k \cup N_k \cup D_{k}$
 by
\begin{equation} \label{Lcomplex2}  \partial p = \sum_{q \in C_{k-1} \cup N_{k-1} \cup D_{k-1} } m_{pq} q \end{equation}
where $m_{pq}$ is the number of flow lines from $p$ to $q$, counted with appropriate signs, and that the homology of the resulting complex is isomorphic to $H_*(\mathcal{N}_{C,n}, \partial_- \mathcal{N}_{C,n};\FF)$. 

Now instead of restricting the smooth map $f:M \to \RR$ to the Morse neighbourhoods $\{ \mathcal{N}_{C,\mathbb{j}(C)}: C \in \mathcal{D}\}$ for some $\mathbb{j}: \mathcal{D} \to \NN$, let us consider $f$ on $M$ itself, but perturb it as above in the interior of each $ \mathcal{N}_{C,\mathbb{j}(C)}$ to become a Morse function.

\begin{figure}[t]
\includegraphics[width = 0.4\linewidth]{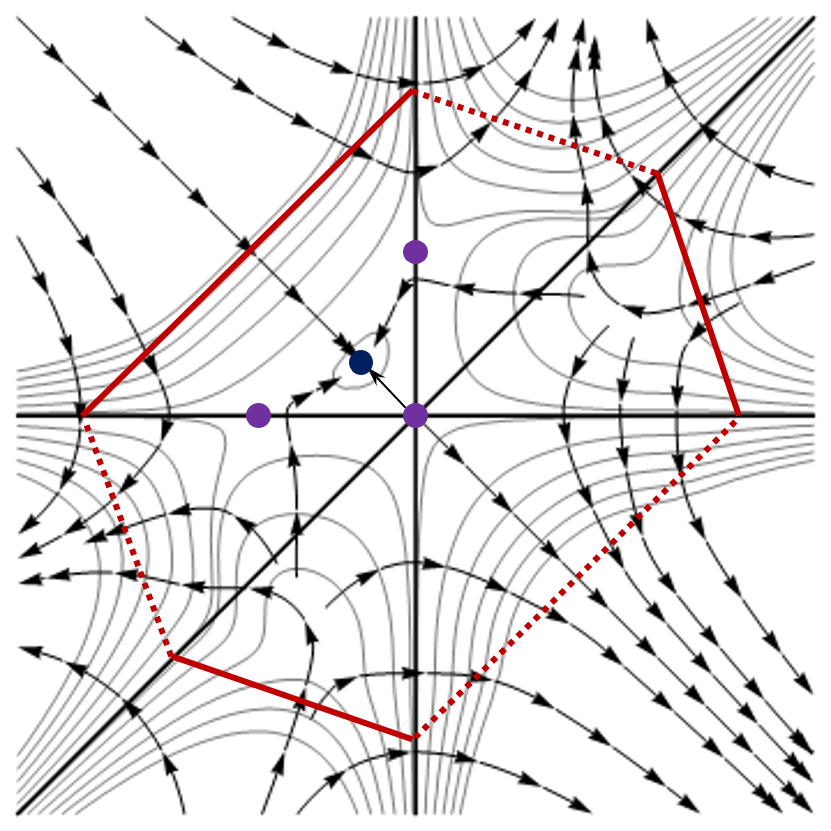}
\centering
\caption{A Morse-Smale perturbation of $f(x,y) = x y\, (x-y)$ near 0.}
\label{fig:perturb}
\end{figure}

Using a pseudo-gradient doubly adapted to the boundaries of the Morse neighbourhoods and given by the gradient flow of the perturbed function elsewhere, we obtain a differential determined by flow lines on the graded vector space with basis given (in the obvious modification of the notation above) by
$$\bigcup_{C \in \mathcal{D}} C^{[C]}_*  \cup N^{[C]}_* \cup D^{[C]}_{*}  ,$$
with homology $H_*(M;\FF)$ and the structure of a multicomplex supported on a quiver $\tilde{\Gamma}$. Here the choice of pseudo-gradient allows (broken) flow lines to cross $\partial \mathcal{N}_{C,\mathbb{j}(C)}$ only at critical points for the restriction of $f$ to $\partial \mathcal{N}_{C,\mathbb{j}(C)}$.

The quiver $\tilde{\Gamma}$ associated to the perturbed function $\tilde{f}:M \to \RR$  is a perturbation $\tilde{\Gamma} \to \Gamma$ of the quiver $\Gamma$ associated as at Definition \ref{defnquiver} to the original smooth function $f:M \to \RR$, in the sense that each vertex of $\Gamma$ has split into finitely many vertices of $\tilde{\Gamma}$, and each path in $\Gamma$ has split into finitely many paths in $\tilde{\Gamma}$ in a compatible way. Indeed we can factor $\tilde{\Gamma} \to \Gamma$ as $\tilde{\Gamma} \to \Gamma' \to \Gamma$ where the quiver $\Gamma'$ is defined as follows (cf. \cite{Bloom}).

\begin{defn} \label{defnquiverrefined}
The refined quiver $\Gamma'$ associated to the smooth function $f:M \to \RR$ on the Riemannian manifold $M$ has  vertices  of three types: a vertex $v_C$ for each connected component $C$ of the critical locus $\text{Crit}(f)$ of $f$, a vertex $u_{G,C}$ for every connected component $G$ of the Morse stratum $W^+_C$ for $C$ (as defined at Definition \ref{defnwc}) and a vertex $w_{H,C}$ for every connected component $H$ of the Morse stratum $W^-_C$ for $C$. The arrows are also of three types: 

(i) there is  an arrow from $u_{G,C}$ to $v_C$ for every $C \in \mathcal{D}$ and  every connected component $G$ of the Morse stratum $W^+_C$ for $C$,

(ii)   there is  an arrow from $v_C$ to $w_{H,C}$  for every $C \in \mathcal{D}$ and  every connected component $H$ of the Morse stratum $W^-_C$ for $C$, and

(iii)  there is  an arrow from $w_{H,C_+}$ to $u_{G,C_-}$ for every $C_+$ and $C_-$  in $\mathcal{D}$ such that there is a connected component of 
$$\{ x \in M : \{ \psi_{t}(x): t \leqslant 0\} \mbox{  has a limit point in $C_+$ and } \{ \psi_{t}(x): t \geqslant 0\} \mbox{  has a limit point in $C_-$} \}$$
that is closed in $f^{-1}(f(C_-),f(C_+))$ and is contained in $G \cap H$.
\end{defn}

The quivers $\Gamma$ and $\Gamma'$ only depend on  the smooth function $f:M \to \RR$ on the Riemannian manifold $M$, but the quiver $\tilde{\Gamma}$ depends on the perturbation of the function $f$ and the gradient vector field.
Similarly the multicomplex supported on $\tilde{\Gamma}$ with underlying vector space $\FF \langle \bigcup_{C \in \mathcal{D}} C^{[C]}_* \cup N^{[C]}_* \cup \tilde{D}^{[C]}_*
\rangle$ (bigraded by critical value and index/homological degree) depends on these perturbations. However as in Remark \ref{splitquiver} we can relate the spectral sequences associated to the Morse levelling functions (see Example \ref{examplelevelling}) for $f:M\to \RR$ and the Morse--Smale perturbation by a suitable quasi-isomorphism respecting the map from $\tilde{\Gamma}$ to $\Gamma$, and combine this with studying the effect on the spectral sequence of modifying the levelling functions in each case.

\section{Examples} \label{sec:examples}

Let us consider some examples of smooth functions $f:M\to \RR$ with finitely many components of $\mathrm{Crit}(f)$ on compact surfaces equipped with Riemannian metrics. Recall that the Morse inequalities in Theorem \ref{mainthm} depend only on the function $f$ and not on the Riemannian metric on $M$, whereas, as we shall see, the quiver $\Gamma$ (see Definition \ref{6.1below}) and  refined quiver $\Gamma'$ (see Definition \ref{defnquiverrefined}) associated to $f:M \to \RR$ may depend on the Riemannian structure.
In addition  the spectral sequences abutting  to the homology $H_*(M;\FF)$ of $M$
whose $E_1$ pages are level 1 graded multicomplexes supported on $\Gamma$ (in the sense of Definition \ref{defngamma}) depend on the Riemannian structure and the function as well as on the choice of a levelling function (see Definition \ref{levelling}) such as the one provided by the Morse function itself.

 First let us consider the height function $f:M \to \RR$ given by projection to the $z$-axis on a torus $M$ embedded in $\RR^3$ in the standard way, as the rotation about the $y$-axis of a circle in the $(x,y)$-plane. This is a Morse function with a maximum, a minimum and two saddle points. If we use the standard Riemannian metric on $\RR^3$ then the corresponding quiver $\Gamma_3$ is as pictured in Figure \ref{fig:standardtorus}. For a generic perturbation $g$ of this metric the pair $(f,g)$ is Morse--Smale and the quiver is given by $\Gamma_1$, while $\Gamma_2$ is the quiver associated to the intermediate case: a perturbation of the standard Riemannian metric which is not generic.

\begin{figure}[t]
\includegraphics[width = 0.6\linewidth]{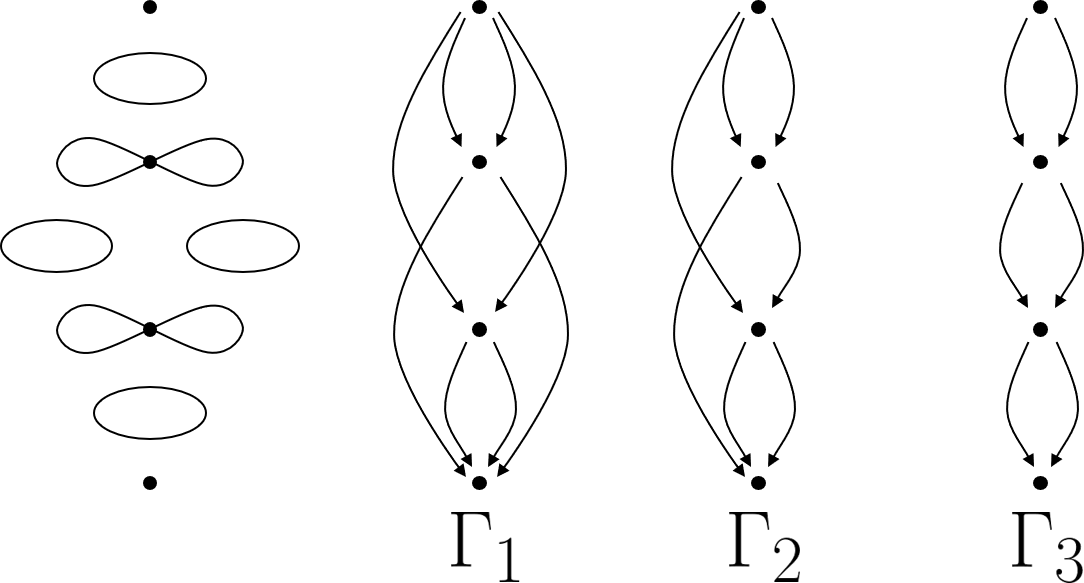}
\centering
\caption{Level sets for the height function on a standard torus in $\RR^3$ and associated quivers using three different Riemannian metrics}
\label{fig:standardtorus}
\end{figure}

The corresponding refined quivers $\Gamma'_1, \Gamma'_2, \Gamma'_3$, defined as at Definition \ref{defnquiverrefined}, are pictured in Figure \ref{fig:standardrefined}.

Note that the local behaviour of $f$ near its two saddle points ensures that their Morse neighbourhoods are topologically squares and that the quiver has two incoming and two outgoing arrows at each saddle point. Moreover if there exist flow lines between critical points with successive critical values $\alpha$ and $\beta$ then they must form a closed subset of $f^{-1}(\alpha,\beta)$, and for each critical point there must exist paths in the quiver to a local minimum and  from a local maximum. Thus the quivers $\Gamma_1, \Gamma_2, \Gamma_3$ and their refinements $\Gamma'_1, \Gamma'_2, \Gamma'_3$ are the only possibilities for this function. In each case, and for each choice of levelling function $\phi:\Gamma \times \ZZ \to \RR$ satisfying the conditions of Definition \ref{levelling}, we have
$$H_i(\mathcal{N}_{c_j,n}, \partial_\pm \mathcal{N}_{c_j,n};\FF) \cong \FF$$
if $(i,j) = (0,1)$ or $(1,2)$ or $(1,3)$ or $(2,4)$ and is zero otherwise, and thus the differential in the $E_1$ page of the associated spectral sequence is zero: this is a perfect Morse function in the sense that the Morse inequalities are actually equalities.

\begin{figure}[t]
\includegraphics[width = 0.4\linewidth]{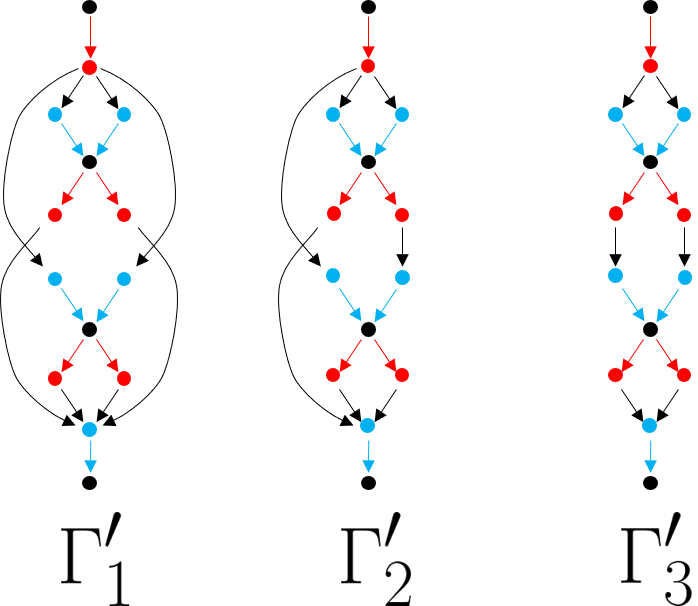}
\centering
\caption{Refined quivers for the height function on a standard torus in $\RR^3$ using three different Riemannian metrics}
\label{fig:standardrefined}
\end{figure}

\medskip

 We can also consider the height function given by projection to the $z$-axis of a circle in the $(y,z)$-plane. This is a Morse--Bott function $f:M \to \RR$ whose critical locus $\mathrm{Crit}(f)$ has two connected components, $C_{\min}$ and $C_{\max}$, each a circle, with $f$ attaining its maximum value on $C_{\max}$ and its minimum value on $C_{\min}$. For any choice of Riemannian metric their Morse neighbourhoods are annuli and 
 $$ P_t(\mathcal{N}_{C_{\max},n}, \partial_\pm \mathcal{N}_{C_{\max},n}) = tP_t(S^1) = t + t^2$$
 while
 $$ P_t(\mathcal{N}_{C_{\min},n}, \partial_\pm \mathcal{N}_{C_{\min},n}) =  P_t(\mathcal{N}_{C_{\min},n}) = P_t(S^1) = 1+t.$$
 The subset
$$\{ x \in M : \{ \psi_{t}(x): t \leqslant 0\} \mbox{  has a limit point in $C_{\max}$ and } \{ \psi_{t}(x): t \geqslant 0\} \mbox{  has a limit point in $C_{\min}$} \}$$
 is closed in $f^{-1}(f(C_{\min}),f(C_{\max}))$ (indeed it is the whole of this set) and has two connected components, so the quiver $\Gamma$ and refined quiver $\Gamma'$ are as pictured in Figure \ref{fig:secondtorus}.

 Any choice of levelling function $\phi:\Gamma \times \ZZ \to \RR$ gives the same spectral sequence as the one defined by the critical values, which has zero differential on the $E_1$ page: this is another perfect Morse function. 
 
 \medskip

 Our next example is the height function on the torus $M$ embedded in $\RR^3$ as in Figure \ref{fig:funnytorus}. The critical points $c_1, c_2, c_3, c_4, c_5$ of this smooth function $f:M\to \RR$ are all isolated, with $f(c_1) < f(c_2) < f(c_3) < f(c_4) < f(c_5)$. There are four nondegenerate critical points $c_1$ (where $f$ takes its minimum value), $c_3$ (a local maximum), $c_4$ (a saddle point) and $c_5$ (the global maximum), as well as one degenerate isolated critical point $c_2$ where the function locally looks like $f(x,y) = xy(x-y)$. 

\begin{figure}[t]
\includegraphics[width = 0.2\linewidth]{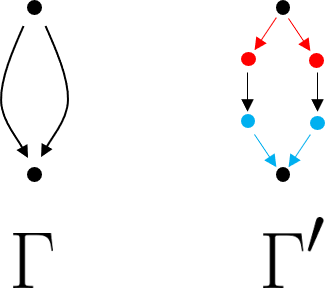}
\centering
\caption{Quiver and refined quiver for the height function on a torus in $\RR^3$ with critical locus two circles}
\label{fig:secondtorus}
\end{figure}
 
 A system of strict Morse neighbourhoods for $f$ near $c_2$ is given in Figure \ref{fig:strict}; each $\mathcal{N}_{c_2,n}$ is homeomorphic to a hexagon, with $\partial_+ \mathcal{N}_{c_2,n}$ and $\partial_- \mathcal{N}_{c_2,n}$ each having three connected components homeomorphic to the interval $[0,1]$, so that 
 $$H_1(\mathcal{N}_{c_2,n},\partial_\pm \mathcal{N}_{c_2,n};\FF) \cong \FF^2$$
 and $ H_i(\mathcal{N}_{c_2,n}, \partial_\pm \mathcal{N}_{c_2,n};\FF) = 0$ if $i \neq 1$. 
 
 Similarly we can choose Morse neighbourhoods $\mathcal{N}_{c_4,n}$ for $f$ near the saddle point $c_4$ such that each $\mathcal{N}_{c_4,n}$ is homeomorphic to a square, with $\partial_+ \mathcal{N}_{c_4,n}$ and $\partial_- \mathcal{N}_{c_4,n}$ each having two connected components homeomorphic to the interval $[0,1]$, so that 
 $$H_1(\mathcal{N}_{c_4,n},\partial_\pm \mathcal{N}_{c_4,n};\FF) \cong \FF$$
 and $ H_i(\mathcal{N}_{c_4,n}, \partial_\pm \mathcal{N}_{c_4,n};\FF) = 0$ if $i \neq 1$. The local maxima $c_3$ and $c_5$ (respectively minimum $c_1$) have Morse neighbourhoods $\mathcal{N}_{c_j,n}$ each homeomorphic to a disc, with $\partial_+ \mathcal{N}_{c_j,n}$ (respectively $\partial_- \mathcal{N}_{c_j,n}$) empty and $\partial_- \mathcal{N}_{c_j,n}$ (respectively $\partial_+ \mathcal{N}_{c_j,n}$) homeomorphic to a circle $S^1$, so that 
 $$ H_i( \mathcal{N}_{c_j,n}, \partial_- \mathcal{N}_{c_j,n};\FF) \cong \FF$$
 when $i=2$ (respectively $i=0$) and $ H_i( \mathcal{N}_{c_j,n}, \partial_- \mathcal{N}_{c_j,n};\FF)  = 0$ otherwise. Thus
 $$ \sum_{j=1}^5 P_t(  \mathcal{N}_{c_j,\infty}, \partial_- \mathcal{N}_{c_j,\infty}) = 1 + 2t + t^2 + t + t^2 = P_t(M) + (1+t)t$$
 while 
 $$ \sum_{j=1}^5 P_t(  \mathcal{N}_{c_j,\infty}, \partial_+ \mathcal{N}_{c_j,\infty}) = t^2 + 2t + 1 + t + 1 = P_t(M) + (1+t),$$
 verifying the descending and ascending Morse inequalities (Theorem \ref{mainthm}) in this example. Let us consider the descending case in more detail.
 
 The level sets, quiver $\Gamma_1$ and refined quiver $\Gamma_1'$ associated to $f:M\to \RR$ and a generic Riemannian metric on $M$ are illustrated in Figure \ref{fig:funnytorusquiver}. Recall that the quiver has vertices $\{v_{c_j}:j = 1,2 ,3,4,5\}$ labelled by the critical points $c_1, c_2, c_3, c_4, c_5$ and arrows from $v_{c_j}$ to $v_{c_k}$ for each connected component of 
$$\{ x \in M : \{ \psi_{t}(x): t \leqslant 0\} \mbox{  has a limit point  $c_j$ and } \{ \psi_{t}(x): t \geqslant 0\} \mbox{  has a limit point $c_k$} \}$$
whenever this subset is closed in $f^{-1}(f(c_k)),f(c_j))$. 

\begin{figure}[t]
\includegraphics[width = 0.5\linewidth]{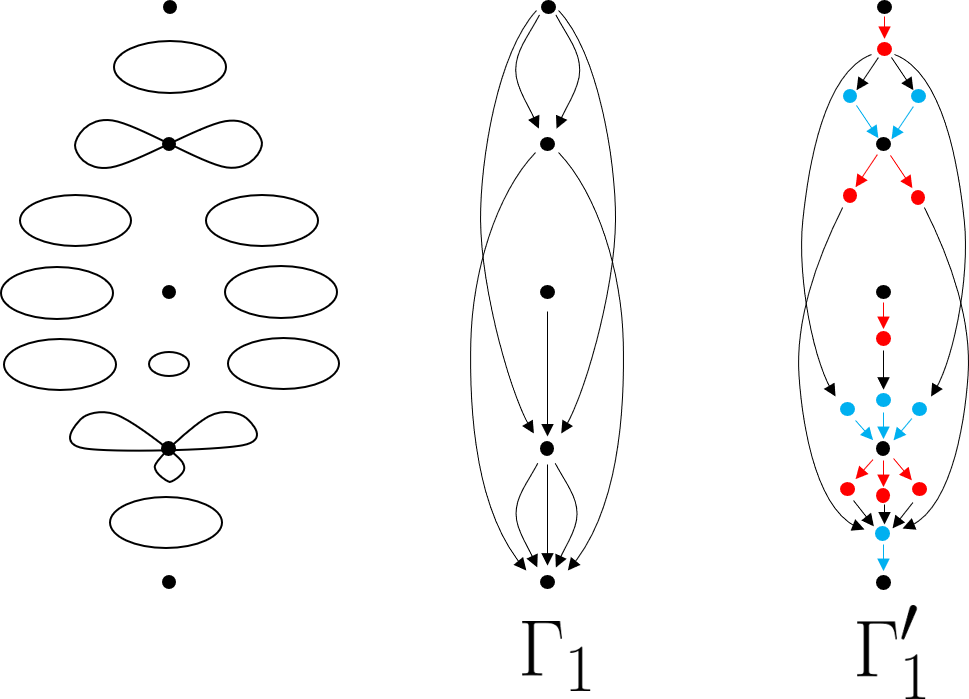}
\centering
\caption{Level sets, quiver and refined quiver for the height function on the torus in Figure \ref{fig:funnytorus} with generic Riemannian metric}
\label{fig:funnytorusquiver}
\end{figure}

As for our first example, the height function on a standard torus in $\RR^3$, there are two other possible quivers $\Gamma_2$ and $\Gamma_3$ for this function $f:M \to \RR$ arising from different choices of the Riemannian metric $g$; these are illustrated in Figure \ref{funnytorusquivers}.

The quiver $\Gamma_1^2$ of paths of length 2 in $\Gamma_1$ has the same vertices as $\Gamma_1$ and has ten arrows from $v_{c_5}$ to $v_{c_1}$ and three arrows from $v_{c_4}$ to $v_{c_1}$; if $r>2$ then $\Gamma_1^r$ has no arrows. 
The quiver $\Gamma_2^2$ of paths of length 2 in $\Gamma_2$ has the same vertices as $\Gamma_2$ (or equivalently the same vertices as $\Gamma_1$) and has five 
arrows from $v_{c_5}$ to $v_{c_1}$, two arrows from $v_{c_5}$ to $v_{c_2}$, three arrows from $v_{c_4}$ to $v_{c_1}$ and three arrows from $v_{c_3}$ to $v_{c_1}$. The quiver $\Gamma_2^3$ of paths of length 3 in $\Gamma_2$ has six arrows from $v_{c_5}$ to $v_{c_1}$ and no other arrows; if $r>3$ then $\Gamma_2^r$ has no arrows. 
The quiver $\Gamma_3^2$ of paths of length 2 in $\Gamma_3$ has four arrows from $v_{c_5}$ to $v_{c_2}$, six arrows from $v_{c_4}$ to $v_{c_1}$ and three arrows from $v_{c_3}$ to $v_{c_1}$, while the quiver $\Gamma_3^3$ of paths of length 3 in $\Gamma_3$ has twelve arrows from $v_{c_5}$ to $v_{c_1}$ and no other arrows; if $r>3$ then $\Gamma_3^r$ has no arrows. 

We saw at Lemma \ref{lemdec} that $ M = \bigsqcup_{C \in \mathcal{D}} W^+_C = \bigsqcup_{C \in \mathcal{D}} W^-_C$ where $\mathcal{D} = \{c_1,c_2,c_3,c_4,c_5\}$ and 
$$ W^+_{c_j} = \{ x \in M : \mbox{ for every $n \geqslant 0$ the downwards gradient flow for $f$ from $x$ enters } $$
$$ \mbox{ and never leaves the Morse neighbourhood } \mathcal{N}_{c_j,n} \} $$
and
$$ W^-_{c_j} = \{ x \in M : \mbox{ for every $n \geqslant 0$ the upwards gradient flow for $f$ from $x$ enters } $$
$$ \mbox{
and never leaves the Morse neighbourhood } \mathcal{N}_{c_j,n} \} $$
as in Definition  
\ref{defnwc}. For a generic choice of Riemannian metric, if there is an arrow in $\Gamma_1$ from $v_{c_j}$ to $v_{c_k}$ then $W^-_{c_j} \cap W^+_{c_k}$ is homeomorphic to a disjoint union of open intervals, one for each such arrow. If $j=k$ then $W^-_{c_j} \cap W^+_{c_k} = \{c_j\}$.
 Otherwise $W^-_{c_j} \cap W^+_{c_k}$ is empty, except for the intersections $W^-_{c_3} \cap W^+_{c_1}$  and $W^-_{c_5} \cap W^+_{c_1}$ which are open subsets of $M$, the former homeomorphic to a disc and the latter having two connected components each homeomorphic to a disc. We obtain Morse covers $\{\mathcal{W}^+\mathcal{N}^\circ_{c_j,n_j}:j=1,\ldots,5\}$ and $\{\mathcal{W}^-\mathcal{N}^\circ_{c_j,n_j}:j=1,\ldots,5\}$ of $M$ in the sense of Definition \ref{Ms} from the Morse stratifications $\{W^+_{c_j}:j=1,\ldots,5\}$ and $\{W^-_{c_j}:j=1,\ldots,5\}$ by taking open neighbourhoods of these strata by using the interiors of Morse neighbourhoods of the critical points combined with the gradient flow. 
 As in Remark \ref{decomp}, we can choose the Morse neighbourhoods such that the submanifolds with corners $\mathcal{W}^+\mathcal{N}^\circ_{c_j,n_j}$ meet along submanifolds (with corners) of their boundaries, so that $M$ has a triangulation which is compatible with these submanifolds and their intersections. 
\begin{figure}[t]
\includegraphics[width = 0.26\linewidth]{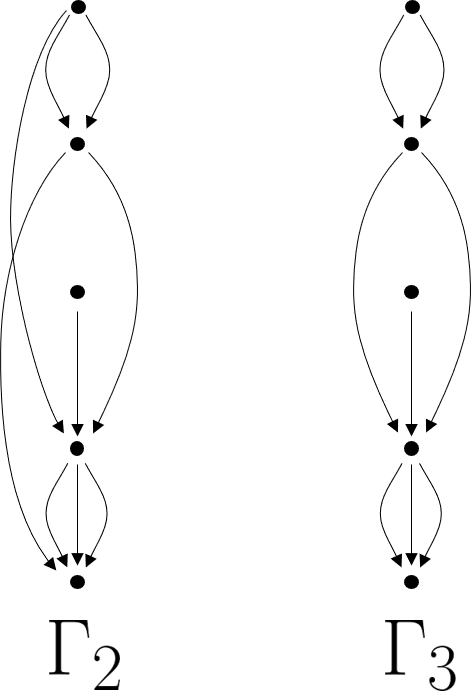}
\centering
\caption{Quivers $\Gamma_2$ and $\Gamma_3$ for the height function on the torus in Figure \ref{fig:funnytorus} with non-generic Riemannian metrics}
\label{fig:funnytorusquivers}
\end{figure}
 
 Recall from $\S$5 that, given any levelling function $\phi:\Gamma_1 \times \ZZ \to \RR$ in the sense of Definition \ref{levelling} (for example the Morse levelling function), we can associate to $f:M \to \RR$
  a spectral sequences abutting  to the homology $H_*(M;\FF)$ of $M$
whose $E_1$ page is given by a differential on
$$ \bigoplus_{v_{c_j} \in \Gamma_1, n \in \ZZ}  H_n(\mathcal{N}_{c_j,n_j},\partial_-\mathcal{N}_{c_j,nj};\FF) $$
which is the sum over all the arrows $a$ in $\Gamma_1$ of maps
$$ \partial^a_i: H_i( \mathcal{N}_{c_k,n_k},\partial_-\mathcal{N}_{c_k,n_k};\FF) \to 
H_{i-1}( \mathcal{N}_{c_\ell,n_\ell},\partial_-\mathcal{N}_{c_\ell,n_\ell};\FF)$$
where $t(a) = v_{c_k}$ and $h(a) = v_{c_\ell}$, defined as follows. If $[\xi] \in H_i( \mathcal{N}_{c_k,n_k},\partial_-\mathcal{N}_{c_k,n_k};\FF)$ is represented by a chain $\xi$ in $\mathcal{N}_{c_k,n_k}$ (for suitable $n_k >\!> 1$) with boundary in $\partial_-\mathcal{N}_{c_k,n_k}$, then the pairing of 
$$ \partial^a_i[\xi] \in H_{i-1}( \mathcal{N}_{c_\ell,n_\ell},\partial_-\mathcal{N}_{c_\ell,n_\ell};\FF) \cong H_{\dim(M) - i+1}( \mathcal{N}_{c_\ell,n_\ell},\partial_+\mathcal{N}_{c_\ell,n_\ell};\FF)^*$$
 with a $(\dim M - i + 1)$-chain $\eta$ in $\mathcal{N}_{c_\ell,n_\ell}$ with boundary in $\partial_+\mathcal{N}_{c_\ell,n_\ell}$ is given by flowing $\partial \eta$ upwards and $\partial \xi$ downwards in a small neighbourhood of the connected component of $W^-_{c_k} \cap W^+_{c_\ell}$ which defines the arrow a and intersecting them in any level set of $f$ strictly between $f(c_\ell)$ and $f(c_k)$.
 
 Such a map $ \partial^a_i$ can only be nonzero when $H_i( \mathcal{N}_{c_k,n_k},\partial_-\mathcal{N}_{c_k,n_k};\FF)$ and $ 
H_{i-1}( \mathcal{N}_{c_\ell,n_\ell},\partial_-\mathcal{N}_{c_\ell,n_\ell};\FF)$ are both nonzero. We have seen that $H_i( \mathcal{N}_{c_k,n_k},\partial_-\mathcal{N}_{c_k,n_k};\FF)$  is only nonzero if $i=0$ and $j=1$, or if $i=1$ and $j \in \{2,4\}$, or if $i=2$ and $j\in \{ 3,5\}$. Looking at the quiver $\Gamma_1$, we see that the sum of the maps $\partial_i^a$ over all the arrows $a$ from a given vertex can only be nonzero when it is 
$$\theta_{3,2}: \FF \cong H_2( \mathcal{N}_{c_3,n_3},\partial_-\mathcal{N}_{c_3,n_3};\FF) \to 
H_{1}( \mathcal{N}_{c_2,n_2},\partial_-\mathcal{N}_{c_2,n_2};\FF)  \cong \FF^2$$
or
$\theta_{5,24} = \theta_{5,2} + \theta_{5,4}$ from $ \FF \cong H_2( \mathcal{N}_{c_5,n_5},\partial_-\mathcal{N}_{c_5,n_5};\FF) $ to $$H_{1}( \mathcal{N}_{c_2,n_2},\partial_-\mathcal{N}_{c_2,n_2};\FF) \oplus 
H_{1}( \mathcal{N}_{c_4,n_4},\partial_-\mathcal{N}_{c_4,n_4};\FF) \cong \FF \oplus \FF$$
or 
$ \theta_{2,1}: \FF^2 \cong H_{1}( \mathcal{N}_{c_2,n_2},\partial_-\mathcal{N}_{c_2,n_2};\FF) \to
H_{0}( \mathcal{N}_{c_1,n_1},\partial_-\mathcal{N}_{c_1,n_1};\FF) \cong \FF.$
In this situation, just as in the Morse--Smale case, there is  a levelling function $\phi_1$ such that $\phi_1(v,q) = q$, and then the differential on the $E_1$ page of the associated spectral sequence is $\theta_{3,2} + \theta_{5,24} + \theta_{2,1}$ where $\theta_{3,2}$ is injective, while $\theta_{2,1} = 0$. To understand $\theta_{5,24}$, consider the generator of $H_2( \mathcal{N}_{c_5,n_5},\partial_-\mathcal{N}_{c_5,n_5};\FF)$  whose boundary is represented by the circle $\partial_-\mathcal{N}_{c_5,n_5}$. This decomposes according to its intersections with $\mathcal{W}^+\mathcal{N}^\circ_{c_j,n_j}$ for $j=1,2,4$ 
 into eight intervals meeting at their endpoints in cyclical order $12141214$. The intersections with $\mathcal{W}^+\mathcal{N}^\circ_{c_2,n_2}$ and  $\mathcal{W}^+\mathcal{N}^\circ_{c_4,n_4}$ flow downwards without meeting each other until they meet $\partial^+\mathcal{N}^\circ_{c_2,n_2}$ and 
 $\partial^+ \mathcal{N}^\circ_{c_4,n_4}$, and by considering the local pictures of the gradient flow of $f$ near $c_2$ (see Figure \ref{fig:strict}) and near the saddle point $c_4$, we find that $\theta_{5,42}$ maps this generator of  
$H_2( \mathcal{N}_{c_5,n_5},\partial_-\mathcal{N}_{c_5,n_5};\FF)$ 
to $(0, - \theta_{3,2}([\zeta])$ where $[\zeta]$ is a similar generator of $H_2( \mathcal{N}_{c_3,n_3},\partial_-\mathcal{N}_{c_3,n_3};\FF)$.
At the second page of the spectral sequence we have a map $\theta_{5,1}$ associated to the paths of length two in $\Gamma_1$ which flow from $c_5$ to $c_1$, but these are all zero on the grounds of degree, so the spectral sequence degenerates at this page. 

If on the other hand we choose the Morse levelling function $\phi_2(v_{c_j},q) = f(c_j)$ (as we can always do), then the differential on the first page of the spectral sequence is $\theta_{3,2} + \theta_{5,4} + \theta_{2,1}$ where $\theta_{3,2}$ as before is injective and $\theta_{5,4}$ and $ \theta_{2,1}$ are both zero. Again the spectral sequence degenerates after the first page.

Notice that we can modify the smooth function $f:M \to \RR$ without changing the quiver $\Gamma_1$ so that the critical value $f(c_3)$ is no longer less than $f(c_4)$; indeed it can be made greater than $f(c_5)$. This provides different choices of levelling functions for which the differential on the first page of the spectral sequence is zero, so that in these cases the spectral sequence does not degenerate after the first page.

\medskip

To construct an example without isolated critical points, we can consider the following modification of the height function $f:M\to \RR$ illustrated in Figure \ref{fig:funnytorus}. First replace each of the local maxima $c_3$ and $c_5$ with a local maximum attained on a small circle and a local minimum inside the circle. Next remove disc neighbourhoods of the two new local minima and (smoothly) glue in a cylinder on which the height function has one local minimum and one saddle point, to create a surface $\hat{M}$ of genus 2 with a smooth function $\hat{f}:\hat{M} \to \RR$ which is Morse--Bott away from the critical point $c_2$.  The connected components of $\mathrm{Crit}(\hat{f})$ are:

(i) the critical points $c_1,c_2,c_4$ of $f:M\to \RR$;

(ii) two circles $C_3$ and $C_5$ which are Morse--Bott local maxima replacing the critical points $c_3$ and $c_5$ of $f:M\to \RR$;

(iii) an additional saddle point $c_6$ and local minimum $c_7$ in the cylinder joining $C_3$ and $C_5$.

Morse neighbourhoods of $c_1,c_2,c_4$ are as before, and $c_6$ and $c_7$ have Morse neighbourhoods just like $c_4$ and $c_1$. Morse neighbourhoods of the circles $C_3$ and $C_5$ are given by annuli $\mathcal{N}_{C_j,n}$ for $j=3,5$  with $\partial_+ \mathcal{N}_{C_j,n}$ empty and $\partial_- \mathcal{N}_{C_j,n}$ the disjoint union of two circles, so $H_i(\mathcal{N}_{C_j,n}, \partial_- \mathcal{N}_{C_j,n}; \FF)$ is isomorphic to $\FF$ if $i=1,2$ and is zero otherwise. Thus
$$P_t(\mathcal{N}_{c_1,\infty},\partial_- \mathcal{N}_{c_1,\infty})  + P_t(\mathcal{N}_{c_2,\infty},\partial_- \mathcal{N}_{c_2,\infty})  + P_t(\mathcal{N}_{C_3,\infty},\partial_- \mathcal{N}_{C_3,\infty}) + P_t(\mathcal{N}_{c_4,\infty},\partial_- \mathcal{N}_{c_4,\infty}) $$   $$
+ P_t(\mathcal{N}_{C_5,\infty},\partial_- \mathcal{N}_{C_5,\infty}) + P_t(\mathcal{N}_{c_6,\infty},\partial_- \mathcal{N}_{c_6,\infty}) + P_t(\mathcal{N}_{c_7,\infty},\partial_- \mathcal{N}_{c_7,\infty}) 
$$  $$= 1 + 2t + (t + t^2) + t + (t + t^2) + t + 1 = P_t(\hat{M}) + (1+t)^2$$
verifying Theorem \ref{mainthm} in this case. The quiver $\Gamma$ 
associated to $\hat{f}: \hat{M} \to \RR$ with a generic Riemannian metric on $\hat{M}$ is illustrated in Figure \ref{fig:genus2quiver}; the analysis of the spectral sequences is similar to the previous example
of the height function $f:M\to \RR$ on a torus $M$ embedded in $\RR^3$ as illustrated in Figure \ref{fig:funnytorus}.

\begin{figure}[t]
\includegraphics[width = 0.1\linewidth]{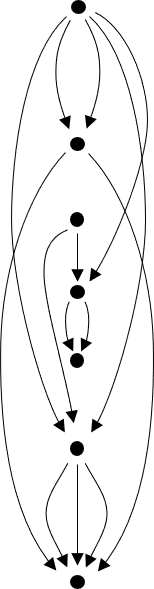}
\centering
\caption{Quiver for $\hat{f}:\hat{M} \to \RR$  with generic Riemannian metric}
\label{fig:genus2quiver}
\end{figure}

 \medskip

Recall that a smooth function $f \colon M \to \RR$ is a Morse--Bott function if the connected components  $C \in \mathcal{D}$ of its critical set $\text{Crit}(f)$ are submanifolds of $M$ and the Hessian of $f$  at any point of $C$ is nondegenerate in directions normal to $C$. In this situation $M = \bigcup_{C \in \mathcal{D}} S_C$ where the Morse strata $S_C$ are given by
\begin{equation} \label{eqnstrata} S_C = \{ x \in M: \mbox{ the downwards gradient flow from $x$ has a limit point in $C$} \} \end{equation}
and 
Morse neighbourhoods $\mathcal{N}_{C,n}$ can be chosen as disc sub-bundles in the normal bundles to the Morse strata $S_C$ over neighbourhoods of $C$ in $S_C$ (with $\partial_- \mathcal{N}_{C,n}$ corresponding to the boundaries of the discs). Thus 
$$H_i(\mathcal{N}_{C,n}, \partial_- \mathcal{N}_{C,n};\RR) \cong H_{i - \text{codim}_\RR (S_C)}(C;\RR)
$$ giving us the classical Morse inequalities (\ref{Morseineq}) as an immediate consequence of Theorem \ref{mainthm}. Morse--Bott homology has also been studiedin ways which are related to but not the same as our viewpoint  \cite{BH,Frau, Hurtubise, Hurtubise2,ZZ} .

An early example of Morse theory applied to a smooth function which is not Morse--Bott was the normsquare of a moment map $\mu:M\ \to \lieks$ for a Hamiltonian action of a compact Lie group $K$ on a compact symplectic manifold $M$ \cite{K}. For any choice of invariant inner product on the Lie algebra $\liek$ of $K$ the normsquare $f = |\!|\mu|\!|^2$  of the moment map is a $K$-equivariantly perfect  minimally degenerate Morse function. 
 More precisely,  the set of critical points for $f=||\mu||^2$ is a finite disjoint union of closed subsets $C \in \mathcal{D}$ along each of which $f$ is {minimally degenerate} in the following sense.

\begin{defn}  \label{mindeg} A locally closed submanifold $\Sigma_C$ containing $C$ with orientable normal bundle in $M$
 is a \emph{minimising (respectively maximising) submanifold} for $f$  along $C$ if
  
  \begin{enumerate} \item  the restriction of $f$ to $\Sigma_C$ achieves its minimum (respectively maximum) value exactly on $C$, and
  
   \item the tangent space to $\Sigma_C$ at any point $x \in C$ is maximal among subspaces of $T_x M$ on which the Hessian $H_x(f)$ is non-negative (respectively non-positive). 
   
    \end{enumerate} 
    If a minimising (respectively maximising) submanifold $\Sigma_C$ exists, then $f$ is called minimally (respectively maximally) degenerate along $C$. The smooth function $f$ is minimally (respectively maximally) degenerate if it is minimally (respectively maximally) degenerate along every $C \in \mathcal{D}$. It is extremally degenerate if for each $C \in \mathcal{D}$ it is either minimally or maximally degenerate along $C$.
\end{defn}
    
    In \cite{K} it was shown that if $f$ is minimally degenerate then it induces a smooth stratification $\{S_C: C \in \mathcal{D}\}$ of  $M$ where $S_C$ is the Morse stratum defined as at (\ref{eqnstrata}) for a suitable choice of Riemannian metric, and that there are thus Morse inequalities as in the Morse--Bott case. The stratum $S_C$ then coincides with $\Sigma_C$ near $C$.

\begin{rem} \label{mindisc}
Indeed suppose that $f \colon M\to \RR$ is any extremally degenerate smooth function on a compact manifold $M$. Then just as in the Morse--Bott case, Morse neighbourhoods $\mathcal{N}_{C,n}$  can be chosen to be disc bundles in the normal bundle to an arbitrarily small neighbourhood of $C$ (which retracts onto $C$) in the minimising or maximising manifold for $f$ along $C$. Here $\partial_\pm \mathcal{N}_{C,n}$ corresponds to the boundary of the disc or the boundary of the neighbourhood (depending on whether there is a minimising or maximising manifold for $f$ along $C$). Thus  $P_t(N_{C,\infty},\partial_- N_{C,\infty}) = t^{\lambda_f(C)} P_t(C)$ where $\lambda_f(C) 
$ is the codimension (respectively dimension) of the minimising (respectively maximising) manifold along $C$. From
Theorem \ref{mainthm}, if the normal bundles are orientable, we obtain Morse inequalities of the classical form (\ref{Morseineq}).
\end{rem}
    
    When $f= |\!|\mu|\!|^2$ 
 is    the normsquare of a moment map for a Hamiltonian action of a compact Lie group $K$ on a compact symplectic manifold then $f$ is minimally degenerate and
    there are also $K$-equivariant Morse inequalities. It was shown in \cite{K} that these are in fact equalities: the normsquare of the moment map is equivariantly perfect. This was then used to obtain inductive formulas for the Betti numbers of symplectic quotients, and of quotient varieties arising in algebraic geometry.
In their fundamental paper \cite{AB} Atiyah and Bott considered an infinite-dimensional version of the normsquare of a moment map given by the Yang--Mills functional arising in gauge theory over a compact Riemann surface. They conjectured that the Yang--Mills functional should induce an equivariantly perfect Morse stratification on an infinite-dimensional space of connections, but avoided the analytical difficulties created by the degeneracy of the functional, combined with the infinite-dimensionality, by using an alternative construction of the stratification. Their conjecture was proved by Daskalopoulos in \cite{Dask} and extended by Witten in \cite{Witten92} to obtain intersection pairings.

\section{Further extensions} \label{sec:applications}

We have associated to a compact Riemannian manifold $M$ with smooth  $f:M \to \RR$,  whose critical locus $\text{Crit}(f)$ has finitely many connected components, a quiver $\Gamma$ (see Definition \ref{6.1below}) and a refined quiver $\Gamma'$ (see Definition \ref{defnquiverrefined}). 
In addition given a levelling function (see Definition \ref{levelling}) such as the one provided by the Morse function itself, we obtain a spectral sequence whose $E_1$ page is a level 1 graded multicomplex supported on $\Gamma$ (in the sense of Definition \ref{defngamma}) involving the relative homology $H_*(\mathcal{N}_{C,n}, \partial_-\mathcal{N}_{C,n};\FF)$ of Morse neighbourhoods, and which abuts to the homology $H_*(M;\FF)$ of $M$. Using this (and in other ways) we have obtained Morse inequalities relating the dimensions of these homology groups. 

More generally we can define a Morse cobordism $\mathcal{M}$ as follows (cf.\! Definition \ref{Mn}), and extend our arguments to obtain similar descriptions for the relative homology $H_*(\mathcal{M},\partial_-\mathcal{M};\FF)$.

\begin{defn} \label{defnmc}
We will say that a {\it Morse cobordism} in dimension $n$ is given by a compact manifold $\mathcal{M}$  with corners (locally modelled on $[0,\infty) ^2 \times \RR^{\dim M -2}$) and a Riemannian metric $g$, together with a smooth function $f:M \to \RR$,
such that 

(a) $\mathcal{M}$ has boundary
$\partial \mathcal{M} = \partial_+ \mathcal{M} \cup \partial_- \mathcal{M} \cup \partial_\perp \mathcal{M}$
where $\partial_+ \mathcal{M}$ , $\partial_- \mathcal{M}$ and $\partial_\perp \mathcal{M}$ are compact submanifolds of $\mathcal{M}$ with boundaries whose connected components form the corners of $\mathcal{M}$, and $ \partial_+ \mathcal{M} \cap \partial_- \mathcal{M} \cap \partial_\perp \mathcal{M} = \emptyset$;

(b) $f$ is locally constant on
$   (\partial_+ \mathcal{M}) \cap  (\partial_- \mathcal{M}) \, = \, \partial(\partial_+\mathcal{M}) \, \cap \, \partial(\partial_- \mathcal{M}), $  and on each of
$ \partial(\partial_\pm \mathcal{M}) \, \cap \, \partial(\partial_\perp \mathcal{M}) \, = \, (\partial_\pm \mathcal{M}) \cap  (\partial_\perp \mathcal{M})  
;$

(c)  the critical locus $\text{Crit}(f)$ has finitely many connected components, and every critical point $p$ of $f$ on $\mathcal{M}$ lies in the interior of $\mathcal{M}$ and satisfies
$$f(x) \leqslant f(p) \leqslant f(y)$$
for all $x \in \partial_- \mathcal{M}$ and $y \in \partial_+ \mathcal{M}$, with strict inequalities when $x \in (\partial_- \mathcal{M})^\circ$ and $  y \in (\partial_+ \mathcal{M})^\circ$;

(d) the gradient vector field $\grad(f)$ on $\mathcal{M}$ associated to its Riemannian metric $g$  satisfies\\
(i) 
 the restriction of $\grad (f)$ to $(\partial_+ \mathcal{M})^\circ $ points inside $\mathcal{M}$;  \\
(ii) 
 the restriction of $\grad (f)$ to $(\partial_- \mathcal{M})^\circ$ points outside $\mathcal{M}$;
\\
(iii) $\partial_\perp \mathcal{M}$ is invariant 
under the gradient flow $\grad (f)$ (where defined);

(e) each connected component $Q$ of $\partial_\perp \mathcal{M}$ has a diffeomorphism 
$$ \Psi_Q: (Q \cap \partial_+\mathcal{M}) \times [b_-^Q,b_+^Q] \cong 
Q$$
whose composition with $f$ is projection onto the interval $ [b_-^Q,b_+^Q]$, such that $\Psi_Q(x,d)$ is given by the intersection of the gradient flow from $x$ with $f^{-1}(d)$, and in particular $Q \cap \partial_+\mathcal{M}$ and $Q \cap \partial_-\mathcal{M} = \Psi_Q((Q \cap \partial_+\mathcal{M})\times \{b_-^Q\}) \cong Q \cap \partial_+\mathcal{M}$ are corners of $\mathcal{M}$. 
\end{defn}

\begin{defn} \label{defnquiver3}
The quiver  associated to a Morse cobordism $\mathcal{M}$  as above  has  vertices of three types: a vertex $v_C$ for each connected component $C$ of the critical locus $\text{Crit}(f)$ of $f$ (on the interior of $\mathcal{M}$), a vertex $u_{G}$ for every connected component $G$ of $\partial_+\mathcal{M}$ and a vertex $w_{H}$ for every connected component $H$ of $\partial_-\mathcal{M}$. There is an arrow from one vertex (labelled by $K=C,G$ or $H$) to another (labelled by $K'=C',G'$ or $H'$) for each connected component of 
$$\{ x \in M :  \mbox{the trajectory of $x$ under $\grad(f)$ flows up to $K$ and down to $K'$} \}$$
when this subset has closed intersection with $f^{-1}(\sup\! f(K'),\inf\! f(K))$.
\end{defn}

A Morse neighbourhood $\mathcal{N}_{C,n}$ for a smooth function $f:M\to \RR$ whose critical locus $\text{Crit}(f)$ has finitely many connected components on a compact Riemannian manifold $M$, equipped with the restriction of $f$ and of the Riemannian metric, is a Morse cobordism in the sense of Definition \ref{defnmc}. Moreover any Morse cobordism $\mathcal{M}$ can be decomposed into a union of Morse neighbourhoods and Morse cobordisms which are trivial in the sense that there are no (interior) critical points, meeting along parts of their boundaries (in a way reminiscient of TQFTs). Given a levelling function generalising Definition \ref{levelling} (for example using the critical values of $f$ as in Example \ref{examplelevelling}) we obtain a spectral sequence whose $E_1$ page is a level 1 graded multicomplex supported on the quiver given as in Definition \ref{defnquiver3} which abuts to  the relative homology $H_*(\mathcal{M},\partial_-\mathcal{M};\FF)$.

\begin{rem}
Our arguments also extend to work for equivariant homology, and to obtain generalisations of the Novikov inequalities for a closed 1-form on $M$ (cf. \cite{BF,BF2,BF3,BF4,BS, Novikov, Novikov2, Pazhitnov}). 
\end{rem}

\end{document}